\documentclass[10pt]{article}
\usepackage{pifont}
\usepackage{amsmath}
\usepackage{amscd}
\usepackage{amsfonts}
\usepackage{latexsym}
\usepackage{amssymb}
\usepackage{stmaryrd}
\usepackage{overpic}
\usepackage{graphicx}
\usepackage{undertilde}
\usepackage{subcaption}
\usepackage{empheq}
\usepackage{multirow}
\usepackage{enumitem}
\usepackage{lineno}
\usepackage{rotating}
\usepackage[unicode,colorlinks,linkcolor=blue,hyperindex]{hyperref}

\usepackage{tikz}
\usetikzlibrary{er}

\newtheorem{theorem}{Theorem}[section]
\newtheorem{assumption}[theorem]{Assumption}
\newtheorem{corollary}[theorem]{Corollary}
\newtheorem{definition}[theorem]{Definition}
\newtheorem{example}[theorem]{Example}
\newtheorem{lemma}[theorem]{Lemma}

\newtheorem{remark}[theorem]{Remark}
\newenvironment{proof}[1][Proof]{\textbf{#1.} }
{\ \rule{0.75em}{0.75em}\smallskip}

\numberwithin{equation}{section}
\numberwithin{table}{section}
\numberwithin{figure}{section}

\newcommand{\vect}[1]{\boldsymbol{#1}} 
\newcommand{\DGnabla}{\nabla_{\rm dg}}
\newcommand{\DGdiv}{{\rm div}_{\rm dg}}

\newcommand\comments[1]{\textcolor{black}{#1}}

\begin{document}

\title{A Unified Study of \\Conforming and Discontinuous Galerkin
Methods\thanks{ 
The work of the last author was supported by US Department of Energy
Grant DE-SC0014400 and NSF grant DMS-1522615.  The work of the second
author was partially supported by the National Natural Science
Foundation of China (Grant No.  11771350).
}}
\author{Qingguo Hong\footnote{huq11@psu.edu@psu.edu, Department of
Mathematics, Pennsylvania State University, University Park, PA,
16802, USA} \quad 
Fei Wang\footnote{feiwang.xjtu@xjtu.edu.cn, School of Mathematics and
Statistics, Xi'an Jiaotong University, Xi'an, Shaanxi, 710049,
China} \quad  
Shuonan Wu\footnote{sxw58@psu.edu, Department of Mathematics,
Pennsylvania State University, University Park, PA, 16802, USA}
\quad
Jinchao Xu\footnote{xu@math.psu.edu, Department of Mathematics,
Pennsylvania State University, University Park, PA, 16802, USA}}
\date{}
\maketitle 

\begin{abstract}
A unified study is presented in this paper  for the design and
analysis of different finite element methods (FEMs),  including
conforming and nonconforming FEMs, mixed FEMs, hybrid FEMs,
discontinuous Galerkin (DG) methods, hybrid discontinuous Galerkin
(HDG) methods and weak Galerkin (WG) methods.  Both HDG and WG are
shown to admit inf-sup conditions that hold uniformly with respect to
both mesh and penalization parameters.  In addition, by taking the
limit of the stabilization parameters, a WG method is shown to
converge to a mixed method whereas an HDG method is shown to converge
to a primal method.  Furthermore, a special class of DG methods, known
as the {\it mixed DG methods}, is presented to fill a gap revealed in
the unified framework.
\end{abstract}


\section{Introduction} \label{sec:intro} 
In this paper, we propose a general framework to derive most of the
existing finite element methods (FEMs), and discuss their
relationships.  We will illustrate the main idea by using the
following elliptic boundary value problem
\begin{equation} \label{pois}
\left\{
\begin{aligned}
-{\rm div} (\alpha\nabla u)&= f \qquad {\rm in} \ \Omega, \\
 u&=0 \qquad {\rm on} \ \partial\Omega,
\end{aligned}
\right.
\end{equation}
where $\Omega\subset \mathbb R^d$ ($d\ge 1$) is a bounded domain and
$\alpha: \mathbb{R}^d \rightarrow \mathbb{R}^d$ is a bounded and
symmetric positive definite matrix, with its inverse denoted by $c =\alpha^{-1}$. 

Setting $\vect{p} =-\alpha\nabla u$, the above problem can be formally
written in two different forms.  The first, known as the primal
formulation, can be written as
\begin{equation} \label{H1}
\left\{
\begin{aligned}
c \vect{p} + \nabla u &= 0 \qquad {\rm in}\ \Omega, \\ 
-\nabla^*  \vect{p} &= f \qquad {\rm in}\ \Omega, 
\end{aligned}
\right.
\end{equation}
which requires that $u\in H^1_0(\Omega)$ and $\vect{p}\in \boldsymbol
L^2(\Omega)$.   

The second, known as the mixed formulation, can be written as
\begin{equation}
  \label{HDIV}
\left\{
\begin{aligned}
c \vect{p} -  {\rm div}^*u &= 0 \qquad {\rm in}\ \Omega, \\ 
{\rm div}  \vect{p} &= f \qquad {\rm in}\ \Omega, 
\end{aligned}
\right.
\end{equation}
which requires that $u\in L^2(\Omega)$ and $\vect{p}\in {\boldsymbol
H}({\rm div}; \Omega)$.  

~

The design of FEMs then becomes an appropriate discretization of
$\nabla$ and ${\rm div}$, which amounts to imposing certain continuity
conditions on $u$ and ${\boldsymbol p}\cdot{\boldsymbol n}$ in the
following five different approaches:
\begin{enumerate}
\item strongly ($H^1$ or $\boldsymbol H({\rm div})$ conforming
    elements), 
\item  weakly (nonconforming), 
\item  via Lagrangian multiplier (hybrid methods), 
\item  via Lagrangian multiplier and stabilization --- \comments{stabilized hybrid primal methods and stabilized hybrid mixed methods}, and
\item  via penalization (DG).
\end{enumerate}

The resulting different types of FEMs can then be
fully described in a uniform framework by the notion of DG-gradient
--- $\DGnabla$, DG-divergence --- $\DGdiv$.  Here, $\DGnabla$ (see
\eqref{DG_gradient}) is a generalization of the piecewise gradient which allows $u_h$ to be discontinuous across element boundaries, but uses
an additional Lagrangian multiplier space on the element boundaries.
We note that $\DGnabla$ and $\DGdiv$ in this context, as shown in
Table~\ref{table:weak_derivatives_relation} below, directly correspond
to the weak derivatives introduced by Wang and Ye \cite{wang2013weak}.
We denote by $\{\mathcal{T}_h\}_h$ a family of triangulations of
$\overline{\Omega}$ which satisfy the minimal angle condition. Let
$h_K ={\rm diam}(K)$ and $h = \max\{h_K: K\in \mathcal{T}_h\}$.  For
any $K\in \mathcal{T}_h$, denote $\vect{n}_K$ as the outward unit
normal of $K$.  Denote by ${\cal E}_h$ the union of the boundaries of
the elements $K$ of $\mathcal{T}_h$. \comments{Let $P_h:L^2(\Omega)\rightarrow
V_h$ be the $L^2$ projection and $f_h=P_h f$}.

\begin{definition}[DG-derivatives] \label{DG_derivative_global} 
Let $\tilde{u}_h = (u_h,\hat{u}_h) \in \widetilde{V}_h$, and
$\tilde{\vect{p}}_h = (\vect{p}_h, \hat{\vect{p}}_h) \in
\widetilde{\vect{Q}}_h$. Then the DG-gradient $\DGnabla:
V_h \mapsto \widetilde{\vect{Q}}_h^*$ and the DG-divergence $\DGdiv:
\vect{Q}_h \mapsto \widetilde{V}_h^*$ are defined by 
\begin{align}
\langle \DGnabla u_h, \tilde{\vect{q}}_h \rangle &=
(\nabla_h u_h, \vect{q}_h)_\Omega-  \sum_{K\in \mathcal{T}_h}\langle
u_h, \hat{\vect{q}}_h\cdot \vect{n}_K\rangle _{\partial K}  \qquad
\forall u_h\in V_h,\; \forall \tilde{\vect{q}}_h \in
\widetilde{\vect{Q}}_h, \label{DG_gradient}\\
\langle \DGdiv \vect{p}_h, \tilde{v}_h \rangle &= ({\rm
div}_h\vect{p}_h, v_h)_\Omega - \sum_{K\in \mathcal{T}_h}  \langle
\vect{p}_h\cdot \vect{n}_K, \hat{v}_h\rangle _{\partial K} \qquad
\forall  \tilde{v}_h \in \widetilde V_h,\; \forall \vect{p}_h \in
\vect{Q}_h. \label{DG_divergence}
\end{align}
Here, $V_h$ and $\vect{Q}_h$ are the piecewise scalar and
vector-valued discrete spaces on the triangulation $\mathcal{T}_h$,
respectively, and $\widehat{V}_h$ and $\widehat{\vect{Q}}_h$ are the
piecewise scalar and vector-valued discrete spaces on ${\cal E}_h$,
respectively. $\widetilde{V}_h$ and $\widetilde{\vect{Q}}_h$
are defined as 
$$
\widetilde{V}_h := V_h \times \widehat{V}_h, \qquad  
\widetilde{\vect{Q}}_h = \vect{Q}_h \times \widehat{\vect{Q}}_h. 
$$
$\widetilde{\vect{Q}}_h^*$ and $\widetilde{\vect{V}}_h^*$ are the dual
spaces of  $\widetilde{\vect{Q}}_h$ and $\widetilde{\vect{V}}_h$,
respectively. 
\end{definition}

As for the DG-derivatives defined in Definition
\ref{DG_derivative_global}, we will specify $V_h$, $\widehat{V}_h$,
$\vect{Q}_h$ and $\widehat{\vect{Q}}_h$ at different concurrences.
Let $\hat{u}_h$ and $\hat{\vect{p}}_h$ be defined as
single-valued functions on ${\cal E}_h$. They can be viewed as
certain numerical traces of DG functions.  We want to emphasize that
the DG-derivatives on the discrete spaces are globally defined. The
dual operators of $\DGnabla$ and $\DGdiv$ are denoted by $\DGnabla^*:
\widetilde{\vect{Q}}_h \mapsto V^*_h $ and $\DGdiv^*: \widetilde{V}_h
\mapsto \vect{Q}^*_h$, respectively, such that
\begin{align}
\langle \DGnabla^* \tilde{\vect{p}}_h, v_h \rangle 
& = \langle \tilde{\vect{p}}_h, \DGnabla v_h \rangle 
~~\qquad \forall v_h \in V_h, \label{weak_gradient_dual}\\
\langle \DGdiv^*\tilde{u}_h, \vect{q}_h \rangle &= \langle
\tilde{u}_h, \DGdiv \vect{q}_h \rangle \qquad \forall \vect{q}_h \in
\vect{Q}_h.\label{weak_divergence_dual}
\end{align}
We shall now use these DG-derivatives to formulate different types of
Galerkin methods in \S\ref{subsec:primal}--\S\ref{subsec:DG} below and
summarize these different methods in Table \ref{tab:FEMs}.  We further
give some brief bibliographic comments on the development of these
methods in \S\ref{sc:history}.

\begin{sidewaystable}[!htbp]
\centering
\scriptsize
\begin{tikzpicture}[node distance=15em] 
  
\node[entity] (DG) [align=center] { {\bf
  \color{blue} DG} \\ 
      $c\vect{p}_h
  -\DGdiv^* \tilde{u}_h = 0$ \\
    $-\DGnabla^* \tilde{\vect{p}}_h = f_h$\\ 
    Define $\hat{\vect{p}}_h$ and  $\hat{u}_h$ explicitly 
};

\node[entity] (LDG) [left of=DG, xshift=2em]
[align=center] {{\bf \color{blue} Primal DG} \\
$\hat{\vect{p}}_h = \{\vect{p}_h\} + \eta_e h^{-1} \llbracket u_h
  \rrbracket$ \\ $\hat{u}_h = \{u_h\}$} 
edge [line width=0.5mm, dotted,red] (DG);

\node[entity] (SHY-PRIMAL) [left of=LDG] [align=center] { 
  {\bf \color{blue} Stabilized Hybrid Primal}\\  
  $ c \vect{p}_h + \mathcal{S}_{\vect{p}}^\eta
  \tilde{\vect{p}}_h +\DGnabla u_h = 0$ \\ 
  $-\DGnabla^* \tilde{\vect{p}}_h = f_h$}
edge [line width=0.5mm, dotted,red] (LDG); 

\node[entity] (MDG) [right of=DG, xshift=-2em] 
[align=center] {{\bf \color{blue} Mixed DG } \\ $\hat{\vect{p}}_h =
  \{\vect{p}_h\}$ \\ $\hat{u}_h = \{u_h\} + \eta_e h^{-1}[\vect{p}_h]$
} 
edge [line width=0.5mm, dotted,red] (DG);

\node[entity] (SHY-MIXED) [right of=MDG] [align=center] { 
  {\bf \color{blue} Stabilized Hybrid Mixed}\\ 
  $c\vect{p}_h - \DGdiv^* \tilde{u}_h = 0$\\  
  $\DGdiv \vect{p}_h + \mathcal{S}_u^\tau \tilde{u}_h = f_h$}
edge [line width=0.5mm, dotted,red] (MDG);

\node[entity] (PRIMAL) [above of=DG, yshift=10em] [align=center] { 
  {\bf
  \color{blue} Primal Methods}\\ 
   $c \vect{p}_h + \nabla_h u_h = 0$ \\ 
  $- \nabla_h^* \vect{p}_h = f_h$ \\ 
}
edge [line width=0.5mm,<-,blue] 
node[rotate=-42, yshift=0.8em] {{\color{red} $\tau =
  \rho^{-1}h_K^{-1}\quad\quad   \rho \to 0$}} (SHY-MIXED);    

\node[entity] (MIXED) [below of=DG, yshift=-10em] [align=center] { 
  {\bf \color{blue} Mixed Methods} \\ 
  $c\vect{p}_h - {\rm div}_h^* u_h = 0$ \\ 
  ${\rm div}_h \vect{p}_h = f_h$
}
edge [line width=0.5mm,<-,blue] 
node[rotate=-42, yshift=0.8em] {{\color{red} $\eta =
\rho^{-1}h_K^{-1}\quad \rho \to 0$}} (SHY-PRIMAL);    

\node[entity] (HY-MIXED) [above right of=MIXED, xshift=2.5em, yshift=1em] [align=center] {
  {\bf \color{blue} Hybrid Mixed} \\  
  $c\vect{p}_h -\DGdiv^* \tilde{u}_h = 0$ \\
  $\DGdiv \vect{p}_h =  f_h$
} 
edge [line width=0.5mm,<-,blue] 
node[rotate=42, yshift=0.7em] {{\color{red} Hybridization}} (MIXED)
edge [transform canvas={xshift=3mm}, line width=0.5mm,->,blue] 
node[rotate=42, yshift=-0.6em] {} (MIXED)
edge [line width=0.5mm,->,blue] 
node[rotate=42, yshift=0.7em] {{\color{red} Stabilization}} (SHY-MIXED)
edge [transform canvas={xshift=3mm}, line width=0.5mm,<-,blue] 
node[rotate=42, yshift=-0.5em] {{\color{red} $\tau=0$}} (SHY-MIXED);

\node[entity] (HY-PRIMAL) [below left of=PRIMAL, xshift=-2.5em,
  yshift=-1em] [align=center] { 
  {\bf
  \color{blue} Hybrid Primal}\\ 
   $c \boldsymbol{p}_h + \DGnabla u_h = 0$ \\ 
  $-\DGnabla^* \tilde{\vect{p}}_h = f_h$
}
edge [line width=0.5mm,<-,blue] 
node[rotate=42, yshift=-0.6em] {{\color{red} Hybridization}} (PRIMAL) 
edge [transform canvas={xshift=-3mm}, line width=0.5mm,->,blue] 
node[rotate=42, yshift=0.7em] {} (PRIMAL)
edge [line width=0.5mm,->,blue] 
node[rotate=42, yshift=-0.6em] {{\color{red} Stabilization}} (SHY-PRIMAL)
edge [transform canvas={xshift=-3mm}, line width=0.5mm,<-,blue] 
node[rotate=42, yshift=0.7em] {{\color{red} $\eta=0$}} (SHY-PRIMAL);

\node[entity] (CONMFEM) [left of=MIXED, xshift=-8em]
  [align=center] {{\bf \color{blue} Conforming}\\ {\bf \color {blue}
    Mixed Methods} 
} 
edge [line width=0.5mm,<-,blue] 
node[right, xshift=-4.5em, yshift=0.8em] {{\color{red}
$\vect{Q}_h\subset \boldsymbol H({\rm div}, \Omega)$}} (MIXED);

\node[entity] (NONMFEM) [right of=MIXED, xshift=8em]
  [align=center] {{\bf \color{blue} Nonconforming}\\ {\bf
    \color{blue}Mixed Methods} 
} 
edge [line width=0.5mm,<-,blue] 
node [left, xshift=4.5em, yshift=0.8em] {{\color{red}
$\vect{Q}_h\not\subset \boldsymbol H({\rm div}, \Omega)$}} (MIXED);

\node[entity] (CON) [left of=PRIMAL, xshift=-8em]
[align=center] {{\bf \color{blue} Conforming}\\ {{\bf
  \color{blue}Primal Methods}} \\ 
$\nabla^*(\alpha\nabla u_h)= f_h$} 
edge [line width=0.5mm,<-,blue]
node[left,text width=3cm, xshift=8em] {{\color{red} $V_h \subset
  H_0^1(\Omega)$ \\ \vspace{1mm}$\nabla_h V_h \subset c \vect{Q}_h$}}
  (PRIMAL); 

\node[entity] (NON) [right of=PRIMAL, xshift=8em]
[align=center] {
  {\bf \color{blue} Nonconforming} \\ {\bf \color{blue} Primal Methods}\\
  $\nabla^{*}_h(\alpha\nabla_h u_h)= f_h$} 
edge  [line width=0.5mm,<-,blue]  
node[right,text width=3cm, xshift=-3em] {{\color{red} $V_h
\not\subset H_0^1(\Omega)$\\ \vspace{1mm} $\nabla_h V_h \subset c
  \vect{Q}_h$}} (PRIMAL);

\node (S1) [above of=DG, yshift=-7em] [align=center] {
$
(\DGnabla u_h, \tilde{\vect{q}}_h) =
(\nabla_h u_h, \vect{q}_h) - \langle
u_h, \hat{\vect{q}}_h\cdot \vect{n}\rangle_{\partial
\mathcal{T}_h}  
$ 
\\ \\
$
\langle \mathcal{S}_p^\eta \tilde{\vect{p}}_h, \tilde{\vect{q}}_h
\rangle  := 
\langle\eta  (\vect{p}_h -
\hat{\vect{p}}_h) \cdot \vect{n}, (\vect{q}_h
-\hat{\vect{q}}_h) \cdot \vect{n}\rangle_{\partial \mathcal{T}_h}
$
};

\node (S2) [below of=DG, yshift=7em] [align=center] { 
$
\langle \mathcal{S}_u^\tau \tilde{u}_h, \tilde{v}_h\rangle :=
  \langle \tau(u_h -\hat{u}_h), v_h
-\hat{v}_h\rangle_{\partial \mathcal{T}_h}
$
\\ \\
$
(\DGdiv \vect{p}_h, \tilde{v}_h) =
 ({\rm div}_h\vect{p}_h, v_h) - \langle
\vect{p}_h\cdot \vect{n}, \hat{v}_h\rangle _{\partial
  \mathcal{T}_h} 
$
};
\end{tikzpicture}
\caption{A unified framework of FEMs} \label{tab:FEMs}
\end{sidewaystable}

\subsection{Primal formulation} \label{subsec:primal}
\paragraph{Conforming and nonconforming methods}
The first example of the FEMs using a primal formulation is the
conforming or nonconforming finite element, which requires the
continuity (or some weak continuity) of $u_h$ across the element
boundaries
\begin{equation}\label{eq:primal}
\left\{
\begin{aligned}
c \vect{p}_h + \nabla_h u_h &= 0 \qquad \text{in}~{\vect{Q}}_h^*,\\
-\nabla_h^*{\vect{p}}_h &= f_h   ~~\quad \text{in}~V_h^*.
\end{aligned}
\right.
\end{equation}
Here, $\nabla_h$ is derived by taking the gradient piecewise on each
element. 

\paragraph{Hybrid primal methods}
The hybrid formulation enforces the (weak) continuity of the
aforementioned primal method through a Lagrangian space on element
boundaries, and it can be formally written as 

\begin{equation}\label{eq:hybrid primal}
\left\{
\begin{aligned}
c \vect{p}_h + \DGnabla u_h &= 0 \qquad
\text{in}~\widetilde{\vect{Q}}_h^*,\\
-\DGnabla^* \tilde{\vect{p}}_h &= f_h ~~\quad\text{in}~V_h^*.
\end{aligned}
\right.
\end{equation}

\paragraph{\comments{Stabilized hybrid primal methods}}
The hybrid primal methods are unstable for even-order elements.  As
a remedy, a stabilization term
can be added to the first equation in \eqref{eq:hybrid primal} to
obtain
\begin{equation} \label{eq:WGprimal}
\left\{
\begin{aligned}
c \vect{p}_h + \mathcal{S}_{\vect{p}}^{\eta} \tilde{\vect{p}}_h
+\DGnabla u_h &= 0 \qquad \text{in}~\widetilde{\vect{Q}}_h^*,\\
-\DGnabla^* \tilde{\vect{p}}_h &= f_h ~~\quad \text{in}~V_h^*,
\end{aligned}
\right.
\end{equation}
where $\langle c {\vect{p}}_h, \tilde{\vect{q}}_h \rangle
:=(c\vect{p}_h, {\vect{q}}_h)$, and
\begin{equation}\label{WG:Stabilizer}
\langle \mathcal{S}_p^{\eta} \tilde{\vect{p}}_h, \tilde{\vect{q}}_h
\rangle := \sum_{K\in \mathcal T_h}\langle  \eta(\vect{p}_h - \hat{\vect{p}}_h) \cdot
\vect{n}_K, (\vect{q}_h -\hat{\vect{q}}_h) \cdot
\vect{n}_K\rangle_{\partial {K}}.
\end{equation}
Here, $\eta > 0$ is a stabilization parameter. In most cases, the stabilized 
hybrid primal methods are also named weak Galerkin (WG) methods.

\subsection{Mixed formulation} \label{subsec:mixed}
\paragraph{Mixed methods} The first example of the FEMs 
using a mixed formulation is the mixed finite element method, which
requires the continuity of $\vect{p}_h$ across elements, i.e.,
\begin{equation}\label{eq:mixedmethods}
\left\{
\begin{aligned}
c\vect{p}_h - {\rm div}_h^* {u}_h &= 0 \qquad
\hbox{in}~{\vect{Q}}_h^*,\\
{\rm div}_h \vect{p}_h &= f_h ~~\quad \hbox{in}~V_h^*.
\end{aligned}
\right.
\end{equation}
Here, $ {\rm div}_h$ is calculated by taking the gradient piecewise on each
element. 

\paragraph{Hybrid mixed methods} 
To reduce the number of the globally-coupled degrees of freedom,
hybrid mixed methods are also developed using the Lagrange multiplier
technique. Similar to the hybrid primal methods \eqref{eq:hybrid
primal}, the mixed form \eqref{HDIV} can be reformulated in terms of
DG-divergence as 
\begin{equation}\label{eq:hybrid-mixed}
\left\{
\begin{aligned}
c\vect{p}_h - \DGdiv^*\tilde{u}_h &= 0\qquad
\hbox{in}~{\vect{Q}}_h^*,\\
\DGdiv \vect{p}_h &= f_h~~\quad \hbox{in}~{\widetilde{V}}_h^*.
\end{aligned}
\right.
\end{equation}

\paragraph{\comments{Stabilized hybrid mixed methods}}
Under some choices of the space $\vect {Q}_h$ and $\widetilde{V}_h$,
the hybrid mixed method is no longer stable. Again, as a natural
device, a stabilization term is then added to the hybrid mixed methods to obtain
\begin{equation} \label{eq:hybridgalerkin}
\left\{
\begin{aligned}
c\vect{p}_h -\DGdiv^*\tilde{u}_h &= 0 \qquad
\hbox{in}~{\vect{Q}}_h^*,\\
\DGdiv \vect{p}_h + \mathcal{S}_u^{\tau} \tilde{u}_h  &= f_h ~~\quad
\hbox{in}~{\widetilde{V}}_h^*,
\end{aligned}
\right.
\end{equation}
where $ \langle c {\vect{p}}_h, \vect{q}_h \rangle :=(c\vect{p}_h,
{\vect{q}}_h)$, and
\begin{equation}\label{HDG:Stabilizer}
\langle \mathcal{S}_u^{\tau} \tilde{u}_h, \tilde{v}_h\rangle :=\sum_{K\in \mathcal T_h}\langle \tau(\hat u_h -\hat{u}_h), v_h -\hat{v}_h\rangle_{\partial K}.
\end{equation}
Here, $\tau > 0$ is a stabilization parameter. In most cases, the stabilized hybrid mixed methods
are also named hybrid discontinuous Galerkin (HDG) methods.
\comments{
\begin{remark}
If $\alpha$ (or $c$) is piecewise constant, we can eliminate $\vect{p}_h$ in \eqref{eq:hybridgalerkin} to
obtain the primal WG-FEM introduced in \cite{weakgalerkinLinMu,wang2013weak}:
\begin{equation} \label{primal_WG}
(\alpha \DGdiv^*  \widetilde{u}_h, \DGdiv^*
 \widetilde{v}_h)_\Omega + \sum_{K\in \mathcal{T}_h} \eta\langle u_h -
\hat{u}_h, v_h -\hat{v}_h\rangle_{\partial K} = (f, v_h)
  \qquad \forall \{v_h,\hat{v}_h\} \in \widetilde{V}_h.
\end{equation}
\end{remark}
}

\subsection{Discontinuous Galerkin methods} \label{subsec:DG}
Instead of using the Lagrange multiplier technique, a penalty term is
added in the bilinear form of the discontinuous Galerkin (DG) method to force continuity.
With the concepts of DG-gradient and DG-divergence defined as in
Definition \ref{DG_derivative_global}, most of the DG methods for
approximating the elliptic problem can be written as

\begin{subequations} \label{WG_global1}
\begin{empheq}[left=\empheqlbrace\,]{align}
c\vect{p}_h -\DGdiv^*  \tilde{u}_h &= 0 \qquad \qquad \quad
~~\hbox{in}~{\vect{Q}}_h^*, \label{WG_global_11}\\
-\DGnabla^*  \tilde{\vect{p}}_h &= f_h \qquad \qquad \quad
~\hbox{in}~{V}_h^*, \label{WG_global_21}\\
\hat{\vect{p}}_h &= \comments{\bar{\vect{p}}}(\vect{p}_h, u_h) \qquad {\rm on} \ {\cal E}_h,
  \label{define_flux_11} \\
\hat{u}_h &= \comments{\bar u}(\vect{p}_h, u_h) \qquad \hbox{on} \ {\cal
  E}_h,\label{define_flux_21}
\end{empheq}
\end{subequations}
where $\comments{\bar{\vect{p}}}$ and $\comments{\bar u}$ are the formulas for defining $\hat{\vect{p}}_h$
and $\hat{u}_h$, respectively, in the terms $\vect{p}_h$ and $u_h$, respectively.

If penalization is forced to impose the continuity of $u_h$, we
refer to the DG method as a primal DG method. In this paper, we also
use penalization to impose the continuity of $\vect p_h\cdot \vect n$,
which is called mixed DG method. 

Using the definition of DG-derivatives, we summarize the
relationships of different FEMs in Table \ref{tab:FEMs}. From the
table, we observe that the primal and mixed methods are the
discretization of primal and mixed forms, respectively.  Under certain
restrictions pertaining to finite element spaces, the
(non)conforming primal and mixed FEMs are obtained.  In addition, by
introducing Lagrange multipliers, the hybrid primal and hybrid mixed
FEMs can be derived from primal and mixed FEMs, respectively.
Generally speaking, these hybrid methods are not stable, and specific
finite element spaces need to be identified to make the schemes work.
By adding stabilization terms to the hybrid primal and hybrid mixed
FEMs, one obtains the stabilized hybrid primal (weak Galerkin) and
stabilized hybrid mixed (hybrid discontinuous Galerkin) FEMs, which
are the approximation of the primal form and mixed form, respectively.
Moreover, when the stabilization terms dominate, the solution of the
stabilized primal (mixed) FEMs approaches the approximation of
the mixed (primal) form under certain choices of discrete spaces, which
completes the outer loop in Table \ref{tab:FEMs}.  We emphasize that
only $\DGnabla$ and $\DGnabla^*$ appear in the primal
formulations, and only $\DGdiv$ and $\DGdiv^*$ are found in the mixed formulations. 

Traditionally, the DG methods in our framework can be viewed as a
compromise between the primal and mixed formulations by introducing
both $\DGnabla^*$ and $\DGdiv^*$.  Then different DG schemes can be
obtained by proper choices of numerical fluxes, see Table
\ref{tab:DG-traces}. With the help of this framework, we derive a new
family of mixed DG methods, which can be regarded as the dual form
of primal DG methods.

\subsection{Brief bibliographic comments} \label{sc:history}
The idea of conforming FEMs can be traced back to the 1940s
\cite{hrennikoff1941solution} and the Courant element
\cite{courant1943variational}.  After a decade, many works, such as
\cite{jones1964generalization, Feng1965finite,
fraeijs1965displacement, Yamamoto1966, zlamal1968finite,
pin1969variational, ciarlet1971multipoint, nicolaides1972class},
proposed more conforming elements and presented serious mathematical proofs
concerning error analysis and, hence, established the basic theory of
FEMs.  The first nonconforming element, the Wilson element
\cite{wilson1971incompatible}, was proposed for rectangle meshes. In
1972, Strang \cite{strang1972variational} developed the concept
of variational crimes in a nonconforming finite element method. 
The well-known Crouzeix-Raviart \cite{crouzeix1973conforming} finite
element appeared in 1973. For more examples on nonconforming elements,
we refer the reader to \cite{fortin1983non, fortin1985three,
santos1999nonconforming, park2003p, brenner2015forty} for second-order
elliptic problems, and \cite{wang2006morley, wang2013minimal,
hu2016canonical, wu2017nonconforming, wu2017pm} for higher-order
elliptic problems.

The mixed FEMs are tailored for approximating
problems with several primary variables. These include linear
elasticity in a stress-displacement system, Stokes equations, and flow
of a porous media within a velocity-pressure system.  The condition for the
well-posedness of mixed formulations is known as inf-sup or the
Ladyzhenskaya-Babu\v{s}ka-Breezi (LBB) condition
\cite{brezzi1974existence}.  Common simplicial mixed finite
elements are the Raviart-Thomas elements \cite{raviart1977mixed,
nedelec1980mixed} and the Brezzi-Douglas-Marini finite element
\cite{brezzi1985two, brezzi1987mixed}. Another family of mixed finite
elements was proposed in \cite{nedelec1986new} for tetrahedra, cubes,
and prisms by N{\'e}d{\'e}lec. Further, a family of rectangular mixed
finite elements in two- and three-space variables were designed in
\cite{brezzi1987efficient} by Brezzi, Douglas, Fortin, and Marini.
Mixed FEMs and their applications to various
problems were summarized in monographs by Brezzi, Fortin, and Boffi
\cite{brezzi1991mixed, boffi2013mixed}.

The idea of hybrid methods was proposed in
\cite{fraeijs1965displacement, pian1969basis, pin1969variational,
pian1972finite, henshell1973hybrid}.  The discrete spaces used in
the hybrid primal FEMs are discontinuous \cite{raviart1975hybrid,
wolf1975alternate, raviart1977primal}, and Lagrange multipliers
are adopted to force continuity across the element boundary.  For
instance, based on a primal hybrid variational principle, Raviart and Thomas
\cite{raviart1977primal} viewed the nonconforming finite elements as
discontinuous spaces on which weak continuity was imposed by multiplier space.
Further work on this hybrid idea was done by Babu\v{s}ka, Oden and Lee
in \cite{babuska1977mixed}, in which a flux variable was introduced in
bilinear form. In 1985, hybridization was shown to be more than an
implementation trick by Arnold and Brezzi in \cite{arnold1985mixed}.
More precisely, it was proven that the new unknown introduced by
hybridization could also be interpreted as the Lagrange multiplier
associated with a continuity condition on the approximate flux which
contains additional information about the exact solution.  After yet
nearly another decade, a new hybridization approach was put forward by
Cockburn and Gopalakrishnan \cite{cockburn2004characterization}, one
in which the hybrid formulation not only simplifies the task of
assembling the stiffness matrix for the multiplier, but also can be
used to establish links between apparently unrelated mixed methods.
Their approach also allows new, variable degree versions of those
methods to be devised and analyzed \cite{cockburn2005error,
cockburn2005incompressibleI, cockburn2005incompressibleII,
cockburn2005new}. Recently, a hybrid mixed method for working with
linear elasticity problem was presented in \cite{gong2017new}. 

For the mathematical theory behind the above methods, we refer to the
monographs by Ciarlet \cite{ciarlet2002finite} and by Brenner and
Scott \cite{brenner2007mathematical}. For more detailed discussion
on mixed and hybrid methods, we refer to the monographs
\cite{brezzi1991mixed, roberts1991mixed, boffi2013mixed}. 

The idea of using DG methods for elliptic equations can be traced back
to the late 1960s \cite{lions1968penalty}, and the same approach was
studied again in \cite{aubin1970approximation}.  Recently, DG methods
have been applied to purely elliptic problems; examples include the
interior penalty methods studied in \cite{babuvska1973nonconforming,
douglas1976interior, wheeler1978elliptic, arnold1982interior}, and the
local DG method for elliptic problem in \cite{cockburn1998local}.  DG
methods for diffusion and elliptic problems were considered in
\cite{brezzi1999discontinuous, brezzi2000discontinuous}. Additional
problems utilizing DG methods can be found in
\cite{brenner2005c,liu2009direct, liu2010direct,
hong2012discontinuous}.  A review of the development of DG methods up
to 1999 can be found in \cite{cockburn2000development} by Cockburn,
Karniadakis, and Shu.  In \cite{arnold2002unified}, Arnold, Brezzi,
Cockburn, and Marini unified the analysis of DG methods for elliptic
problems.  An abstract theory for using quasi-optimal nonconforming
and DG methods for elliptic problems was very recently presented in
\cite{veeser2017quasi,veeser2017quasiDG}.

In 2004, a new hybridization approach was presented in
\cite{cockburn2004characterization} by Cockburn and Gopalakrishnan.
This new hybridization approach was also applied to a DG method in
\cite{carrero2006hybridized}.  Using the LDG method to define the
local solvers, a super-convergent LDG-hybridizable Galerkin method for
second-order elliptic problems was designed in \cite{
cockburn2008superconvergent}.  In 2009, a unified analysis for the
hybridization of discontinuous Galerkin, mixed, and continuous
Galerkin methods for second order elliptic problems was presented in
\cite{cockburn2009unified} by Cockburn, Gopalakrishnan, and Lazarov. 
In 2010, a projection-based error analysis of HDG methods was
presented in \cite{cockburn2010projection} by Cockburn, Gopalakrishnan
,and Sayas, in which a projection operator was tailored to obtain the
$L^2$ error estimates for both potential and flux. A projection-based
analysis of the HDG methods used for convection-diffusion equations
for semi-matching nonconforming meshes was presented in
\cite{chen2014analysis}. A connection between the staggered DG method
and an HDG method was shown in \cite{chung2014staggered}.
Applications of HDG methods to linear and nonlinear elasticity
problems can be found in
\cite{soon2009hybridizable,fu2015analysis,kabaria2015hybridizable}.
More references to the recent developments of HDG methods can be found
in \cite{cockburn2016static}, to other applications in
\cite{chen2012analysis, cockburn2015hybridizable,
cockburn2016hybridizable}, to super-convergence analysis in
\cite{cockburn2012conditions, cockburn2012superconvergent,
cockburn2017devising, cockburn2017superconvergence}, and to a
posteriori error analysis in \cite{cockburn2012posteriori,
cockburn2013posteriori, cockburn2016contraction} . 

With the introduction of weak gradient and weak divergence, Wang and
Ye \cite{wang2013weak, wang2014weak} proposed and analyzed a WG method
for a second order elliptic equation formulated as a system of two
first-order linear equations.  Wang and Wang \cite{wang2015weak} gave
a summary of the idea and applications of WG methods for various
problems since the publication of \cite{wang2013weak}.  Wang and Wang
\cite{wang2017primal} also developed a primal-dual WG finite method
for second-order elliptic equations in non-divergence form, and then
they further extended their methods to Fokker-Planck type equations in
\cite{wang2017primalFP}.  Applications of WG to other problems can be
found.  See \cite{wang2016weak} for Stokes equations,
\cite{chen2016weak} for Darcy-Stokes flow, and \cite{mu2015weak} for
Maxwell equations. 

\comments{
In view of derivations, the {\it stablized hybrid mixed method} aims
to properly choose the numerical trace of the flux and can be viewed
as a stabilization approach for the hybrid mixed method, while the
{\it stablized hybrid primal method} stems from the proper definitions
of the weak gradient and weak divergence and can be viewed as a
stabilization approach for the hybrid primal method.  In the contexts
of HDG and WG, their relationship has been discussed in
\cite{cockburn2016static}.
}

The rest of the paper is organized as follows. In \S
\ref{sec:preliminaries}, we present some preliminary materials. In
\S \ref{sec:framework}, using a second-order elliptic problem
as an example, we show that most of the existing FEMs can be rewritten
in compact form by using the concept of DG-derivatives and
discuss their relationships.  In \S\ref{sec:WG}, we present the
well-posedness of WG methods under the specific, parameter-dependent
norms and further the convergence analysis of WG. We present
a similar result for the HDG methods in \S \ref{sec:HDG}.  In \S
\ref{sec:DG}, we discuss the DG and derive the mixed DG methods by
making a dual choice of numerical traces of primal DG methods. In
\S\ref{relationship}, we analyze the relationship between various  
FEMs and show that WG methods converge to mixed methods and that HDG
methods converge to primal methods under the limitation of parameter.
We give a brief summary in the last section.

Throughout this paper, we shall use letter $C$, which is independent
of mesh-size and stabilization parameters, to denote a generic
positive constant which may stand for different values at different
occurrences.  The notations $x \lesssim y$ and $x \gtrsim y$ mean $x
\leq Cy$  and $x \geq Cy$, respectively.

\section{Preliminaries} \label{sec:preliminaries}
In this section, we shall describe some basic notation and properties
of DG-derivatives.  

\subsection{DG notation}
Given a bounded domain $D\subset \mathbb{R}^d$ and a positive integer
$m$, $H^m(D)$ is the Sobolev space with the corresponding usual norm
and semi-norm, which are denoted by $\|\cdot\|_{m,D}$ and
$|\cdot|_{m,D}$, respectively. We abbreviate these by $\|\cdot\|_{m}$ and
$|\cdot|_{m}$, respectively, when $D$ is chosen as $\Omega$.  The
$L^2$-inner product on $D$ and $\partial D$ are denoted by $(\cdot,
\cdot)_{D}$ and $\langle\cdot, \cdot\rangle_{\partial D}$,
respectively.  $\|\cdot\|_{0,D}$ and $\|\cdot\|_{0,\partial D}$ are
the norms of Lebesgue spaces $L^2(D)$ and $L^2(\partial D)$,
respectively, and $\|\cdot\|=\|\cdot\|_{0,\Omega}$.

Let ${\cal E}_h^i={\cal E}_h \setminus \partial\Omega$ be the set of
interior edges and ${\cal E}_h^\partial={\cal E}_h\setminus{\cal
E}_h^i$ be the set of boundary edges. Further, let $h_e ={\rm diam}(e)$. For
$e\in {\cal E}_h^i$, we select a fixed normal unit direction, denoted by
$\boldsymbol{n}_e$. For $e\in {\cal E}_h^\partial$, we specify the
outward unit normal as $\boldsymbol{n}_e$. 
Let $e$ be the common edge of two elements $K^+$ and $K^-$, and let
$\vect{n}^i$ = $\vect{n}|_{\partial K^i}$ be the unit outward normal
vector on $\partial K^i$ with $i = +,-$.  For any scalar-valued
function $v$ and vector-valued function $\vect{q}$, let $v^{\pm}$ =
$v|_{\partial K^{\pm}}$, $\vect{q}^{\pm}$ = $\vect{q}|_{\partial
K^{\pm}}$.  Then, we define averages $\{\cdot\}, \{\!\!\{\cdot
\}\!\!\}$ and jumps $\llbracket \cdot \rrbracket$, $[\cdot]$ as
follows:
\begin{align*}
&\{v\} = \frac{1}{2}(v^+ + v^-),\qquad \{\vect{q}\} =
\frac{1}{2}(\vect{q}^+ + \vect{q}^-), \qquad\{\!\!\{\vect
q\}\!\!\}=\comments{ \frac{1}{2}(\vect{q}^+ + 
\vect{q}^-)\cdot \vect{n}_e}\qquad &{\rm on}\ e\in {\cal E}_h^i,\\
&\llbracket v \rrbracket  = v^+\vect{n}^+ + v^-\vect{n}^- ,\qquad [v]
= \comments{\llbracket v \rrbracket \cdot \vect{n}_e},
\qquad[\vect{q}] = \vect{q}^+\cdot \vect{n}^+ + \vect{q}^-\cdot
\vect{n}^-  \qquad &{\rm on}\ e\in {\cal E}_h^i,\\
&\llbracket v \rrbracket  = v \vect{n}, \qquad[v] = v,
\qquad\{\vect{q}\}=\vect{q},\qquad \{\!\!\{\vect q\}\!\!\}=\vect
q\cdot \vect n \qquad &{\rm on}\ e \in {\cal E}_h^\partial.
\end{align*}
Here, we specify $\vect{n}$ as the outward unit normal direction on
$\partial \Omega$.  

We define some inner products as follows:
\begin{equation} \label{equ:inner-product}
(\cdot,\cdot)_{\mathcal T_h}=\sum_{K\in \mathcal
T_h}(\cdot,\cdot)_{K},~~~~~ \langle\cdot,\cdot\rangle_{\mathcal
E_h}=\sum_{e\in \mathcal E_h}\langle\cdot,\cdot\rangle_{e},
~~~~~\langle\cdot,\cdot\rangle_{\mathcal E^i_h}=\sum_{e\in
\mathcal E^i_h}\langle\cdot,\cdot\rangle_{e},
~~~~\langle\cdot,\cdot\rangle_{\partial\mathcal T_h }=\sum_{K\in
\mathcal T_h}\langle\cdot,\cdot\rangle_{\partial K}.
\end{equation}
We now give more details about the last inner product. For any
scalar-valued function $v$ and vector-valued function $\boldsymbol q$, 
$$
\langle v, \boldsymbol  q \cdot \boldsymbol n \rangle_{\partial
  \mathcal T_h}=\sum_{K\in \mathcal T_h}\langle v, \boldsymbol q\cdot
  \boldsymbol n\rangle_{\partial K}=\sum_{K\in \mathcal T_h}\langle v,
  \boldsymbol q\cdot \boldsymbol n_K\rangle_{\partial K}.
$$
Here, we specify the outward unit normal direction $\boldsymbol n$
corresponding to the element $K$, namely $\boldsymbol n_K$.  In
addition, let $\nabla_h$ and ${\rm div}_h$ be defined by the relation
$$
\nabla_hv|_K=\nabla v|_K, \quad
{\rm div}_h\vect{q}|_K={\rm div} \vect{q}|_K \qquad \forall K \in
\mathcal{T}_h.
$$

With the definition of averages and jumps, we have the
following identity:
\begin{equation} \label{equ:dg-identity_1}
\sum_{K\in\mathcal{T}_{h}} \int_{\partial K} (\vect{q} \cdot \vect{n}_K) v ~ds
= \int_{\mathcal{E}_{h}}\{\vect{q}\}\cdot \llbracket v \rrbracket ~ds 
+ \int_{\mathcal{E}_{h}^i}[\vect{q}]\cdot\{v\} ~ds,
\end{equation}
namely,
\begin{equation}\label{eq:7}
\langle v, \boldsymbol q\cdot \boldsymbol n\rangle_{\partial \mathcal T_h}
=\langle \{\boldsymbol q\},\llbracket v\rrbracket\rangle_{\mathcal E_h}+
\langle [\boldsymbol q],\{v\}\rangle_{\mathcal E_h^i}.
\end{equation}
Further,
\begin{equation}\label{equ:dg-identity_2}
\{\vect{q}\}\cdot \llbracket v
\rrbracket=\{\!\!\{\vect{q}\}\!\!\}[v] \qquad \forall e \in
\mathcal{E}_h.
\end{equation}

Before discussing various Galerkin methods, we need to introduce the
finite element spaces associated with the triangulation
$\mathcal{T}_h$. For $k\ge 0$, we first define the spaces as follows
\begin{equation}\label{Spaces}
\begin{aligned}
V^{k}_h&=\{v_h\in L^2(\Omega): v_h|_{K}\in \mathcal{P}_k(K), \forall
K\in \mathcal T_h \},\\
\boldsymbol Q^k_h& = \{\boldsymbol p_h\in \boldsymbol L^2(\Omega):
\boldsymbol p_h|_K\in \boldsymbol {\mathcal{P}_k(K)}, \forall K\in
\mathcal T_h \},\\
\boldsymbol Q^{k,RT}_h &= \{\boldsymbol p_h\in \boldsymbol L^2(\Omega):
\boldsymbol p_h|_K\in \boldsymbol {\mathcal{P}_k(K)}+\boldsymbol x
P_k(K), \forall K\in \mathcal T_h \},
\end{aligned}
\end{equation}
where $\mathcal{P}_k(K)$ is the space of polynomial functions of
degree at most $k$ on $K$.  We also use the
following spaces associated with $\mathcal{E}_h$
\begin{equation}\label{Edge:spaces}
\begin{aligned}
\widehat{\boldsymbol{Q}}_h &=\{\hat{\boldsymbol p}_h:
\hat{\boldsymbol{p}}_h|_e\in \widehat Q(e)\boldsymbol{n}_e, \forall
e\in \mathcal{E}_h\},\\
{\widehat Q}_h&=\{{\hat p}_h: {\hat p}_h|_e\in \widehat Q(e), \forall e\in
\mathcal{E}_h\}, \\ 
{\widehat V}_h &= \{{\hat v}_h: {\hat v}_h|_e\in \widehat V(e), e\in
\mathcal{E}^i_h, {\hat v}_h|_{\mathcal E_h^{\partial}}=0\}, \\
{\widehat Q}^k_h&=\{{\hat p}_h\in L^2(\mathcal E_h): {\hat
p}_h|_e\in \mathcal{P}_k(e), \forall e\in \mathcal{E}_h\}, \\
{\widehat V}^k_h&=\{{\hat v}_h \in L^2(\mathcal E_h): {\hat
v}_h|_e\in \mathcal{P}_k(e), \forall e\in \mathcal{E}^i_h, {\hat
  v}_h|_{\mathcal E_h^{\partial}}=0\},
\end{aligned}
\end{equation}
where $\widehat Q(e), \widehat V(e)$ are some local spaces on $e$ and
$\mathcal{P}_k(e)$ is the space of polynomial functions of
degree at most $k$ on $e$.

\subsection{Some properties of DG-derivatives}

Are the DG-gradient $\DGnabla$ and DG-divergence $-\DGdiv$ operators 
dual operators with each other like the
classical gradient and divergence operators?  Generally, the answer is no.  
However, by the definition of DG-derivatives
\eqref{DG_gradient}, \eqref{DG_divergence}, 
we have the following relationship:
\begin{equation}
\begin{aligned}
\langle -\DGdiv^* \tilde{u}_h, {\vect{p}}_h \rangle 
& =  -\langle \tilde{u}_h, \DGdiv \vect{p}_h \rangle \\ 
& = ( \nabla_h u_h, \vect{p}_h)_{\mathcal T_h} + \langle \hat {u}_h -
u_h, \vect{p}_h\cdot \vect{n}\rangle _{\partial \mathcal T_h}, \\
\langle \DGnabla {u}_h, \tilde{\vect{p}}_h \rangle
& = (\nabla_h u_h, \vect{p}_h)_{\mathcal T_h}- \langle u_h,
\hat{\vect{p}}_h\cdot \vect{n}\rangle _{\partial \mathcal T_h}, 
\end{aligned}
\end{equation}
and
\begin{equation}
\begin{aligned}
\langle \DGnabla^*  \tilde{\vect{p}}_h, {u}_h \rangle & = \langle
\tilde{\vect{p}}_h, \DGnabla u_h \rangle \\
&= - ({\rm div}_h \vect{p}_h, u_h)_{\mathcal T_h} + \langle (
\vect{p}_h - \hat{\vect{p}}_h)\cdot \vect{n}, u_h\rangle _{\partial
\mathcal T_h}, \\
\langle -\DGdiv {\vect{p}}_h, \tilde{u}_h \rangle &= - ({\rm div}_h
    \vect{p}_h, u_h)_{ \mathcal{T}_h } +   \langle \vect{p}_h\cdot
\vect{n},  \hat{u}_h\rangle _{\partial  \mathcal{T}_h }. 
\end{aligned}
\end{equation}
By direct calculation, we have the following lemma.  

\begin{lemma} \label{lem:conditional-dual}
Assume $\hat u_h=0~\hbox{on}~\mathcal E_h^\partial$. Then, $\DGnabla =
-\DGdiv^*$ if one of the following conditions holds:
\begin{enumerate}
\item [(i)] $\vect{p}_h\cdot \vect{n}_K |_{{\cal E}_h} =
\hat{\vect{p}}_h\cdot \vect{n}_K$;
\item [(ii)] $u_h|_{{\cal E}_h} = \hat{u}_h$;
\item [(iii)] $\hat{\vect{p}}_h\cdot \vect{n}_K = \{\vect{p}_h\}\cdot
\vect{n}_K$ and $\hat{u}_h = \{u_h\}$.
\end{enumerate}
\end{lemma}

Therefore, we say $\DGnabla $ and $-\DGdiv $ are {\it conditionally
dual} with each other.  Based on this observation, we find that the
DG-derivatives are good approximations for classical weak derivatives when one
of the above conditions is satisfied. Furthermore, in the following
subsections, we see that the (stabilized) hybrid primal and
(stabilized) hybrid mixed methods weakly satisfy condition (i) or
(ii), and that the DG methods adopt condition (iii) approximately.

\begin{table}[!htbp]
\begin{center}
\begin{tabular}{|c|c|c|c|c|}
\hline 
DG-derivatives & $\DGnabla$ & $\DGnabla^*$ & $\DGdiv$ & $\DGdiv^*$ \\
\hline 
Weak derivatives \cite{fraeijs1965displacement, wang2013weak} & $-{\rm
div}_w^*$ & $-{\rm div}_w$ & $-\nabla_w^*$ & $-\nabla_w$ \\
\hline 
\end{tabular}
\end{center}
\caption{Relationship of DG-derivatives and weak derivatives when
  $\widehat{u}$ and $\widehat{\vect{p}}$ are single-valued}
\label{table:weak_derivatives_relation}
\end{table}%

The DG-derivatives introduced in this paper are essentially the same
as the weak derivatives first introduced by Wang and Ye
\cite{wang2013weak}, where the weak derivatives are locally defined at
the first place.  Weak derivatives then can be defined globally
element-by-element.  To distinguish the original definition of weak
derivatives by Wang and Ye \cite{wang2013weak}, we introduce the
different names {\it DG-derivatives}.  If $\hat{u}$ and
$\hat{\vect{p}}$ are single-valued globally, then the relationship
between DG-derivatives and weak derivatives is shown in Table
\ref{table:weak_derivatives_relation}.  We introduce DG-derivatives
mainly due to their consistency with the classical weak derivatives on
the conforming spaces as shown in the following lemma.

\begin{lemma} \label{lem:dg-property}
It holds that 
\begin{equation} \label{equ:DG-consistency1}
\DGnabla = \nabla_h\quad {\rm if}\quad V_h \subset
V_{non}=\{u_h\in V_h: \langle u_h,
\hat{\vect{q}}_h\cdot \vect{n}_K \rangle_{\partial \mathcal
T_h}=0, \forall \hat{\vect{q}}_h\in \widehat{\vect{Q}}_h\} . 
\end{equation}
Further, 
\begin{equation} \label{equ:DG-consistency}
\DGnabla = \nabla \quad {\rm if}\quad V_h \subset H^1(\Omega), 
\quad {\rm and} \quad \DGdiv = {\rm div} \quad {\rm if}\quad
\vect{Q}_h \subset \boldsymbol H({\rm div},\Omega). 
\end{equation}
\end{lemma}
\begin{proof}
If $u_h \in V_{non}$, for any $\tilde{\vect q}_h \in \widetilde{\vect
Q}_h$, we have 
$$
\begin{aligned}
\langle \DGnabla u_h, \tilde{\vect{q}}_h \rangle &= -(u_h, {\rm
div}\vect{q}_h)_{\mathcal T_h} + \langle u_h, (\vect{q}_h -
  \hat{\vect{q}}_h)\cdot \vect{n}_K \rangle_{\partial \mathcal T_h} \\
&= -(u_h, {\rm div}\vect{q}_h)_{\mathcal T_h} + \langle u_h, \vect{q}_h
  \cdot \vect{n} \rangle_{\partial \mathcal{T}_h}\\
  &=(\nabla_h u_h, \vect{q}_h)_{\mathcal T_h},
\end{aligned}
$$  
which gives rise to the first consistency relationship
\eqref{equ:DG-consistency1}.  The second can be proven in a similar
way.  
\end{proof}

\section{Basic setup} \label{sec:framework}

Now we start with the second-order elliptic equation \eqref{pois} and
set $\vect {p}=-\alpha\nabla u$ to obtain the following form 
\begin{equation}\label{H11}
\left\{
\begin{aligned}
c \vect{p} + \nabla u &= 0 \qquad  {\rm in}\ \Omega, \\ 
{\rm div} \vect{p} &= f  \qquad {\rm in}\ \Omega. \\
\end{aligned}
\right.
\end{equation}
Multiplying the first and second equations by $\boldsymbol q_h\in
\boldsymbol Q_h$ and $v_h\in V_h$, respectively, then integrating on an
element $K\in \mathcal T_h$ we get
\begin{equation}\label{eq:element}
\left\{
\begin{aligned}
\displaystyle (c\boldsymbol p, \boldsymbol q_h)_K-(u, {\rm div}
\boldsymbol q_h)_K+\langle u, \boldsymbol q_h\cdot \boldsymbol
n_K\rangle_{\partial K} &= 0 \qquad \qquad\quad \forall \boldsymbol
q_h\in \boldsymbol Q_h,\\
\displaystyle (\boldsymbol p, \nabla v_h)_K-\langle \boldsymbol p\cdot
\boldsymbol n_K, v_h\rangle_{\partial K}&=-(f, v_h)_K ~\quad \forall
v_h\in V_h.
\end{aligned}
\right.
\end{equation}
Summing on all $K\in \mathcal T_h$, we have  
\begin{equation} \label{eq:summing}
\left\{
\begin{aligned}
(c\boldsymbol p, \boldsymbol q_h)_{\mathcal T_h} - (u, {\rm div}_h
\boldsymbol q_h)_{\mathcal T_h} + \langle u, \boldsymbol q_h\cdot
\boldsymbol n\rangle_{\partial \mathcal{T}_h} &= 0 \qquad \qquad \quad
\forall \boldsymbol q_h\in \boldsymbol Q_h,\\
(\boldsymbol p, \nabla_h v_h)_{\mathcal T_h} - \langle \boldsymbol
p\cdot \boldsymbol n, v_h\rangle_{\partial \mathcal{T}_h} &= -(f,
v_h)_{\mathcal T_h} ~\quad \forall v_h\in V_h.
\end{aligned}
\right.
\end{equation}
Now, we approximate $u$ and $\boldsymbol p$ by $u_h\in V_h$ and
$\boldsymbol p_h\in \boldsymbol Q_h$, respectively, and the trace of
$u$ and the flux $\boldsymbol p\cdot\boldsymbol n$ on $\partial K$ by
$\check u_h$, $\check{\boldsymbol p}_h \cdot\boldsymbol n$ (see Figure
\ref{fig:checks}). Hence, we have
\begin{equation}\label{eq:common:1}
\left\{
\begin{aligned}
(c\boldsymbol p_h, \boldsymbol q_h)_{\mathcal T_h} - (u_h, {\rm div}_h
    \boldsymbol q_h)_{\mathcal T_h} + \langle \check u_h, \boldsymbol
q_h\cdot \boldsymbol n\rangle_{\partial \mathcal{T}_h} &= 0 \qquad
\qquad \quad \forall \boldsymbol q_h\in \boldsymbol Q_h,\\
(\boldsymbol p_h, \nabla_h v_h)_{\mathcal T_h} - \langle
\check{\boldsymbol p}_h\cdot \boldsymbol n, v_h\rangle_{\partial
\mathcal{T}_h} &= -(f, v_h)_{\mathcal T_h} ~\quad \forall v_h\in
V_h.
\end{aligned}
\right.
\end{equation}
Next, we derive appropriate equations for the variables of
$\check u_h$ and $\check{\boldsymbol p}_h$. There are three different
approaches. The starting point of the first two approaches is the
following relationship:
\begin{equation}\label{hdg-wg}
\check {\boldsymbol p}_h \cdot \boldsymbol n_K + \tau \check u_h
=\boldsymbol p_h \cdot \boldsymbol n_K + \tau \hat {P}_h(u_h), \qquad
\check{\boldsymbol p}_h={\check p}_h\boldsymbol n_e.
\end{equation}
Here, $\hat {P}_h$ is an operator which will be specified later.
Note that we only use one of ${\check p}_h$ and $\check u_h$ as an
unknown and then use \eqref{hdg-wg} to determine the other.  The third
approach is to define ${\check p}_h$ and $\check u_h$ in terms of $u_h$
and $\boldsymbol p_h$. The details are given below. 
\begin{figure}[!htbp]
\centering
\begin{tikzpicture}[scale=1.2]

\draw [thick] (0,0)--(2.5,1);
\draw [thick] (2.5,1)--(2.5,-1);
\draw [thick] (0,0)--(2.5,-1);
\draw [thick] (5,0)--(2.5,-1);
\draw [thick] (5,0)--(2.5,1);

\draw(1,0)node{$K_1$};
\draw(4,0)node{$K_2$};

\draw(2.02,-0.3)node{{\footnotesize $\check{\boldsymbol p}_1\cdot
  \boldsymbol{n}_{K_1}$}};
\draw(2.35,0.4)node{{\footnotesize $\check{u}_1$}};
\draw(3.06,-0.3)node{{\footnotesize $\check{\boldsymbol p}_2\cdot
  \boldsymbol{n}_{K_2}$}};
\draw(2.7,0.4)node{{\footnotesize $\check{u}_2$}};
\end{tikzpicture}
\caption{$\check{\boldsymbol p}_h\cdot {\boldsymbol n}$ and
  $\check{u}_h$}
\label{fig:checks}
\end{figure}

\paragraph{First approach: \comments{Stabilized hybrid primal (WG)}}
We first set $\hat P_h$ to be an identity operator in \eqref{hdg-wg} and
$\check{\boldsymbol p}_h := \hat{\boldsymbol p}_h = \hat
p_h\boldsymbol n_e\in \widehat{\boldsymbol Q}_h $ to be a single-valued
unknown.  The ``continuity'' of $\check u_h$ is then enforced weakly
as follows:
\begin{equation}\label{eq:common:2}
\langle{\check u}_h,\hat{\boldsymbol q}_h\cdot \boldsymbol
n\rangle_{\partial\mathcal T_h}=0 \qquad \forall \hat{\boldsymbol q}_h
\in \widehat{\boldsymbol Q}_h,
\end{equation}
where $\check u_h$ is again given by \eqref{hdg-wg}.  From the
identity \eqref{equ:dg-identity_1} and the fact that
$[\hat{\boldsymbol q}_h]=0$, a straightforward calculation shows that
\eqref{eq:common:2} can be rewritten as 
\begin{equation}\label{jump:mean:product}
\langle[{\check u_h}], {\hat q}_h\rangle_{\mathcal E_h} := \sum_{e\in
\mathcal{E}_h}\langle[{\check u_h}], {\hat q}_h\rangle_{e}=0 \qquad
\forall {\hat q}_h\in {\widehat Q}_h.
\end{equation} 
From Definition \ref{DG_derivative_global} for the DG-gradient, WG
methods can be rewritten exactly as \eqref{eq:WGprimal}.  We should note
that the stabilization parameter $\eta$ in \eqref{WG:Stabilizer} is in
fact $\tau^{-1}$, where $\tau$ is the parameter shown in
\eqref{hdg-wg}.  As a special case, when $\eta=0$, we obtain the
hybrid primal methods \cite{raviart1975hybrid, wolf1975alternate,
raviart1977primal}, namely \eqref{eq:hybrid primal}. 

As a further special case, when $\eta=0$ and $\widehat{\boldsymbol
Q}_h=\{\boldsymbol 0\}$, we obtain the primal methods. Then, under
certain ``continuity'' properties pertaining to $V_h$ and $\boldsymbol
Q_h$, the operator $\DGnabla^*$ reduces to ${\nabla_h^*}$ and the
operator $\DGnabla$ reduces to $\nabla_h$ (see Lemma
\ref{lem:dg-property}).  Hence, the primal methods read as
\eqref{eq:primal}.  For instance, if we choose $V_h$ satisfying the
``continuity'' condition $V_h\subset H^1_0(\Omega)$ and $\boldsymbol
Q_h$ such that $\nabla V_h\subset c\boldsymbol Q_h$, we obtain the
conforming FEMs \cite{argyris1954energy, turner1956stiffness,
mikhlin1964variational, cea1964approximation, pian1964derivation,
Feng1965finite}.  And if we choose the $V_h$ satisfying the weak
``continuity'' condition 
\begin{equation} \label{nonconforming}
V_h = V_{non}^{k+1} = \{v_h \in V_h^{k+1}, \int_{e} [v_h]\cdot
\hat{{q}}_h\, ds = 0, \forall \hat{{q}}_h\in
\widehat{Q}^k_h, \forall e\in {\cal E}_h^i\},
\end{equation}
we obtain the Crouzeix-Raviart (CR) nonconforming element
\cite{crouzeix1973conforming} when $k=0$ and $\nabla_h V_h\subset
c\boldsymbol Q_h$.

\paragraph{Second approach: \comments{Stabilized hybrid mixed (HDG)}}
We set $\check u_h:=\hat u_h\in \hat V_h$ to be a single-valued unknown.
The ``continuity'' of $\check {\boldsymbol p}_h$ is then enforced
weakly as follows:
\begin{equation}\label{eq:common:4}
\langle\check{\boldsymbol p}_h\cdot\boldsymbol{n},
\hat{v}_h\rangle_{\partial\mathcal T_h}=0 \qquad\forall \hat v_h \in
\widehat V_h,
\end{equation}
where $\check{\boldsymbol p}_h$ is given by \eqref{hdg-wg}.  From the
identity \eqref{equ:dg-identity_1} and the fact that $[\hat{v}_h]=0$,
a straightforward calculation shows that \eqref{eq:common:4} can be
rewritten as 
\begin{equation} 
\langle [{\check {\boldsymbol p}}_h], \hat{v}_h\rangle_{\mathcal E^i_h} := 
\sum_{e\in \mathcal{E}^i_h} \langle [\check{\boldsymbol p}_h], \hat v_h \rangle_e =0
\qquad \forall \hat v_h\in \widehat V_h.
\end{equation}
If $\hat P_h$ is an identity operator, using the DG-divergence from
Definition \ref{DG_derivative_global}, we can rewrite the standard HDG
method as \eqref{eq:hybridgalerkin}.  If $\hat P_h$ is a local $L^2$
projection, namely 
$$
\hat{P}_h|_{\partial K} := \hat P_{\partial K}:L^2(\partial
K)\mapsto \widehat Q(\partial K),
$$
where $\widehat Q(\partial K)$ is the trace space on $\partial K$,
namely $\widehat Q(\partial K)=\bigcup\limits_{e\in \partial K}
\widehat{Q}(e)$, we obtain the modified HDG methods with reduced
stabilization \cite{oikawa2015hybridized}. 


As a special case, when $\tau=0$, we obtain the hybrid mixed methods
\cite{arnold1985mixed,brezzi1991mixed,cockburn2004characterization},
namely \eqref{eq:hybrid-mixed}.  As a further special case, when
$\tau=0$ and ${\widehat V}_h=\{0\}$, we obtain the mixed methods.
Then, under certain ``continuity'' properties pertaining to $\boldsymbol
Q_h$ and $V_h$, the operator $\DGdiv$ reduces to ${\rm div}_h$. Hence,
the mixed methods \cite{raviart1977mixed, nedelec1980mixed,
brezzi1985two, brezzi1987mixed, brezzi1991mixed, boffi2013mixed}
read as \eqref{eq:mixedmethods}.
 
\paragraph{Third approach: DG}
We define $\check{\boldsymbol p}_h= \hat{\boldsymbol p}_h$ and $\check
u_h=\hat u_h$ in terms of $u_h$ and $\boldsymbol p_h$, namely 
\begin{equation}
\left\{
\begin{array}{l}
\hat{\boldsymbol{p}}_h = \comments{\bar{\vect{p}}}(\boldsymbol p_h, u_h) \qquad \hbox{on}~
\mathcal E_h, \\
\hat u_h = \comments{\bar u}(\boldsymbol p_h, u_h) \qquad \hbox{on}~ \mathcal E_h.
\end{array}
\right.
\end{equation}

These three different approaches can be summarized in Table \ref{WG:HDG:DG}.

\begin{table}[!htbp]
\begin{center}
\begin{tabular}{|c|c|c|c|}
 \hline
Volume Equations
& \multicolumn{3}{|c|}{
$
\begin{aligned}\displaystyle
(c\boldsymbol p_h, \boldsymbol q_h)_{\mathcal T_h}+(u_h, {\rm
div}_h\boldsymbol q_h)_{\mathcal T_h} 
+\langle \check u_h, \boldsymbol  q_h\rangle_{\partial\mathcal T_h}&=0,
~~\qquad\qquad\forall \boldsymbol q_h\in \boldsymbol Q_h,\\
\displaystyle
(\boldsymbol p_h, \nabla_h v_h)_{\mathcal T_h}-\langle v_h,
\check{\boldsymbol p}_h\rangle_{\partial \mathcal T_h} &= -(f,
v_h)_{\mathcal T_h} \quad \forall v_h\in V_h.
\end{aligned}
$
}\\
\hline
\multirow{2}{*}{Interelement Equations}
&   \multicolumn{2}{|c|}{
$\check{\boldsymbol p}_h\cdot\boldsymbol n_K+\tau \check u_h
=\boldsymbol p_h\cdot\boldsymbol n_K+\tau u_h$
}& \multirow{2}{*}{
$
\begin{array}{l}
\hat{\boldsymbol p}_h=\comments{\bar{\vect{p}}}(\boldsymbol p_h, u_h) \\
\hat u_h=\comments{\bar u}(\boldsymbol p_h, u_h) 
\end{array}
$
}
\\
\cline{2-3}
& 
$\langle[{\check u_h}], {\hat q}_h\rangle_{\mathcal E_h}=0~~\forall
{\hat q}_h \in {\widehat Q}_h$
&
$
\langle  \hat{v}_h,[{\check {\boldsymbol p}}_h]\rangle_{\mathcal E^i_h}=0~~\forall
\hat v_h \in \widehat V_h $&\\[0.85ex]
\hline
Type of Methods & WG & HDG & DG\\
\hline
\end{tabular}
\end{center}
\caption{Summary of WG,  HDG and DG methods}
\label{WG:HDG:DG}
\end{table}

\section{Stability and convergence analysis of \comments{stabilzed hybrid primal (WG) methods}} \label{sec:WG}
The \comments{stabilized hybrid primal (WG)} methods read: Find $(\tilde{\boldsymbol p}_h, u_h)\in
\widetilde{\boldsymbol Q}_h\times V_h$ such that for any
$(\tilde{\boldsymbol q}_h, v_h)\in \widetilde{\boldsymbol Q}_h\times
V_h$,
\begin{equation}\label{Stabilizedpweak_WG}
\left\{
\begin{aligned}
a_w(\tilde{\boldsymbol p}_h,\tilde{\boldsymbol
    q}_h)+b_w(u_h,\tilde{\boldsymbol q}_h) &= 0,\\
b_w(\tilde{\boldsymbol p}_h,v_h) &= -(f,v_h)_{\mathcal T_h}.
\end{aligned}
\right.
\end{equation}
Here,
\begin{equation}\label{WG_aform}
\begin{aligned}
a_w(\tilde{\boldsymbol p}_h, \tilde{\boldsymbol q}_h) &= (c
\boldsymbol{p}_h, \boldsymbol{q}_h) +
\langle\eta(\boldsymbol{p}_h-\hat{\boldsymbol p}_h)\cdot
\boldsymbol{n}, (\boldsymbol{q}_h-\hat{\boldsymbol q}_h)\cdot
\boldsymbol{n}\rangle_{\partial {\mathcal T_h}},\\
b_w(\tilde{\boldsymbol p}_h,v_h) &= (\boldsymbol{p}_h, \nabla_h
v_h)_{\mathcal T_h}-( \hat{\boldsymbol
p}_h\cdot\boldsymbol{n}_K,{v}_h)_{\partial \mathcal T_h},
\end{aligned}
\end{equation}
and $\eta > 0$ is the stabilized parameter.

The following lemma shows the consistency property of \comments{stabilized hybrid primal (WG)} methods.
\begin{lemma}\label{Error:common:WG}
Let $f\in L^2(\Omega)$ and $(\boldsymbol p, u)$ be the solution of
\eqref{H1} or \eqref{HDIV}, then $(\boldsymbol p,  u)$ satisfies the
following consistency property 
\begin{equation}\label{Stabilizedpweak_WG:consistence}
\left\{
\begin{aligned}
a_w(\boldsymbol{p},\tilde{\boldsymbol q}_h) + b_w(u,\tilde{\boldsymbol
q}_h) &= 0 \qquad\qquad\qquad \forall \tilde{\boldsymbol q}_h\in
\widetilde{\boldsymbol Q}_h, \\
b_w(\boldsymbol{p},v_h) &= -(f,v_h)_{\mathcal T_h} \qquad \forall v_h\in V_h.
\end{aligned}
\right.
\end{equation}
\end{lemma}

\subsection{Gradient-based uniform inf-sup condition}

In this subsection, we shall show the well-posedness of WG methods
\eqref{WG_aform} when choosing the parameter $\eta$ as $\eta =\rho
h_K$ for $\rho>0$. More precisely, we will give the uniform inf-sup
condition under the following parameter-dependent norms 
\begin{equation}\label{WG:grad:norm:v}
\begin{aligned}
\|\tilde{\boldsymbol p}_h\|^2_{0,\rho,h} &= (c \boldsymbol{p}_h,
\boldsymbol{p}_h)_{\mathcal T_h} + \rho\sum_{K\in
\mathcal{T}_h} h_K \langle(\boldsymbol{p}_h-\hat{\boldsymbol
p}_h)\cdot \boldsymbol{n}_K, (\boldsymbol{p}_h-\hat{\boldsymbol
p}_h)\cdot \boldsymbol{n}_K\rangle_{\partial K}, \\
\|v_h\|^2_{1,\rho,h} &= \|\nabla_h v_h\|^2 + \rho^{-1}\sum_{e\in
  \mathcal{E}_h}h_e^{-1}\|\hat Q_e([v_h])\|^2_{0,e},
\end{aligned}
\end{equation}
where $\hat Q_e$ is the $L^2$ projection from $L^2(e)$ to $\widehat
Q(e)$.  We point out that $\|v_h\|_{1,\rho,h}$ is indeed a norm on
$V_h$ if $\widehat Q_h^0\subset \widehat Q_h$, namely $\widehat Q_h$
contains the piecewise constant space on $\mathcal E_h$.

Using these parameter-dependent norms \eqref{WG:grad:norm:v}, we
have the following results, whose details will be reported in
\cite{hong2017uniformly, hong2018extended}.

\begin{theorem}[\cite{hong2017uniformly, hong2018extended}]\label{bounded:coercivity:aw}
For any $0<\rho\leq 1$, and for any $\tilde{\boldsymbol
p}_h, \tilde{\boldsymbol q}_h\in \widetilde{\boldsymbol Q}_h, v_h\in
V_h$, we have 
\begin{equation} \label{equ:WG-bound-grad}
\begin{aligned}
|a_w(\tilde{\boldsymbol p}_h, \tilde{\boldsymbol q}_h)| &\leq
\|\tilde{\boldsymbol p}_h\|_{0,\rho,h} \|\tilde{\boldsymbol
  q}_h\|_{0,\rho,h}, \\
b_w(\tilde{\boldsymbol p}_h, v_h) & \leq C_w \|\tilde{\boldsymbol
  p}_h\|_{0,\rho,h} \|v_h\|_{1,\rho,h}, \\
a_w(\tilde{\boldsymbol p}_h, \tilde{\boldsymbol p}_h) &\geq
\|\tilde{\boldsymbol p}_h\|^2_{0,\rho,h},
\end{aligned}
\end{equation}
where $C_{w}$ is independent of both mesh size $h$ and $\rho$.
\end{theorem}

\begin{theorem}[\cite{hong2017uniformly, hong2018extended}] \label{inf-sup:bw}
Assume that $\nabla_h V_h\subset \boldsymbol Q_h$, then for any
$0<\rho\leq 1$, we have 
\begin{equation} \label{equ:WG-infsup-grad}
 \inf_{v_h\in V_h} \sup_{\tilde{\boldsymbol p}_h \in
\widetilde{\boldsymbol Q}_h} \frac{b_w(\tilde{\boldsymbol
p}_h,v_h)}{\|v_h\|_{1,\rho,h} \|\tilde{\boldsymbol p}_h\|_{0,\rho,h}}
\geq \beta_{w,0},
\end{equation}
where $\beta_{w,0}$ is independent of both mesh size $h$ and $\rho$.
\end{theorem}

\begin{corollary}[\cite{hong2017uniformly, hong2018extended}] \label{uniform_Stable_WG_Hong}
Assume that $\nabla_h V_h\subset \boldsymbol Q_h$. Then there exists
a unique solution $(\tilde{\boldsymbol p}_h, u_h)\in \widetilde
{\boldsymbol Q}_h\times V_h$ satisfying \eqref{Stabilizedpweak_WG}
with $\eta=\rho h_K$. Further, for any $0<\rho\leq 1$ the following
estimate holds
\begin{equation} \label{equ:WG-well-grad}
\|\tilde{\boldsymbol p}_h\|_{0,\rho,h} + \|u_h\|_{1,\rho,h}\leq
C_{w,1} \|f\|_{-1,\rho,h},
\end{equation}
where $C_{w,1}$ is a uniform constant with respect to both $\rho$ and
$h$ and $ \|f\|_{-1,\rho,h} := \sup\limits_{v_h \in V_h} \frac{(f,
v_h)_{\mathcal T_h}}{\|v_h\|_{1,\rho,h}}$.
\end{corollary}

\begin{theorem}[\cite{hong2017uniformly}] \label{Error:uniform_WG:Primal}
Let $(\boldsymbol p,  u)$ be the solution of \eqref{H1} and assume
that $\boldsymbol p\in \boldsymbol H^1(\Omega)$. Further, let
$(\tilde{\boldsymbol p}_h, u_h)\in \widetilde{\boldsymbol Q}_h\times
{V}_h$ be the solution of \eqref{Stabilizedpweak_WG} with $\eta=\rho
h_K$. If we choose the spaces $\widetilde {\boldsymbol Q}_h\times {V}_h$
such that $\nabla_h V_h\subset \boldsymbol Q_h$, then for any
$0<\rho\leq 1$ the following estimate holds 
\begin{equation} 
\|\boldsymbol p-\tilde{\boldsymbol p}_h\|_{0,\rho,h} +
\|u-u_h\|_{1,\rho,h} \leq C_{e,1}
\inf\limits_{\tilde{\boldsymbol q}_h \in \widetilde{\boldsymbol Q}_h,
v_h\in V_h} \left( \|\boldsymbol p-\tilde{\boldsymbol q}_h\|_{0,
\rho,h} + \|u-v_h\|_{1,\rho,h} \right),
\end{equation}
where $C_{e,1} $ is a uniform constant with respect to both $\rho$ and
$h$.
\end{theorem}

\begin{corollary}[\cite{hong2017uniformly}]\label{WG-order:grad}
Let $(\boldsymbol p, u)$ be the solution of \eqref{H1} and
$\boldsymbol p\in \boldsymbol H^{k+1}(\Omega), u\in H^{k+2}(\Omega)$,
and $(\tilde{\boldsymbol p}_h, u_h)\in \widetilde{\boldsymbol
Q}_h \times V_h$ be the solution of \eqref{Stabilizedpweak_WG} with
$\eta=\rho h_K$. If we choose the spaces ${V}_h\times\boldsymbol
{Q}_h\times\widehat Q_h$ as ${V}_h^{k+1}\times\boldsymbol
{Q}_h^k\times\widehat Q_h^k$, then for any $0<\rho\leq 1$ the
following estimate holds
\begin{equation}
\|\boldsymbol p-\tilde{\boldsymbol p}_h\|_{0,\rho,h} +
\|u-{u}_h\|_{1,\rho,h} \leq C_{r,1} h^{k+1} (\|\boldsymbol
p\|_{k+1}+\|u\|_{k+2}),
\end{equation}
where $ C_{r,1}$ is independent of both $h$ and $\rho$.
\end{corollary}

The above three theorems improve the results of \cite{wang2014weak}, where
the inf-sup condition for a given constant $\rho$ was proven.

\subsection{Divergence-based uniform inf-sup condition}

Next, we shall show the well-posedness of WG methods under another
pair of the parameter-dependent norms.  We choose the parameter $\eta$
as $\eta =\rho^{-1} h^{-1}_K$ in \eqref{WG_aform} and define the norms as
follows
\begin{equation}
\begin{aligned}
\|\tilde{\boldsymbol p}_h\|^2_{\widetilde{\rm div},\rho,h}
& =(c \boldsymbol{p}_h, \boldsymbol{p}_h)_{\mathcal{T}_h} +
({\rm div} \boldsymbol{p}_h, {\rm
div}\boldsymbol{p}_h)_{\mathcal{T}_h} + \rho^{-1} \sum_{K\in
\mathcal{T}_h} h^{-1}_K\langle(\boldsymbol{p}_h-\hat{\boldsymbol
p}_h)\cdot \boldsymbol{n}_K, (\boldsymbol{p}_h-\hat{\boldsymbol
p}_h)\cdot \boldsymbol{n}_K\rangle_{\partial K}, \\
\|u_h\|^2 &= (u_h,u_h)_{\mathcal T_h}.
\end{aligned}
\end{equation}
We have the uniform inf-sup condition for the following formulation 
\begin{equation}
A_w((\tilde{\boldsymbol p}_h, u_h),(\tilde{\boldsymbol q}_h,
v_h)) = a_w(\tilde{\boldsymbol p}_h, \tilde{\boldsymbol q}_h) +
b_w(\tilde{\boldsymbol q}_h,u_h) + b_w(\tilde{\boldsymbol p}_h,v_h).
\end{equation}

\begin{theorem}[\cite{hong2017uniformly, hong2018extended}] \label{uniform:WG:2}
Let $\boldsymbol R_h\subset \boldsymbol H({\rm div}, \Omega)\cap
\boldsymbol Q_h$ be the Raviart-Thomas finite element space. Assume
that $\{\!\!\{\boldsymbol R_h\}\!\!\}\subset \widehat Q_h$ and
$V_h = {\rm div}_h \boldsymbol Q_h$. Then, we have 
\begin{equation} \label{equ:WG-infsup-div}
 \inf_{(\tilde{\boldsymbol p}_h, u_h)\in \widetilde{\boldsymbol Q}_h\times V_h} 
 \sup_{(\tilde{\boldsymbol q}_h, v_h)\in \widetilde{\boldsymbol
Q}_h\times V_h}   \frac{A_w((\tilde{\boldsymbol p}_h,
u_h),(\tilde{\boldsymbol q}_h, v_h))}{(\|u_h\|+\|\tilde{\boldsymbol
p}_h\|_{\widetilde {\rm div},\rho,h})(\|v_h\|+\|\tilde{\boldsymbol
  q}_h\|_{\widetilde{\rm div},\rho,h})} \geq \beta_{w,1},
\end{equation}
where $\beta_{w,1}$ is independent of both $\rho$ and mesh size $h$.
\end{theorem}

\begin{corollary}[\cite{hong2017uniformly, hong2018extended}] \label{uniform_stable_WG}
Assume that the spaces $ \widetilde{\boldsymbol Q}_h\times V_h$
satisfy the conditions of Theorem \ref{uniform:WG:2}. Then there
exists a unique solution $(\tilde{\boldsymbol p}_h, u_h)\in
\widetilde{\boldsymbol Q}_h\times {V}_h$ satisfying
\eqref{Stabilizedpweak_WG} with $\eta=\rho^{-1}h_K^{-1}$, and for any
$0<\rho\leq 1$ the following estimate holds
\begin{equation}
\|\tilde{\boldsymbol p}_h\|_{\widetilde{\rm div},\rho,h} + \|u_h\|\leq
C_{w,2} \|f_h\|,
\end{equation}
where $C_{w,2}$ is a uniform constant with respect to both $\rho$ and
$h$.
\end{corollary}

\begin{theorem}[\cite{hong2017uniformly}]\label{Error:uniform_WG:mixed}
Let $(\boldsymbol p, u)$ be the solution of \eqref{HDIV} and assume
that $\boldsymbol p\in \boldsymbol H^1(\Omega)$. Let
$(\tilde{\boldsymbol p}_h, u_h)\in \widetilde {\boldsymbol Q}_h\times
V_h$ be the solution of \eqref{Stabilizedpweak_WG}. If we choose the
spaces $\widetilde {\boldsymbol Q}_h\times V_h$ such that the inf-sup
condition \eqref{equ:WG-infsup-div} is satisfied, then for any
$0<\rho\leq 1$ the following estimate holds  
\begin{equation}
 \|\boldsymbol p-\tilde{\boldsymbol p}_h\|_{\widetilde{\rm div},
\rho,h} + \|u-u_h\| \leq C_{e,2} \inf\limits_{\tilde{\boldsymbol
q}_h\in \widetilde{\boldsymbol Q}_h, v_h\in V_h } \left( \|\boldsymbol
p-\tilde{\boldsymbol q}_h\|_{\widetilde{\rm div}, \rho, h} + \|u-v_h\|
\right),
\end{equation}
where $C_{e,2} $ is a uniform constant with respect to both $\rho$ and
$h$.
\end{theorem}

\begin{corollary}[\cite{hong2017uniformly}] \label{WG-order:mixed}
Let $(\boldsymbol p,  u)$ be the solution of \eqref{HDIV} and assume
that $\boldsymbol p\in \boldsymbol H^{k+1}(\Omega), {\rm div}
\boldsymbol p\in H^{k+1}(\Omega), u\in H^{k+1}(\Omega)$. Let
$(\tilde{\boldsymbol p}_h, u_h)\in \widetilde{\boldsymbol Q}_h\times
V_h$ be the solution of \eqref{Stabilizedpweak_WG} with
$\eta=\rho^{-1} h^{-1}_K$. If we choose the spaces
$V_h\times\boldsymbol {Q}_h\times\widehat Q_h$ as
${V}_h^{k}\times\boldsymbol {Q}_h^{k,RT}\times\widehat Q_h^k$, then
for any $0<\rho\leq 1$ the following estimate holds
\begin{equation}
 \|\boldsymbol p-\tilde{\boldsymbol p}_h\|_{\widetilde{\rm div},
\rho,h}+\|u-{u}_h\|\leq C_{r,2} h^{k+1} (\|\boldsymbol
p\|_{k+1}+\|{\rm div}\boldsymbol p\|_{k+1}+\|u\|_{k+1}),
\end{equation}
where $ C_{r,2}$ is independent of both $h$ and $\rho$.
\end{corollary}

We show here some convergence results which are uniform with respect
to the stabilization parameter, which, again, improve the convergence
results in \cite{wang2014weak, chen2016weak}. These are mainly based
on the uniform inf-sup conditions we present in Theorems
\ref{inf-sup:bw} and \ref{uniform:WG:2}.

\section{Stability and convergence analysis of \comments{stabilized hybrid mixed (HDG) methods}}\label{sec:HDG}

The \comments{stabilized hybrid mixed (HDG) methods} read: Find $(\boldsymbol{p}_h,  \tilde{u}_h)\in
\boldsymbol Q_h\times \widetilde{V}_h$ such that for any
$(\boldsymbol{q}_h, \tilde{v}_h)\in \boldsymbol Q_h\times
\widetilde{V}_h$,
\begin{equation}\label{Stabilizedpweak}
\left\{
\begin{aligned}
a_h(\boldsymbol{p}_h,\boldsymbol{q}_h) + b_h(\boldsymbol{ q}_h,\tilde
{u}_h) &= 0, \\
b_h(\boldsymbol{p}_h, \tilde v_h) + c_h( \tilde{u}_h,
    \tilde{v}_h) & =-(f,v_h)_{\mathcal T_h}.
\end{aligned}
\right.
\end{equation}
Here, the bilinear forms are defined as follows
\begin{equation}\label{definition:ahbhch}
\begin{aligned}
a_h(\boldsymbol{p}_h,\boldsymbol{q}_h) &= (c \boldsymbol{p}_h,
    \boldsymbol{q}_h)_{\mathcal T_h},\\
b_h(\boldsymbol{q}_h, \tilde u_h) &= -(u_h, {\rm div}
    \boldsymbol{q}_h)_{\mathcal T_h}+\langle\hat{u}_h,
    \boldsymbol{q}_h\cdot\boldsymbol{n} \rangle_{\partial {\mathcal
      T_h}},\\
c_h( \tilde{u}_h, \tilde{v}_h) &= -\langle \tau(\hat
    P_h(u_h)-\hat{u}_h),  \hat P_h(v_h)-\hat{v}_h\rangle_{\partial
  {\mathcal T_h}},
\end{aligned}
\end{equation}
and $\tau > 0$ is the stabilization parameter.  If $\hat P_h$ is an
identity operator, then we obtain the standard HDG method. 
If $\hat P_h$ is a local $L^2$ projection, i.e. 
\begin{equation}\label{Local:projection}
\hat{P}_h|_{\partial K} := \hat P_{\partial K}:L^2(\partial
K)\rightarrow  \widehat Q(\partial K),
\end{equation}
where $\widehat Q(\partial K)$ is the trace space on $\partial K$, namely
$\widehat Q(\partial K)=\bigcup\limits_{e\in \partial K}
\widehat{Q}(e)$, we obtain the modified HDG methods with reduced
stabilization \cite{oikawa2015hybridized}.

The following lemma shows the consistency property of \comments{stabilized hybrid mixed (HDG)} methods.
\begin{lemma}\label{Error:common:HDG}
Let $f\in L^2(\Omega)$, and $(\boldsymbol p, u)$ be the solution of
\eqref{H1} or \eqref{HDIV}, then $(\boldsymbol p,  u)$ satisfies the
following consistency property
 \begin{equation}\label{Error:common:HDG:consistence}
\left\{
\begin{aligned}
a_h(\boldsymbol{p},\boldsymbol{q}_h)+b_h(\boldsymbol{ q}_h,
    u) &= 0 \qquad\qquad~~\quad \forall \boldsymbol{q}_h\in
\boldsymbol{Q}_h,\\
b_h(\boldsymbol{p}, \tilde v_h) + c_h( u,
    \tilde{v}_h) &= -(f,v_h)_{\mathcal T_h} \qquad \forall \tilde
v_h\in \widetilde V_h.
\end{aligned}
\right.
\end{equation}
\end{lemma}

\subsection{Divergence-based uniform inf-sup condition}

In this subsection, we will give the uniform inf-sup condition for
\eqref{definition:ahbhch} when $\tau=\rho h_K$ under the following
parameter-dependent norms
\begin{equation}\label{HDG:divnorm}
\begin{aligned}
\|\boldsymbol{p}_h\|_{{\rm div},\rho,h}^2 &= (c \boldsymbol{p}_h,
\boldsymbol{p}_h)_{\mathcal T_h} + ( {\rm div}\boldsymbol{p}_h, {\rm
div} \boldsymbol{p}_h)_{\mathcal T_h} + \rho^{-1}\sum\limits_{e\in
\mathcal{E}^i_h} h_e^{-1} \langle\hat P_e([\boldsymbol{p}_h]),\hat
P_e([\boldsymbol{p}_h])\rangle_e, \\
\|\tilde v_h\|_{0, \rho,h}^2 &= (v_h,v_h)_{\mathcal T_h}+\rho
\sum\limits_{e\in \mathcal{E}^i_h}h_e\langle\hat v_h,\hat
v_h\rangle_e.
\end{aligned}
\end{equation}
where $\hat P_e: L^2(e)\mapsto \widehat V(e)$ is the $L^2$ projection.

Using the parameter-dependent norms \eqref{HDG:divnorm}, we have
the following results, whose details will be reported in
\cite{hong2017uniformly, hong2018extended}.

\begin{theorem}\label{bounded:ahbhch}
For any $0<\rho\leq 1$, the boundedness of $a_h(\cdot,\cdot)$,
$b_h(\cdot,\cdot)$ and $c_h( \cdot, \cdot)$ is as follows
\cite{hong2017uniformly}
\begin{equation} \label{HDG;div-bound}
\begin{aligned}
|a_h(\boldsymbol{p}_h,\boldsymbol{q}_h)| & \leq
\|\boldsymbol{p}_h\|_{{\rm div}, \rho,h} \|\boldsymbol{q}_h\|_{{\rm
div}, \rho, h},\\
|b_h(\boldsymbol{q}_h, \tilde u_h)| & \leq \|\boldsymbol{q}_h\|_{{\rm
  div}, \rho,h} \|\tilde u_h\|_{0,\rho,h},\\
|c_h( \tilde{u}_h, \tilde{v}_h)| & \leq C \|\tilde u_h\|_{0,\rho,h}
\|\tilde v_h\|_{0,\rho,h},
\end{aligned}
\end{equation}
where $C$ is independent of both mesh size $h$ and $\rho$.
\end{theorem}

Denote 
$$
{\rm Ker}(B) := \{\boldsymbol q_h\in \boldsymbol Q_h: b_h(\boldsymbol
q_h,\tilde u_h)=0, \forall \tilde u_h\in \widetilde V_h\}.
$$
Then we have the coercivity of $a_h(\cdot,\cdot)$ on the ${\rm Ker}(B)$
and the inf-sup condition of $b_h(\cdot,\cdot)$ as follows. 

\begin{theorem}[\cite{hong2017uniformly, hong2018extended}] \label{coercivity:ah}
Assume that ${\rm div}_h \boldsymbol Q_h \subset V_h$. Then  
\begin{equation}
a_h(\boldsymbol{p}_h,\boldsymbol{p}_h)\geq
\|\boldsymbol{p}_h\|^2_{{\rm div}, \rho, h} \qquad \forall \boldsymbol{p}_h
\in {\rm Ker}(B).
\end{equation}
\end{theorem}

\begin{theorem}[\cite{hong2017uniformly, hong2018extended}] \label{uniform:inf-sup:bh}
For $k\ge 0$, if $\boldsymbol Q_h=\boldsymbol Q_h^{k+1},
V_h=V_h^k, \widehat V_h=\widehat V_h^r$ where $0\le r\le k+1$, or $\boldsymbol
Q_h=\boldsymbol Q_h^{k,RT}, V_h=V_h^k, \widehat V_h=\widehat V_h^r$ where
$0\le r\le k$, then we have 
\begin{equation}
 \inf_{\tilde u_h\in \widetilde V_h}\sup_{\boldsymbol{q}_h\in
\boldsymbol{Q}_h}   \frac{b_h(\boldsymbol q_h,\tilde
u_h)}{\|\boldsymbol q_h\|_{{\rm div},\rho,h}\|\tilde
u_h\|_{0,\rho,h}}    \geq \beta_{0},
\end{equation}
where $\beta_0$ is a constant independent of both $\rho$ and mesh size
$h$.
\end{theorem}

\begin{remark}
When $\tau = 0$, we can also have the stability result as Theorem
\ref{uniform:inf-sup:bh}, when choosing the following norms for any
$\tilde v_h\in \widetilde V_h$, and $\boldsymbol{p}_h\in
\boldsymbol{Q}_h$  
$$
\begin{aligned}
\|\boldsymbol{p}_h\|_{{\rm div},1,h}^2 &= (c \boldsymbol{p}_h,
\boldsymbol{p}_h)_{\mathcal T_h}+( {\rm div}\boldsymbol{p}_h, {\rm
div} \boldsymbol{p}_h)_{\mathcal T_h}+\sum\limits_{e\in
\mathcal{E}^i_h}h_e^{-1} \langle\hat
P_e([\boldsymbol{p}_h]),\hat P_e([\boldsymbol{p}_h])\rangle_e, \\
\|\tilde v_h\|_{0,1,h}^2 &= (v_h,v_h)_{\mathcal T_h}+\sum\limits_{e\in
  \mathcal{E}^i_h}h_e\langle\hat v_h,\hat v_h\rangle_e.
\end{aligned}
$$
\end{remark}

\begin{theorem} [\cite{hong2017uniformly}] \label{Error:uniform_HDG:mixed}
Let $(\boldsymbol p, u)$ be the solution for \eqref{HDIV} and
$(\boldsymbol{p}_h, \tilde{u}_h)\in \boldsymbol {Q}_h \times
\widetilde{V}_h$ be the solution for \eqref{Stabilizedpweak} with $\tau=\rho
h_K$. If we choose the spaces $\boldsymbol{Q}_h\times \widetilde{V}_h$
that satisfy the condition in Theorem \ref{uniform:inf-sup:bh}, then
for any $0<\rho\leq 1$ the following estimate holds
\begin{equation}
\|\boldsymbol p-\boldsymbol{p}_h\|_{{\rm div},\rho,h}+\|u-\tilde
u_h\|_{0,\rho,h}\leq C_{e,3}  \inf\limits_{\boldsymbol{q}_h\in
\boldsymbol{Q}_h, \tilde v_h\in \widetilde V_h } \left(\|\boldsymbol
p - \boldsymbol{q}_h\|_{{\rm div},\rho,h}+\|u-\tilde
v_h\|_{0,\rho,h}\right),
\end{equation}
where $C_{e,3} $ is a uniform constant with respect to both $\rho$ and
$h$.
\end{theorem}

\begin{corollary} [\cite{hong2017uniformly}]
\label{Error:uniform_HDG:mixed:rate}
Let $(\boldsymbol p,  u)$ be the solution of \eqref{HDIV} and
$\boldsymbol p\in \boldsymbol H^{k+1}(\Omega),{\rm div} \boldsymbol
p\in H^{k+1}(\Omega), u\in H^{k+1}(\Omega) $, and $(\boldsymbol{p}_h,
\tilde{u}_h)\in \boldsymbol {Q}_h\times \widetilde{V}_h$ be the
solution of \eqref{Stabilizedpweak} with $\tau=\rho h_K $. If we
choose the spaces $V_h\times\boldsymbol Q_h\times \hat V_h$ as
${V}_h^{k}\times\boldsymbol {Q}_h^{k,RT}\times \widehat V_h^k$, then
the following estimate holds 
\begin{equation}
 \|\boldsymbol p-\boldsymbol{ p}_h\|_{{\rm div},\rho,h}+\|u-\tilde
{u}_h\|_{0,\rho,h}\leq C_{r,3} h^{k+1} (\|\boldsymbol p\|_{k+1}+\|{\rm
div}\boldsymbol p\|_{k+1}+\|u\|_{k+1}),
\end{equation}
where $ C_{r,3}$ is independent of both $h$ and $\rho$.
\end{corollary}

\subsection{Gradient-based uniform inf-sup condition}

Next, we shall present the well-posedness of \comments{stabilized hybrid mixed (HDG)} methods under another
pair of parameter-dependent norms.  We choose
$\tau = \rho^{-1}h_K^{-1}$ in \eqref{definition:ahbhch} and define   
for any $\boldsymbol q_h\in \boldsymbol Q_h $ and $\tilde v_h \in
\widetilde V_h$
\begin{equation} \label{equ:HDG-norms-grad}
\begin{aligned}
\|\boldsymbol q_h\|^2 &= (c\boldsymbol q_h,\boldsymbol q_h)_{\mathcal
T_h} \\
\|\tilde v_h\|^2_{\tilde{1},\rho,h} &= (\nabla_h v_h,\nabla_h v_h)_{\mathcal
T_h}+\rho^{-1} \sum\limits_{K\in \mathcal{T}_h}h_K^{-1}\langle \hat
P_h(v_h)-\hat{v}_h, \hat P_h(v_h)-\hat{v}_h\rangle_{\partial K},
\end{aligned}
\end{equation}
where $\hat P_h$ is either an identity operator or a local projection as
illustrated in \eqref{Local:projection}. A straightforward calculation
shows that 
$$
\langle \hat P_h(v_h)-\hat{v}_h, \hat
P_h(v_h)-\hat{v}_h\rangle_{\partial {\mathcal T_h}}=2\langle\{\hat
P_h(v_h)-\hat{v}_h\}, \{\hat
P_h(v_h)-\hat{v}_h\}\rangle_{\mathcal{E}_h}+ \frac{1}{2}\langle
\lbrack\!\lbrack \hat P_h(v_h)- \hat{v}_h\rbrack\!\rbrack,
\lbrack\!\lbrack \hat
P_h(v_h)-\hat{v}_h\rbrack\!\rbrack\rangle_{\mathcal{E}_h}.
$$ 
Hence, if $\hat P_h$ is the identity operator, then $\|\tilde
v_h\|_{1,\rho,h}$ is indeed a norm on $\widetilde V_h$.
Moreover, if $\hat P_h$ is the local projection defined in
\eqref{Local:projection}, then $\|\tilde v_h\|_{\tilde{1},\rho,h}$ is
indeed a norm on $\widetilde V_h$ when $\widehat V_h^0\subset \widehat
V_h$, i.e. $\widehat V_h$ contains the piecewise constant space on
$\mathcal E_h$. 

We have the uniform inf-sup condition for the following formulation
\begin{equation}\label{compact:form:HDG}
A_h((\boldsymbol{p}_h, \tilde u_h),(\boldsymbol{q}_h, \tilde
v_h))=a_h(\boldsymbol{p}_h,\boldsymbol{q}_h)+b_h(\boldsymbol{
q}_h,\tilde {u}_h)+b_h(\boldsymbol{p}_h, \tilde v_h)+c_h(
\tilde{u}_h, \tilde{v}_h).
\end{equation}

\begin{theorem}[\cite{hong2017uniformly,hong2018extended}]
Assume that $\nabla_h V_h \subset \boldsymbol Q_h$. Then there exists a
positive constant $\rho_0$ which only depends on the shape regularity
of the mesh, such that for any $0<\rho \leq\rho_0$, we have
\begin{equation}
\inf_{(\boldsymbol{p}_h, \tilde u_h)\in \boldsymbol{Q}_h\times
\widetilde V_h} \sup_{\boldsymbol{q}_h, \tilde v_h\in
\boldsymbol{Q}_h\times \widetilde V_h}   \frac{A_h((\boldsymbol{p}_h,
\tilde u_h),(\boldsymbol{q}_h, \tilde v_h))}{(\|{\tilde
u}_h\|_{\tilde 1,\rho,h}+\|\boldsymbol{p}_h\|)(\|{\tilde
v}_h\|_{\tilde 1,\rho,h}+\|\boldsymbol{q}_h\|)}    \geq \beta_1,
\end{equation}
where $\beta_1$ is independent of both $\rho$ and mesh size $h$.
\end{theorem}

\begin{corollary} [\cite{hong2017uniformly, hong2018extended}] \label{uniform_stable}
Assume that $\nabla_h V_h \subset \boldsymbol Q_h$. Then there exists
a unique solution $(\boldsymbol{p}_h,  \tilde{u}_h)\in \boldsymbol
Q_h\times \widetilde{V}_h$ satisfying \eqref{Stabilizedpweak} with
$\tau=\rho^{-1}h_K^{-1}$. Further, there exists a positive constant
$\rho_0$ such that for any $0<\rho\leq \rho_0$ the following estimate
holds
\begin{equation}
\|\boldsymbol{p}_h\|+\|\tilde u_h\|_{\tilde{1},\rho,h}\leq C_{d,2}
\|f\|_{-\tilde{1},\rho,h},
\end{equation}
where $C_{d,2}$ is a uniform constant with respect to both $\rho$ and
$h$ and $ \|f\|_{-\tilde{1},\rho,h}=\sup\limits_{\tilde v_h \in \widetilde
  {V}_h}\frac{(f, v_h)_{\mathcal T_h}}{\|\tilde v_h\|_{\tilde{1},\rho,h}}$.
\end{corollary}

From the above corollary and the discrete Poincar{\'e}--Friedrichs
inequalities for piecewise $H^1$ functions \cite{brenner2003poincare},
i.e., $\|v_h\|\lesssim \|\nabla_h v_h\|+\sum\limits_{e\in
\mathcal{E}_h}h^{-1}_e\|\lbrack\!\lbrack v_h\rbrack\!\rbrack\|_{0,e}$,
we further have $\|\boldsymbol{p}_h\|+\|\tilde u_h\|_{\tilde 1,\rho,h}\leq
C_{d,3} \|f_h\|$.  The stability and consistency results of the \comments{stabilized hybrid mixed (HDG)}
methods lead to the following quasi-optimal approximation.

\begin{theorem}[\cite{hong2017uniformly}]\label{Error:uniform_ModiHDG:Primal}
Let $(\boldsymbol p,  u)$ be the solution of \eqref{H1} and
$(\boldsymbol{p}_h, \tilde{u}_h)\in \boldsymbol {Q}_h\times
\widetilde{V}_h$ be the solution of \eqref{Stabilizedpweak} with $\tau=\rho^{-1}
h^{-1}_K$. If we choose the spaces $\boldsymbol {Q}_h\times
\widetilde{V}_h$ such that $\nabla_h V_h\subset \boldsymbol Q_h$, then there
exists a constant $\rho_0$ such that for any $0<\rho\leq \rho_0$ the
following estimate holds
\begin{equation}
\|\boldsymbol p-\boldsymbol{p}_h\|+\|u-\tilde {u}_h\|_{\tilde 1,\rho,h}\leq
C_{e,4} \inf\limits_{\boldsymbol{q}_h\in \boldsymbol{Q}_h, \tilde
{v}_h\in \widetilde {V}_h } \left( \|\boldsymbol p-\boldsymbol{
q}_h\|+\|u-\tilde {v}_h\|_{\tilde 1,\rho,h}\right),
\end{equation}
where $C_{e,4} $ is a uniform constant with respect to both $\rho$ and $h$.
\end{theorem}

\begin{corollary}[\cite{hong2017uniformly}] \label{Error:uniform_HDG:primal:rate}
Let $(\boldsymbol p,  u)$ be the solution of \eqref{H1} and $\boldsymbol
p\in \boldsymbol H^{k+1}(\Omega), u\in H^{k+2}(\Omega) $,
$(\boldsymbol{p}_h,  \tilde{u}_h)\in \boldsymbol {Q}_h\times
\widetilde{V}_h$ be the solution of \eqref{Stabilizedpweak} with
$\tau=\rho^{-1}h_K^{-1}$, and $\hat P_h$ be an identity operator. If
we choose the spaces $V_h \times \boldsymbol Q_h\times \widehat V_h$
as $V^{k+1}_h\times\boldsymbol Q^k_h\times \hat V^{k+1}_h$, then the
following estimate holds 
\begin{equation}
\|\boldsymbol p-\boldsymbol{p}_h\|+\|u-\tilde {u}_h\|_{\tilde 1,\rho,h}\leq
C_{r,4} h^{k+1} (\|\boldsymbol p\|_{k+1}+\|u\|_{k+2}),
\end{equation}
where $ C_{r,4}$ is independent of both $h$ and $\rho$.
\end{corollary}

\begin{corollary}[\cite{hong2017uniformly}] \label{Error:uniform_ModiHDG:primal:rate}
Let $(\boldsymbol p, u)$ be the solution of \eqref{H1} and
$\boldsymbol p\in \boldsymbol H^{k+1}(\Omega), u\in H^{k+2}(\Omega)$,
$(\boldsymbol{\tilde p}_h,  {u}_h)\in \boldsymbol {Q}_h\times
\widetilde{V}_h$ be the solution of \eqref{Stabilizedpweak} with
$\tau=\rho^{-1} h^{-1}_K$ and $\hat P_h$ be a local $L^2$ projection
illustrated in \eqref{definition:ahbhch}. If we choose the spaces
$V_h\times\boldsymbol Q_h\times \widehat V_h$ as
$V^{k+1}_h\times\boldsymbol Q^k_h\times \widehat V^{k}_h$, then the
following estimate holds 
\begin{equation}
 \|\boldsymbol p-\boldsymbol{ p}_h\|+\|u-\tilde
{u}_h\|_{\tilde{1},\rho,h}\leq C_{r,5} h^{k+1} (\|\boldsymbol
p\|_{k+1}+\|u\|_{k+2}),
\end{equation}
where $ C_{r,5}$ is independent of both $h$ and $\rho$.
\end{corollary}

In \cite{cockburn2010projection}, Cockburn, Gopalakrishnan and Sayas
established the error analysis for \comments{stabilized hybrid mixed (HDG)} methods based on a carefully
designed projection operator. In this paper, we present several
uniform convergence results with respect to the stabilization
parameter. As a result, the constants in the error estimates $C_{r,i}
(i=1,2,\cdots,5)$ are independent of $\rho$.


\section{Discontinuous Galerkin methods}\label{sec:DG}
In recent years, DG methods have been applied
to the solution of various differential equations due to their flexibility in
constructing feasible local-shape function spaces and their advantage in
capturing non-smooth or oscillatory solutions effectively.  Instead of
using the Lagrange multiplier technique, a penalty term is added to the
bilinear form of the DG method to force the continuity (see
\cite{arnold2002unified, brezzi2000discontinuous, honguniformly,
  hong2016robust, hong2017parameter} and the references therein).
With the concept of DG-gradient and DG-divergence defined as in
Definition \ref{DG_derivative_global}, most of the DG methods for
approximating the elliptic problem can be written as

\begin{subequations} \label{WG_global}
\begin{empheq}[left=\empheqlbrace\,]{align}
c\vect{p}_h -\DGdiv^*  \tilde{u}_h &= 0  ~~\qquad\qquad\quad {\rm in} \ \Omega,
  \label{WG_global_1}\\
-\DGnabla^*  \tilde{\vect{p}}_h   &= f_h  ~\qquad\qquad\quad {\rm in} \ \Omega,
  \label{WG_global_2}\\
\hat{\vect{p}}_h &= \comments{\bar{\vect{p}}}(\vect{p}_h, u_h)  \qquad {\rm on} \ {\cal E}_h,
\label{define_flux_1} \\
\hat{u}_h &= \comments{\bar u}(\vect{p}_h, u_h)  \qquad {\rm on} \ {\cal
  E}_h,\label{define_flux_2}
\end{empheq}
\end{subequations}
where $\comments{\bar{\boldsymbol p}}$ and $\comments{\bar u}$ are the formulas for defining $\hat{\vect{p}}_h$
and $\hat{u}_h$ in the terms of $\vect{p}_h$ and $u_h$, respectively. A crucial
feature of DG methods is that $\hat{\vect{p}}_h$ and $\hat{u}_h$
are given explicitly in \eqref{define_flux_1} --
\eqref{define_flux_2}. A basic question is: How do we define
$\hat{\vect{p}}_h$ and $\hat{u}_h$ in order for the DG schemes to result in good
approximations of the original problems? In the DG schemes, the local
problems on each element $K$ are connected through the
$\hat{\vect{p}}_h\cdot \vect{n}_K$ and $\hat{u}_h$. Therefore,
in order for make the schemes to be good approximations, $\hat{\vect{p}}_h$ and
$\hat{u}_h$ should be single-valued on the element edges.
Recalling condition (iii) in Lemma \ref{lem:conditional-dual} when
$\DGnabla $ and $-\DGdiv $ are mutually dual, we see that
$\hat{\vect{p}}_h = \{\vect{p}_h\}$ and $\hat{u}_h = \{u_h\}$ are
natural choices. However, it is known that such choices cannot ensure
the stability of the DG schemes. Hence, penalty terms are used to force
the continuity of either $\vect{p}_h$ or $u_h$.  Consequently, in
general, to define the numerical traces, \eqref{define_flux_1} --
\eqref{define_flux_2} can be given as 
\begin{equation}\label{numerical_fluxes}
\left\{
\begin{aligned}
\hat{\vect{p}}_h &= \gamma\{\vect{p}_h\} + (1-\gamma) \{-\alpha \nabla_h
u_h\} - \vect{\beta}[\vect{p}_h] + \mu_1(\llbracket u_h \rrbracket)  \quad
{\rm on}\ \mathcal{E}_h, \\
\hat{u}_h &= \{u_h\} +\vect{\beta}\cdot \llbracket u_h
\rrbracket + \mu_2([\vect{p}_h]) ~~\quad\qquad\qquad\qquad\qquad~ {\rm
  on}\ \mathcal{E}_h^i,
\quad\quad  \hat{u}_h=0 \quad {\rm on}\ \mathcal{E}_h^\partial.
\end{aligned}
\right.
\end{equation}
Here, $\gamma$ and $\vect{\beta}$ are parameters
that we can choose, and $\mu_1(\llbracket u_h \rrbracket)$ and
$\mu_2([\vect{p}_h])$ are penalty terms.  The possible choices of
numerical fluxes in the literature are summarized in Table
\ref{tab:DG-traces}.  In \cite{bassi1997high,brezzi1999discontinuous},
$r_e : \boldsymbol L^2({\cal E}_h) \rightarrow \vect{Q}_h$ is a
lifting operator defined by
\begin{equation} \label{equ:lifting}
\int_\Omega r_e(\boldsymbol w)\cdot \vect{q}_h ~dx = -\int_{e}
\boldsymbol w \cdot \{\vect{q}_h\} ~ds \qquad  \forall \vect{q}_h \in
\vect{Q}_h. 
\end{equation}

\begin{table}[!htbp]
\centering
\begin{tabular}{|c|c|c|c|c|}
\hline
Method & $\gamma$ & $\vect{\beta}$ & 
$\mu_1(\llbracket u_h \rrbracket)$ & $\mu_2([\vect{p}_h])$ \\ \hline 
IP method \cite{douglas1976interior, wheeler1978elliptic} & $0$ &
$\vect{0}$ & $\eta_eh_e^{-1}\llbracket u_h \rrbracket$ &
0 \\ \hline
LDG method \cite{cockburn1998local} & $1$ & 
$\mathcal{O}(1)$ & $\eta_e h_e^{-1}
\llbracket u_h \rrbracket$ & $0$ \\ \hline
DG Method of Bassi et. al. \cite{bassi1997high} & $0$ & 
$\vect{0}$ & $\eta_e\{r_e(\llbracket u_h \rrbracket )\}$
& $ 0$ \\ \hline 
DG Method of Brezzi et. al. \cite{brezzi1999discontinuous} & $1$ &
$\vect{0}$ & $\eta_e\{r_e(\llbracket u_h \rrbracket )\}$
& $0$ \\ \hline 
Mixed DG method & $1$ & $\vect{0}$ & $0$ &
$\eta_e h_e^{-1}[\vect{p}_h]$ \\ \hline
\end{tabular}
\caption{DG methods: Numerical fluxes, $\eta_e = \mathcal{O}(1)$} 
\label{tab:DG-traces}
\end{table}

In the next two subsections, we introduce two classes of DG methods.
The first class of DG methods is used to approximate the 
 form \eqref{H1} so a penalty term $\mu_1(\llbracket u_h
\rrbracket)$ is needed to force the continuity of $u_h$, and we
name this class of DG methods primal DG methods. The second class of DG
methods, which is named mixed DG methods, is aimed to approximate the mixed
form of the elliptic problem. Hence, a penalty term
$\mu_2([\vect{p}_h])$ is added to force the normal continuity of
$\vect{p}_h$.

\subsection{Primal discontinuous Galerkin methods}

By Definition \ref{DG_derivative_global}, since $\hat{u}_h$
and $\hat{\vect{p}}_h$ are single-valued, we establish the following
relations using integration by parts and \eqref{equ:dg-identity_1}: 
\begin{align}
\langle-\DGdiv^*  \tilde{u}_h, \vect{q}_h\rangle
& = - (u_h, {\rm div} \vect{q}_h)_{\mathcal T_h}+
\langle \hat{u}_h, [\vect{q}_h]\rangle_{{\cal E}_h^i}
\label{weak_gradient_1}\\
& = (\nabla_h u_h, \vect{q})_{\mathcal T_h}- \langle
\llbracket u_h\rrbracket, \{\vect{q}_h\}\rangle_{{\cal E}_h} + \langle
\hat{u}_h-\{u_h\}, [\vect{q}_h]\rangle_{{\cal
  E}_h^i},\label{weak_gradient_2}\\
\langle-\DGnabla^*  \tilde{\vect{p}}_h, v_h\rangle
& =- (\vect{p}_h, \nabla_h v_h)_{\mathcal T_h} + \langle
\hat{\vect{p}}_h, \llbracket v_h\rrbracket\rangle_{{\cal
  E}_h}\label{weak_divergence_1} \\
& = ({\rm div} \vect{p}_h, v_h)_{\mathcal T_h} + \langle
\hat{\vect{p}}_h - \{\vect{p}_h\}, \llbracket v_h \rrbracket
\rangle_{\mathcal{E}_h} - \langle [\vect{p}_h], \{v_h\}
\rangle_{\mathcal{E}_h^i}. \label{weak_divergence_2}
\end{align}

Motivated by \eqref{weak_gradient_2} and \eqref{weak_divergence_1},
most of the existing primal DG methods can be written as
\eqref{WG_global_1} -- \eqref{WG_global_2} and
\eqref{numerical_fluxes} with specific choices of the parameters
$\gamma$, $\vect{\beta}$, and the penalty term $\mu_1(\llbracket u_h
\rrbracket)$. Examples of primal methods are the IP method \cite{douglas1976interior,
wheeler1978elliptic}, the LDG method \cite{cockburn1998local}, the method of
Bassi et. al. \cite{bassi1997high}, and the method of Brezzi et. al.
\cite{brezzi1999discontinuous} listed in Table \ref{tab:DG-traces}.
Note that all of these DG methods have a penalty term on $u_h$ so they
intend to approximate the solution of the primal form.  To put it
simply, if $\alpha$ is piecewise constant and $\nabla_h V_h \subset
\vect{Q}_h $, then we can eliminate $\vect{p}_h$ to obtain the DG
formulations with $u_h$ solely.  We refer to \cite{arnold2002unified}
for a detailed discussion of primal DG methods.

\begin{remark}
As a combination of continuous and discontinuous Galerkin methods, the
so-called enriched DG (EDG) methods \cite{becker2004reduced,
sun2009locally}, which are locally conservative, enrich the
approximation space of the continuous Galerkin methods with piecewise
constant functions. EDG methods adopt the same weak formulation as DG
methods, but have a smaller number of degrees of freedom than the DG
methods.
\end{remark}

\subsection{Mixed discontinuous Galerkin methods} 
In this subsection, we derive a new family of mixed DG methods, which
can be regarded as the dual form of primal DG methods. In the
literature, there are some existing works that discuss mixed DG methods
for elliptic problems \cite{brezzi2005mixed, burman2009local,
chen2010local}, but all of these schemes are aimed at approximating the
primal form \eqref{H1}.  Alternatively, the mixed DG methods we propose are
designated to approximate the mixed form \eqref{HDIV}.

Instead of penalizing $u_h$, we consider a penalty term for
$\vect{p}_h$ to obtain the mixed DG schemes.  Let us choose $\gamma =
1, \vect{\beta} = \vect{0}, \mu_1(\llbracket u_h \rrbracket) = 0$, and
$\mu_2([\vect{p}_h]) = \eta_e h_e^{-1}[\vect{p}_h]$ in
\eqref{numerical_fluxes}, i.e.,  
\begin{equation} \label{equ:traces-MLDG}
\left\{
\begin{aligned}
&\hat{\vect{p}}_h=  \{\vect{p}_h\} 
\qquad \qquad \qquad~ {\rm on}\ \mathcal{E}_h, \\
&\hat{u}_h=\{u_h\} + \eta_e h_e^{-1} [\vect{p}_h] \quad {\rm on}\
    \mathcal{E}_h^i,
\quad\quad  \hat{u}_h=0 \quad {\rm on}\ \mathcal{E}_h^\partial.
\end{aligned}
\right.
\end{equation}
We can see that this choice is the dual of the simplified LDG method
\cite{cockburn1998local} (when $\vect{\beta} = \vect{0}$) in the
sense that the definitions $\bar p$ and $\bar u$ in \eqref{define_flux_1} --
\eqref{define_flux_2} are exchanged in the two schemes. The numerical
scheme \eqref{WG_global_1} -- \eqref{define_flux_2} with such choices
can be written as the mixed DG formulation: Find $(\vect{p}_h,u_h)\in
\vect{Q}_h\times V_h$ such that
\begin{equation} \label{equ:mixed_DG}
\left\{
\begin{aligned}
a_h^{\rm MDG}( \vect{p}_h, \vect{q}_h) + b_h^{\rm MDG}(\vect{q}_h,
u_h) &= 0 & \forall \vect{q}_h\in \vect{Q}_h,\\
b_h^{\rm MDG}(\vect{p}_h, v_h) \qquad\qquad\qquad\quad &=
-\int_\Omega f \, v_h\, dx & \forall v_h\in V_h. 
\end{aligned}
\right.
\end{equation}
Here, we choose $\eta_e = \mathcal{O}(1)$, and define  
\begin{align}
a_h^{\rm MDG}( \vect{p}, \vect{q}) &= (c\vect{p},
\vect{q})_{\mathcal T_h} +  \langle \eta_e h_e^{-1} [\vect{p}],
[\vect{q}] \rangle_{\mathcal E_h^i} \qquad \forall \vect{p},
\vect{q} \in \vect{Q}_h \cup \boldsymbol H({\rm div};\Omega),
\label{equ:MDG_a}\\
b_h^{\rm MDG}(\vect{p}, v) &= - ({\rm div}_h \vect{p},
v)_{\mathcal T_h} + \langle [\vect{p}], \{v\} \rangle_{\mathcal
E_h^i} \qquad \forall \vect{p}\in \vect{Q}_h \cup \boldsymbol H({\rm
div};\Omega), \forall v \in V_h \cup H^1(\Omega).\label{equ:MDG_b}
\end{align}

\begin{remark} \label{rmk:MDG-Brezzi}
With the choice of the numerical traces: $\gamma = 1$, $\vect{\beta} =
\vect{0}$, $\mu_1(\llbracket u_h \rrbracket) = 0$ and
$\mu_2([\vect{p}_h]) = \eta_e\{r_e([\vect{p}_h])\}$, we can obtain
another mixed DG scheme, which is the dual form of the method of
Brezzi et al. \cite{brezzi1999discontinuous}. 
Here, the lifting operator $r_e: L^2({\cal E}_h) \mapsto V_h$ is
defined by
\begin{equation} \label{equ:lifting}
\int_\Omega r_e(w) v_h ~dx = -\int_{e}
w \{v_h\} ~ds \qquad  \forall v_h \in V_h. 
\end{equation}
We are also aware that if $\gamma = 0$ or $\vect{\beta} \neq \vect{0}$,
the resulting mixed DG schemes are not symmetric.
\end{remark}

Next, we prove the well-posedness of the mixed
DG formulation \eqref{equ:mixed_DG} when choosing
\begin{equation} \label{equ:spaces-Poisson}
\begin{aligned}
V_h = V_h^k,~~~
\vect{Q}_h = \vect{Q}_h^{k+1},
\end{aligned}
\end{equation}
for $k \geq 0$, which leads to the optimal order of convergence in the
$L^2$ norm $\|\cdot\|$ for $u$ and the following norm for
$\boldsymbol{p}$:
\begin{equation} \label{equ:mixed_DG_space} 
\|\vect{q}\|_{{\rm MDG},h}^2 := (c \vect{q}, \vect{q})_{\mathcal
  T_h} + ({\rm div}_h\vect{q}, {\rm div}_h \vect{q} )_{\mathcal T_h}
  + \langle \eta_e h_e^{-1} [\vect{q}],
  [\vect{q}]\rangle_{\mathcal{E}_h^i} \qquad \forall \vect{q} \in
  \vect{Q}_h \cup \boldsymbol H({\rm div};\Omega).
\end{equation}

\paragraph{Boundedness.}
A direct calculation shows that $a_h^{\rm MDG}(\cdot, \cdot)$ satisfies 
\begin{equation} \label{equ:boundedness-a}
a_h^{\rm MDG}(\vect{p}, \vect{q}) \leq \|\vect{p}\|_{{\rm MDG}, h}
\|\vect{q}\|_{{\rm MDG}, h} \qquad \forall \vect{p},
\vect{q} \in \vect{Q}_h \cup \boldsymbol H({\rm div};\Omega).
\end{equation}
From the estimate of lifting operator (see also
\cite{arnold2002unified, brezzi2000discontinuous})
\begin{equation} \label{equ:lifting-boundedness}
\|r_e(w)\|_{0,\Omega} \lesssim  h_e^{-1/2} \|w\|_{0,e},
\end{equation}
we have the boundedness of $b_h^{\rm MDG}(\cdot, \cdot)$.

\begin{lemma} \label{lem:boundedness-b}
It holds that 
\begin{align}
b_h^{\rm MDG}(\vect{q}, v_h) &\lesssim \|\vect{q}\|_{{\rm MDG},h}
\|v_h\|_{0} & \forall \vect{q} \in \vect{Q}_h \cup
\boldsymbol H({\rm div};\Omega),\; \forall v_h \in V_h,
\label{equ:boundedness-b1}\\
b_h^{\rm MDG}(\vect{q}, v) &\lesssim \|\vect{q}\|_{{\rm MDG},h}
(\|v\|_{0} + h|v|_{1,h}) &\forall \vect{q} \in
\vect{Q}_h \cup \boldsymbol H({\rm div};\Omega),\; \forall v \in
H^1(\Omega).
\label{equ:boundedness-b2}
\end{align}
\end{lemma}
  
\paragraph{Stability.}  
According to the theory of mixed methods, the stability of the
saddle point problem \eqref{equ:mixed_DG} is the corollary of the
following two conditions \cite{brezzi1974existence,brezzi1991mixed}:
\begin{enumerate}
\item K-ellipticity: 
\begin{equation} \label{equ:K-ellipticity}
a_h^{\rm MDG}(\vect{q}_h, \vect{q}_h) \gtrsim
\|\vect{q}_h\|_{{\rm MDG},h}^2 \qquad \forall \vect{q}_h \in Z_h,
\end{equation}
where $Z_h = \{\vect{q}_h \in \vect Q_h~|~ b_h^{\rm
MDG}(\vect{q}_h, v_h) = 0, \forall v_h \in V_h\}$.
\item The discrete inf-sup condition:  
\begin{equation} \label{equ:inf-sup}
\inf_{v_h \in V_h} \sup_{\vect{q}_h \in \vect{Q}_h}
\frac{b_h^{\rm MDG}(\vect{q}_h, v_h)}{\|\vect{q}_h\|_{{\rm MDG}, h}
  \|v_h\|} \gtrsim 1.
\end{equation}
\end{enumerate}

\begin{theorem} \label{thm:well-posedness}
The mixed DG schemes \eqref{equ:mixed_DG} are well-posed for
$(\vect{Q}_h^{k+1}, \|\cdot\|_{{\rm MDG},h})$ and $(V_h^k, \|\cdot\|)$. 
\end{theorem}
\begin{proof}
We first show the K-ellipticity \eqref{equ:K-ellipticity}. By the
definition of the lifting operator \eqref{equ:lifting}, we have 
$$ 
b_h^{\rm MDG}(\vect{q}_h, v_h) = \int_\Omega \left( {\rm
div}_h\vect{q}_h + \sum_{e\in{\cal E}_h^i} r_e([\vect{q}_h])
\right) v_h ~dx, 
$$
which implies that  
$$
Z_h = \{ \vect{q}_h\in \vect Q_h^{k+1}~|~ {\rm div}_h\vect{q}_h
+ \sum_{e\in{\cal E}_h^i} r_e([\vect{q}_h]) = 0\}.
$$
Let $\eta_0=\inf_{e\in{\cal E}_h^i} \eta_e$ be a positive constant
that is independent of the grid size. Then \eqref{equ:lifting-boundedness} implies  
\begin{equation}
a(\vect{q}_h, \vect{q}_h) \geq \|\vect{q}_h\|_{0,\Omega}^2 +
\eta_0 \sum_{e\in{\cal E}_h^i}  h_e^{-1} \|[\vect{q}_h]\|_{0,e}^2
\gtrsim \|\vect{q}_h\|_{{\rm MDG},h}^2 \qquad \forall \vect{q}_h
\in Z_h. 
\end{equation}
The inf-sup condition \eqref{equ:inf-sup} follows from the inf-sup
condition for the BDM element. 
\end{proof}

\begin{remark}
A similar argument shows that the penalty term
$\langle \eta_e h_e^{-1} [\vect{p}_h], [\vect{q}_h] \rangle_{\mathcal
E_h^i}$ can be replaced by $
\langle \eta_e r_e([\vect{\sigma}_h]) \cdot r_e([\vect{q}_h])
\rangle_{\mathcal E_h^i}$, and the well-posedness of the corresponding
scheme can be proved similarly with the modified norm
$ \|\vect{q}\|_{*,h}^2 = (c \vect{q}, \vect{q})_{\mathcal T_h} +
({\rm div}_h\vect{q}, {\rm div}_h\vect{q} )_{\mathcal T_h} + \langle
\eta_e r_e([\vect{q}]), r_e([\vect{q}])\rangle_{\mathcal{E}_h^i}
$.
\end{remark}

\begin{lemma}\label{lem:consistency}
Assume that the solution $(\vect{q},u)\in \boldsymbol H({\rm
div};\Omega) \times H^1(\Omega)$. Then we have
\begin{equation} \label{equ:consistency}
\left\{
\begin{aligned}
a_h^{\rm MDG}(\vect{p} - \vect{p}_h, \vect{q}_h) + 
b_h^{\rm MDG}(\vect{q}_h, u -u_h) & = 0  \qquad \forall \vect{q}_h\in
\vect{Q}_h, \\
b_h^{\rm MDG}(\vect{p} - \vect{p}_h, v_h)
\qquad\qquad\qquad\qquad\quad &= 0
\qquad  \forall v_h\in V_h. 
\end{aligned}
\right.
\end{equation}
\end{lemma}

By combining Lemma \ref{lem:consistency} and the well-posedness of the
mixed DG formulation \eqref{equ:mixed_DG}, we have the following a
priori error estimates.

\begin{theorem} \label{thm:priori}
Let $(\vect{p}_h, u_h) \in \boldsymbol Q_h^{k+1}\times V_h^k$ be the
solution for the mixed DG formulation \eqref{equ:mixed_DG}, and
$(\vect{p}, u) \in \boldsymbol H({\rm div};\Omega) \times H^1(\Omega)$
be the solution for \eqref{HDIV}.  Then we have 
\begin{equation} \label{equ:cea}
\|\vect{p} - \vect{p}_h\|_{{\rm MDG},h} + \|u - u_h\| \lesssim
\inf_{\vect{p}_h \in \vect{Q}_h^{k+1}} \|\vect{p} - \vect{p}_h\|_{{\rm
MDG},h} + \inf_{v_h \in V_h^k} (\|u-v_h\|_{0} + h|u-v_h|_{1,h}). 
\end{equation}
\end{theorem}

Using the Scott-Zhang interpolation \cite{scott1990finite}, we
have the following theorem.

\begin{theorem} \label{thm:error-estimate} 
Let $(\vect{p}_h, u_h) \in \boldsymbol Q_h^{k+1}\times V_h^k$ be the
solution of the mixed DG formulation \eqref{equ:mixed_DG}.  Assume
that the solution of \eqref{HDIV} satisfies $(\vect{p},u)\in
\boldsymbol H^{k+2}(\Omega) \times H^{k+1}(\Omega)$.  Then we have  
\begin{equation} \label{equ:error-estimate}
\|\vect{p} - \vect{p}_h\|_{{\rm MDG},h} + \|u - u_h\|_{0} \lesssim
h^{k+1}(|\vect{p}|_{k+2} + |u|_{k+1}).
\end{equation}
\end{theorem}

\begin{remark}
$\vect{Q}_h$ can be chosen as a discontinuous RT finite element, i.e.,
  $\vect{Q}_h=\vect{Q}_h^{k, RT}$,
and the corresponding well-posedness and error estimates can also be
obtained similarly.
\end{remark}

\begin{remark}
The mixed DG method for linear elasticity and the proof of the well-posedness
can also be provided due to the stability analysis in
\cite{gong2015mixed}, which shows that optimal convergence rates are
achieved for both stress and displacement variables. 
\end{remark}

\paragraph{Numerical examples of mixed DG methods}
We next illustrate the performance of mixed DG methods for the Poisson problem in
2D. The problem is computed on the unit
square $\Omega = [0,1]^2$ with a homogeneous boundary condition that
$u = 0$ on $\partial \Omega$. The coefficient matrix $\alpha =
\vect{I}_2$, where $\vect{I}_2 \in \mathbb{R}^{2\times2}$ is the
identity matrix.  The exact solution satisfies  
\begin{equation} \label{equ:numerical-example}
u = \sin(2\pi x)\sin(\pi y) \qquad  {\rm and} \qquad
\vect{p} = \begin{pmatrix}
2\pi \cos(2\pi x)\sin(\pi y) \\
\pi \sin(2\pi x)\cos(\pi y)
\end{pmatrix}.
\end{equation} 
The exact load function $f$ can be analytically derived for the given
$u$. Non-nested, quasi-uniform unstructured grids with different
grid sizes are used in the computation. The parameter $\eta_e$ is set
to be $1$. For the unstructured grid, we define $h = N_{ele}^{-1/d}$,
where $N_{ele}$ is the number of elements. From Table
\ref{tab:Poisson-unstructured}, the optimal convergence can be
observed. Moreover, we observe that the $L^2$ error of $\vect{p}$ is
of order $k+2$, which is one order higher than the error estimate
\eqref{equ:error-estimate}. 

\begin{table}[!htbp] 
\centering
\begin{subtable}{\textwidth} 
\caption{Poisson problem: $\mathcal{P}_1^{-1} - \mathcal{P}_0^{-1}$,
  unstructured grids}
\centering
\begin{tabular}{c|cc|cc|cc}
\hline
$N_{ele}$	& $\|u -u_h\|_{0,\Omega}$	&$h^n$	&$\|\vect{p} -
\vect{p}_h\|_{0,\Omega}$ & $h^n$ &$\|{\rm div}_h(\vect{p} -
\vect{p}_h)\|_{0,\Omega}$	&$h^n$ \\ \hline
220 & 0.0772588 & -- & 0.111405 & -- & 3.79529 & -- \\ \hline
976 & 0.0369387 & 0.99 & 0.0255594 & 1.98 & 1.82077 & 0.99 \\ \hline
4054 & 0.0180824 & 1.00 & 0.006156 & 2.00 &  0.89211 & 1.00 \\ \hline 
\end{tabular}
\label{tab:Poisson-P1-P0-unstructured}
\end{subtable}

\vspace{1mm}

\begin{subtable}{\textwidth}
\centering
\caption{Poisson problem: $\mathcal{P}_2^{-1} - \mathcal{P}_1^{-1}$,
  unstructured grids}
\begin{tabular}{c|cc|cc|cc}
  \hline
$N_{ele}$	& $\|u -u_h\|_{0,\Omega}$	&$h^n$	&$\|\vect{p} -
\vect{p}_h\|_{0,\Omega}$ & $h^n$ &$\|{\rm div}_h(\vect{p} -
\vect{p}_h)\|_{0,\Omega}$	&$h^n$ \\ \hline
220 & 0.0066776 & -- & 0.00557092  & -- & 0.332017 & -- \\ \hline 
976 & 0.00146501 & 2.04 & 0.000581936 & 3.03 & 0.0726511 & 2.04\\ \hline
4054 & 0.000357469 & 1.98 & 6.98718e-05 & 2.98 & 0.0177161 & 1.98\\ \hline 
\end{tabular}
\label{tab:Poisson-P2-P1-unstructured}
\end{subtable}

\vspace{1mm}

\begin{subtable}{\textwidth}
\centering 
\caption{Poisson problem: $\mathcal{P}_3^{-1} - \mathcal{P}_2^{-1}$,
  unstructured grids}
\begin{tabular}{c|cc|cc|cc}
  \hline
$N_{ele}$	& $\|u -u_h\|_{0,\Omega}$	&$h^n$	&$\|\vect{p} -
\vect{p}_h\|_{0,\Omega}$ & $h^n$ &$\|{\rm div}_h(\vect{p} -
\vect{p}_h)\|_{0,\Omega}$	&$h^n$ \\ \hline
220 & 0.000389143 & -- & 0.00022748 & -- & 0.0193726 & -- \\ \hline 
976 & 4.20827e-05 & 2.99 & 1.16707e-05 & 3.99 & 0.00209091 &  2.99 \\ \hline 
4054 & 4.95135e-06 &  3.01 & 6.79925e-07 & 3.99 & 0.000245922 & 3.01 \\ \hline 
\end{tabular}
\label{tab:Poisson-P3-P2-unstructured}
\end{subtable}
\caption{Poisson problem: the convergence order on 2D non-nested unstructured grids}
\label{tab:Poisson-unstructured}
\end{table}

\section{Relationship between different methods}\label{relationship}

In this section, we shall discuss the relationship between different
methods.   

\subsection{From stabilized hybrid methods to the LDG methods}
Let us show that the stabilized hybrid mixed methods can deduce the
LDG scheme if we set $\hat{u}_h = \{u_h\} + \vect{\beta}\cdot
\llbracket u_h \rrbracket$ and $\hat{v}_h = \{v_h\} +
\vect{\beta}\cdot \llbracket v_h \rrbracket$.  First, we can see that
the first equation for the \comments{stabilized hybrid mixed (HDG)} methods \eqref{eq:hybridgalerkin} is
equivalent to \eqref{WG_global_1} formally. Further, the left hand of
the second equation of \eqref{eq:hybridgalerkin} is
\begin{equation} \label{equ:HDG2LDG}
\begin{aligned}
& \quad -({\rm div}\vect{p}_h, v_h)_{\mathcal T_h} + \langle
\vect{p}_h\cdot \vect{n}, \hat{v}_h\rangle_{\partial{\mathcal T_h}} -
\langle \tau (u_h - \hat{u}_h), v_h -\hat{v}_h\rangle_{\partial
{\mathcal T_h}}\\
& = (\vect{p}_h, \nabla v_h)_{\mathcal T_h} +
\langle \vect{p}_h\cdot \vect{n}, \hat{v}_h - v_h \rangle_{\partial
{\mathcal T_h}} - \langle \tau (u_h - \hat{u}_h), v_h
-\hat{v}_h\rangle_{\partial {\mathcal T_h}}\\
& = (\vect{p}_h, \nabla v_h)_{\mathcal T_h} + \langle \{\vect{p}_h\},
  \llbracket \hat{v}_h - v_h\rrbracket \rangle_{{\cal E}_h} + \langle
  [\vect{p}_h], \{\hat{v}_h - v_h\} \rangle_{{\cal E}_h^i}
  \qquad \qquad\qquad\qquad\qquad\qquad (\text{by } \eqref{eq:7})\\ 
 &~~~ -2 \langle \tau \{u_h - \hat{u}_h\}, \{v_h
 -\hat{v}_h\}\rangle_{{\cal E}_h} - \frac{1}{2} \langle \tau
 \llbracket u_h - \hat{u}_h\rrbracket, \llbracket v_h -\hat{v}_h
 \rrbracket\rangle_{{\cal E}_h} \\
& = (\vect{p}_h, \nabla v_h)_{\mathcal T_h} - \langle
\{\vect{p}_h\}, \llbracket v_h\rrbracket \rangle_{{\cal E}_h} 
+ \langle[\vect{p}_h], \vect{\beta}\cdot \llbracket v_h \rrbracket
\rangle_{{\mathcal{E}}_h}
-2 \langle \tau\vect{\beta}\cdot \llbracket u_h \rrbracket,
  \vect{\beta} \cdot \llbracket v_h \rrbracket \rangle_{{\cal E}_h}
- \frac{1}{2} \langle \tau \llbracket u_h \rrbracket, \llbracket
v_h \rrbracket\rangle_{{\cal E}_h} \\
& = (\vect{p}_h, \nabla v_h)_{\mathcal T_h} - \langle \{\vect{p}_h\} -
\vect{\beta}[\vect{p}_h] + \frac{1}{2}\tau (1+4|\vect{\beta}\cdot
\vect{n}_e|^2) \llbracket u_h \rrbracket, \llbracket v_h\rrbracket
\rangle_{{\cal E}_h},
\end{aligned}
\end{equation}
which is same as \eqref{weak_divergence_1} under $\hat{\vect{p}}_h =
\{\vect{p}_h\} - \vect{\beta}[\vect{p}_h] +
\frac{1}{2}\tau(1+4|\vect{\beta}\cdot \vect{n}_e|^2) \llbracket
u_h \rrbracket $. This is exactly the definition of the numerical trace
$\hat{\vect{p}}_h$ of the primal LDG methods (see the LDG methods in Table
\ref{tab:DG-traces}) when $\tau = \mathcal{O}(h^{-1})$.  Therefore,
primal LDG methods can be formally deduced from \comments{stabilized hybrid mixed (HDG)} methods when choosing the
space
\begin{equation} \label{equ:HDG2LDG-space} 
\widetilde{V}_h = \left\{ 
(v_h, \hat{v}_h):~v_h \in V_h, \hat{v}_h =
\{v_h\} + \vect{\beta}\cdot \llbracket v_h \rrbracket 
\right\}. 
\end{equation}

On the other hand, when considering the DG scheme \eqref{WG_global} with
$\hat{\vect{p}}_h = \{\vect{p}_h\} + \eta_e h_e^{-1}\llbracket u_h
\rrbracket$, we obtain the following from \eqref{equ:HDG2LDG} that 
$$ 
\begin{aligned}
(\DGnabla^* \tilde{\vect{p}}_h, v_h) & =  
 (\vect{p}_h, \nabla v_h)_{\mathcal T_h} - \langle v_h,
 \hat{\vect{p}}_h \cdot \vect{n} \rangle_{\partial {\mathcal T_h}} \\
& =   (\vect{p}_h, \nabla v_h)_{\mathcal T_h} - \langle \{\vect{p}_h\}
+ \eta_e h_e^{-1} \llbracket u_h \rrbracket, \llbracket v_h
\rrbracket\rangle_{{\cal E}_h} \\
& = - ({\rm div} \vect{p}_h,  v_h)_{\mathcal T_h} + \langle
\vect{p}_h\cdot \vect{n}, \{v_h\} \rangle_{\partial {\mathcal T_h}} -
\sum_{K\in \mathcal{T}_h} 2\eta_e h_e^{-1} \langle u_h - \{u_h\}, v_h
-\{v_h\} \rangle_{\partial K} \\
& = -\big(\DGdiv \vect{p}_h, (v_h, \{v_h\})\big) - \sum_{K\in
\mathcal{T}_h} 2\eta_e h_e^{-1} \langle u_h - \{u_h\}, v_h -\{v_h\}
\rangle_{\partial K}.
\end{aligned}
$$ 
Hence, the (simplified) primal LDG will formally return to
the stabilized hybrid mixed methods \eqref{eq:hybridgalerkin} by
replacing $\{u_h\}$ and $\{v_h\}$ with the new trial variable
$\hat{u}_h$ and test variable $\hat{v}_h$, respectively.  In the same
way, it is readily seen that  mixed LDG methods \eqref{equ:traces-MLDG} can be
formally deduced from \comments{stabilized hybrid primal (WG)} methods
\eqref{eq:WGprimal} if the space is specified as 
\begin{equation} \label{equ:WG2MLDG-space} 
\widetilde{\vect{Q}}_h = \{(\vect{q}_h, \hat{\vect{q}}_h):~
\vect{q}_h \in \vect{Q}_h, \hat{\vect{q}}_h = \{\vect{q}_h\} \}. 
\end{equation}
And the mixed LDG methods \eqref{equ:traces-MLDG} can also formally return to
the WG methods \eqref{eq:WGprimal} by replacing $\{\vect{p}_h\}$ and
$\{\vect{q}_h\}$ with the new trial variable $\hat{\vect{p}}_h$ and test
variable $\hat{\vect{q}}_h$, respectively. 

\begin{remark}
In order to derive the other DG schemes in a similar fashion, we need to
introduce another stabilization to the hybrid method. For instance,
instead of $\mathcal{S}_u^\tau$, we can introduce another non-symmetric
stabilization term in the stabilized hybrid primal methods
\eqref{HDG:Stabilizer}, i.e.,
$$ 
\langle \tau (\{r_e(u_h\vect{n})\}\cdot \vect{n} - \hat{u}_h), v_h -
\hat{v}_h \rangle_{\partial \mathcal{T}_h}.
$$ 
In light of \eqref{equ:HDG2LDG}, when choosing the special space with
$\vect{\beta} = \vect{0}$, we obtain 
\begin{equation} \label{equ:HDG2Brezzi}
\begin{aligned}
& \quad - ({\rm div} \vect{p}_h,  v_h)_{\mathcal T_h} + \langle
\vect{p}_h\cdot \vect{n}, \hat{v}_h \rangle_{\partial \mathcal{T}_h} - 
\langle \tau (\{r_e(u_h\vect{n})\} \cdot \vect{n} -
\hat{u}_h), v_h - \hat{v}_h \rangle_{\partial \mathcal{T}_h} \\
& = (\vect{p}_h, \nabla v_h)_{\mathcal T_h} + \langle \{\vect{p}_h\},
  \llbracket \hat{v}_h - v_h\rrbracket \rangle_{{\cal E}_h} + \langle
  [\vect{p}_h], \{\hat{v}_h - v_h\} \rangle_{{\cal E}_h^i} 
 - \frac{1}{2} \langle \tau \llbracket r_e(u_h\vect{n}_e)\cdot
 \vect{n}_e - \hat{u}_h\rrbracket,
 \llbracket v_h -\hat{v}_h
 \rrbracket\rangle_{{\cal E}_h} \\
& = (\vect{p}_h, \nabla v_h)_{\mathcal T_h} - \langle
\{\vect{p}_h\} +
\frac{1}{2}\tau\{r_e(\llbracket
u_h \rrbracket)\}, \llbracket
v_h\rrbracket \rangle_{{\cal E}_h},
\end{aligned}
\end{equation}
which gives rise to the $\hat{\vect{p}}_h$ by Brezzi et al.
\cite{brezzi1999discontinuous} when $\tau = \mathcal{O}(1)$ (see the DG
Method of Brezzi et. al. in Table \ref{tab:DG-traces}).
Similarly, the following non-symmetric stabilization term can be
adopted for \eqref{WG:Stabilizer} in place of
$\mathcal{S}_{\vect{p}}^\eta$
$$ 
\langle \eta\left(\{r_e(\vect{p}_h)\} -
\hat{\vect{p}}_h \right) \cdot \vect{n} , (\vect{q}_h -
\hat{\vect{q}}_h) \cdot \vect{n} \rangle_{\partial {\mathcal T_h}}. 
$$ 
Then the mixed DG method of Brezzi et. al (see Remark
\ref{rmk:MDG-Brezzi}) can be derived when $\eta = \mathcal{O}(1)$. 
\end{remark}

\subsection{Mixed methods as the limiting case of WG methods}
For a given mesh, we will now try to prove the convergence of WG methods
\eqref{eq:WGprimal} to mixed methods \eqref{eq:mixedmethods} when
$\rho \to 0$, where the stabilization parameter is set as $\eta =
\rho^{-1}h_K^{-1}$.

Consider the $\boldsymbol H({\rm div})$ conforming subspace
$\boldsymbol Q^c_h := \boldsymbol Q_h\cap \boldsymbol H({\rm div},
\Omega)\subset \boldsymbol Q_h$, the mixed methods
\eqref{eq:mixedmethods} in variational form are written as: Find
$(\boldsymbol p^c_h, u^c_h) \in \boldsymbol Q^c_h \times V_h$ such
that  
\begin{equation}\label{mixedmethod}
\left\{
\begin{aligned}
(c\boldsymbol p^c_h,\boldsymbol q^c_h)_{\mathcal T_h} - (u^c_h,{\rm
div} \boldsymbol q^c_h)_{\mathcal T_h} &= (g_1, \boldsymbol
  q_h^c)_{\mathcal T_h}
~\qquad \qquad \qquad \quad \forall \boldsymbol q^c_h\in \boldsymbol Q^c_h,\\
 ({\rm div} \boldsymbol p^c_h, v_h)_{\mathcal T_h}&= (f,v_h)_{\mathcal
   T_h} + \langle g_2, v_h \rangle_{\partial \mathcal T_h} \qquad
   \forall v_h\in V_h,
\end{aligned}
\right.
\end{equation}
where $g_1 = 0$ and $g_2 = 0$ when applied to the Poisson equation
\eqref{HDIV}.  Then, by $V_h\subset {\rm div} \boldsymbol Q^c_h\subset
{\rm div}_h \boldsymbol Q_h\subset V_h$, the well-posedness of the
mixed methods (cf. \cite{brezzi1991mixed, boffi2013mixed}) implies
that  
\begin{equation} \label{equ:well-posedness-mixed}
\|\boldsymbol p_h^c\|_{\boldsymbol H({\rm div})} + \|v_h^c\| \leq C_M
\left( 
\|f\| + \sup_{\boldsymbol q_h^c \in \boldsymbol Q_h^c} \frac{(g_1,
  \boldsymbol q_h^c)_{\mathcal T_h}}{\|\boldsymbol
q_h^c\|_{\boldsymbol H({\rm div})}} + \sup_{\boldsymbol v_h \in V_h}
\frac{\langle g_2, v_h \rangle_{\partial \mathcal T_h}}{\|v_h\|}
\right).
\end{equation}

Recall that the spaces defined on $\mathcal E_h$ (see
\eqref{Edge:spaces}) of WG methods are given by 
$$ 
\widehat{\boldsymbol Q}_h = \{\hat{\boldsymbol p}_h: \hat{\boldsymbol
p}_h|_e \in \widehat{Q}(e)\vect{n}_e, \forall e\in \mathcal{E}_h\},
  \qquad \widehat{Q}_h = \{\hat{p}_h: \hat{p}_h|_{e}, e\in
  \widehat{Q}(e), \forall e\in \mathcal{E}_h\}.
$$ 
We make the following assumption on the finite element spaces of WG
methods.

\begin{assumption}\label{Ass:WG:mixed}
Assume that the spaces $\boldsymbol{Q}_h$, $\widehat{\boldsymbol Q}_h$
and $V_h$ satisfy
\begin{enumerate}
\item ${\rm div}_h \boldsymbol{Q}_h = V_h$;
\item $\{\!\!\{\boldsymbol{Q}_h\}\!\!\}|_e \subset \widehat Q(e),
~\forall e \in \mathcal{E}_h$;   
\item There exists a constant $C_{M}^I$ independent of $h$, such that
for any $\boldsymbol{p}_h \in \boldsymbol{Q}_h$, 
\begin{equation}\label{approx_mixed}
\inf_{\boldsymbol{p}^{I}_h\in
\boldsymbol{Q}^{c}_h} (\|\boldsymbol{p}^{I}_h-\boldsymbol{p}_h\|+\|{\rm
div}_h(\boldsymbol{p}^{I}_h-\boldsymbol{p}_h)\|) \leq
C_{M}^I \sum_{e\in
\mathcal{E}^i_h} h_e^{-1/2}\|[\boldsymbol{p}_h]\|_{0,e},
\end{equation}
where $\boldsymbol Q^c_h = \boldsymbol Q_h \cap \boldsymbol H({\rm
div}; \Omega)$.
\end{enumerate}
\end{assumption}

We note that the first assumption in Assumption \ref{Ass:WG:mixed}
ensures well-posedness of the mixed methods \eqref{mixedmethod}.
Several examples are given below.

\begin{example}\label{Example:WG:RT}
Raviart-Thomas type: $ \boldsymbol{Q}_h = \boldsymbol{Q}_h^{k, RT},
\widehat Q(e) = \mathcal{P}_k(e), V_h = V_h^k$, for $k \geq 0$. 
\end{example}

\begin{example}\label{Example:WG:BDM}
Brezzi-Douglas-Marini type: $
\boldsymbol{Q}_h = \boldsymbol Q_h^{k+1}, \widehat Q(e) =
\mathcal{P}_{k+1}(e), V_h = V_h^k$, for $k\geq 0$. 
\end{example}

\begin{lemma}
If we choose the spaces as in Example \ref{Example:WG:BDM} or Example
\ref{Example:WG:RT}, then Assumption \ref{Ass:WG:mixed} holds. 
\end{lemma}

\begin{proof}
We only sketch the proof of \eqref{approx_mixed} in Assumption
\ref{Ass:WG:mixed}. Denote the set of degrees of freedom of RT or BDM
element by $D$, see \cite{brezzi1991mixed, boffi2013mixed}. We then
define $\boldsymbol p_h^I$ as 
$$ 
d(\boldsymbol p_h^I) = \frac{1}{|\mathcal T_d|} \sum_{K \in \mathcal
  T_d} d(\boldsymbol p_h|_T) \qquad \forall d \in D,
$$ 
where $\mathcal T_d$ denotes the set of elements that share the degrees
of freedom $d$ and $|\mathcal T_d|$ denotes the cardinality of this
set. By the standard scaling argument, 
$$ 
\sum_{K \in \mathcal T_h} \|\boldsymbol p_h^I - \boldsymbol p_h\|
\lesssim \sum_{e\in \mathcal E_h^i} h_e^{1/2} \|[\boldsymbol
p_h]\|_{0,e}.
$$ 
Then \eqref{approx_mixed} follows from the inverse inequality.
\end{proof}

We rewrite the WG methods \eqref{eq:WGprimal} in the variational
form as: Find $(\boldsymbol{p}^{\eta}_h, u^{\eta}_h, \hat{\boldsymbol
p}^{\eta}_h)\in \boldsymbol Q_h\times V_h\times \widehat{\boldsymbol
Q}_h$ such that for any  $(\boldsymbol{q}_h, v_h, \hat{\boldsymbol
q}_h)\in \boldsymbol Q_h\times V_h\times \widehat{\boldsymbol Q}_h$
\begin{equation} \label{newWG}
\left\{
\begin{aligned}
(c \boldsymbol{p}^{\eta}_h, \boldsymbol{q}_h)_{\mathcal T_h} +
\rho^{-1} \langle h^{-1} (\boldsymbol p_h^\eta - \hat{\boldsymbol
    p}_h) \cdot \vect{n}, (\boldsymbol q_h - \hat{\boldsymbol
    q}_h) \cdot \vect{n} \rangle_{\partial \mathcal T_h} + 
(\nabla u_h, q_h)_{\mathcal T_h} - \langle u_h, \hat{\boldsymbol q}_h
\cdot \vect{n} \rangle_{\partial \mathcal T_h} & = 0,\\
-(\boldsymbol p_h^\eta, \nabla v_h)_{\partial \mathcal T_h} + \langle
\hat{\boldsymbol p}_h^\eta \cdot \vect{n}, v_h \rangle_{\partial
  \mathcal{T}_h}& = (f, v_h).
\end{aligned}
\right.
\end{equation}

\begin{theorem} \label{thm:WG-mixed}
Under the Assumption \ref{Ass:WG:mixed}, WG methods
\eqref{eq:WGprimal} converge to the mixed methods
\eqref{eq:mixedmethods} as $\rho \to 0$ with $\eta =
\rho^{-1}h_K^{-1}$.  More precisely, we have  
\begin{equation} \label{equ:WG-mixed}
\|\boldsymbol{p}^{\eta}_h-\boldsymbol{p}^{c}_h\|_{\boldsymbol H_h({\rm
  div})}+\|u^{\eta}_h-u_h^c\| \leq C_{w,3}\rho^{1/2}\|f\|,
\end{equation}
where $C_{w,3}$ is independent of both mesh size $h$ and $\rho$.
\end{theorem}

\begin{proof}
From the assumption $\{\!\!\{\boldsymbol Q_h\}\!\!\}|_e \subset
\widehat{Q}(e)$, by taking $\boldsymbol q_h = \boldsymbol q_h^c$ and
$\hat{\boldsymbol q}_h|_e = (\boldsymbol q_h^c \cdot
\vect{n}_e)\vect{n}_e$ in \eqref{newWG} and integrating by parts, we
see that $(\boldsymbol{p}_h^{\eta}, u_h^{\eta})$ satisfies 
\begin{equation} \label{equ:WG-conforming}
(c \boldsymbol{p}^{\eta}_h, \boldsymbol{q}^c_h)_{\mathcal
T_h}-(u^{\eta}_h,  {\rm div} \boldsymbol{q}^c_h)_{\mathcal T_h}=0
\qquad \forall \boldsymbol{q}^c_h\in \boldsymbol Q^c_h.
\end{equation}
Subtracting \eqref{mixedmethod} from \eqref{equ:WG-conforming} and
the second equation of \eqref{newWG}, we have 
\begin{equation}\label{mixederror}
\left\{
\begin{aligned}
(c (\boldsymbol{p}^{\eta}_h-\boldsymbol{p}^{c}_h),
\boldsymbol{q}^c_h)_{\mathcal T_h} 
- (u^{\eta}_h-u_h^c,  {\rm div} \boldsymbol{q}^c_h)_{\mathcal
T_h} &= 0 \qquad \qquad\qquad\qquad\qquad ~~\forall \boldsymbol{q}^c_h\in
\boldsymbol Q^c_h,\\
({\rm div}(\boldsymbol{p}^{\eta}_h - \boldsymbol{p}^{c}_h),
 v_h)_{\mathcal T_h} &=
\langle(\boldsymbol{p}^{\eta}_h-\hat{\boldsymbol p}^{\eta}_h)\cdot
\boldsymbol {n}, v_h\rangle_{\partial {\mathcal T_h}} \qquad  \forall
v_h \in V_h. 
\end{aligned}
\right.
\end{equation}
Noting that $\boldsymbol p_h^\eta \not\in \boldsymbol{Q}_h^c$, we have
that, for any  $\boldsymbol{p}^{I}_h\in \boldsymbol Q^c_h$, 
\begin{equation}\label{mixederror1}
\left\{
\begin{aligned}
(c (\boldsymbol{p}^{I}_h-\boldsymbol{p}^{c}_h),
\boldsymbol{q}^c_h)_{\mathcal T_h} - (u^{\eta}_h-u_h^c,  {\rm div}
\boldsymbol{q}^c_h)_{\mathcal T_h} &=
(c(\boldsymbol{p}^{I}_h-\boldsymbol{p}^{\eta}_h),\boldsymbol{q}^c_h)_{\mathcal
T_h}~~~\qquad\qquad\qquad\qquad\qquad\qquad \forall \boldsymbol{q}^c_h\in
\boldsymbol Q^c_h,\\
({\rm div} (\boldsymbol{p}^{I}_h-\boldsymbol{p}^{c}_h),v_h)_{\mathcal
T_h} &= \langle(\boldsymbol{p}^{\eta}_h - \hat{\boldsymbol
p}^{\eta}_h)\cdot \boldsymbol{n}, v_h\rangle_{\partial \mathcal T_h} 
+({\rm div}(\boldsymbol{p}^{I}_h-\boldsymbol{p}^{\eta}_h),v_h)_{\mathcal
  T_h} \qquad  \forall  v_h \in V_h. 
\end{aligned}
\right.
\end{equation}
Because $(\boldsymbol{p}^{I}_h-\boldsymbol{p}^{c}_h)\in
\boldsymbol Q^c_h, (u^{\eta}_h-u_h^c)\in V_h$, by the well-posedness
of the mixed methods \eqref{equ:well-posedness-mixed}, trace
inequality, inverse inequality and Cauchy inequality, we have
$$ 
\begin{aligned}
& \|\boldsymbol{p}^{I}_h-\boldsymbol{p}^{c}_h\|_{\boldsymbol H({\rm
div}) }+\|u^{\eta}_h-u_h^c\| \\
\leq &~ C_M\left( \sup_{\boldsymbol{q}^c_h\in
\boldsymbol{Q}^c_h} \frac{(c(\boldsymbol{p}^{I}_h -
\boldsymbol{p}^{\eta}_h), \boldsymbol{q}^c_h)_{\mathcal
T_h}}{\|\boldsymbol{q}^c_h\|_{\boldsymbol H({\rm div})}} 
+\sup_{v_h\in V_h} \frac{\langle(\boldsymbol{p}^{\eta}_h - \hat
{\boldsymbol p}^{\eta}_h)\cdot \boldsymbol {n}, v_h\rangle_{\partial
{\mathcal T_h}} + ({\rm
div}(\boldsymbol{p}^{I}_h-\boldsymbol{p}^{\eta}_h),v_h)_{\mathcal
T_h}}{\|v_h\|} \right)\\
\lesssim &~ 
\|\boldsymbol{p}^{I}_h-\boldsymbol{p}^{\eta}_h\|+\|{\rm
div}_h(\boldsymbol{p}^{I}_h-\boldsymbol{p}^{\eta}_h)\|+
\langle h^{-1}(\boldsymbol{p}^{\eta}_h-\hat {\boldsymbol
p}^{\eta}_h)\cdot \boldsymbol {n},
(\boldsymbol{p}^{\eta}_h-\hat{\boldsymbol p}^{\eta}_h)\cdot
\boldsymbol {n}\rangle_{\partial \mathcal T_h}^{1/2}.
\end{aligned}
$$
Hence, by Assumption \ref{Ass:WG:mixed} and inverse inequality, we have
$$ 
\begin{aligned}
\|\boldsymbol{p}^{\eta}_h-\boldsymbol{p}^{c}_h\|_{\boldsymbol H({\rm
    div})}+\|u^{\eta}_h-u_h^c\| 
& \lesssim \langle h^{-1}(\boldsymbol{p}^{\eta}_h-\hat {\boldsymbol
p}^{\eta}_h)\cdot \boldsymbol {n},
(\boldsymbol{p}^{\eta}_h-\hat{\boldsymbol p}^{\eta}_h)\cdot
\boldsymbol {n}\rangle_{\partial \mathcal T_h}^{1/2}
+ \inf_{\boldsymbol p_h^I \in \boldsymbol Q_h^c}
\left( \|\boldsymbol{p}^{I}_h-\boldsymbol{p}^{\eta}_h\| 
+ \|{\rm div}_h(\boldsymbol{p}^{I}_h-\boldsymbol{p}^{\eta}_h)\| 
\right) \\
& \lesssim \langle h^{-1}(\boldsymbol{p}^{\eta}_h-\hat {\boldsymbol
p}^{\eta}_h)\cdot \boldsymbol {n},
(\boldsymbol{p}^{\eta}_h-\hat{\boldsymbol p}^{\eta}_h)\cdot
\boldsymbol {n}\rangle_{\partial \mathcal T_h}^{1/2} + \sum_{e\in
  \mathcal{E}_h^i} h_e^{-1/2}\|[\boldsymbol p_h^\eta]\|_{0,e}.
\end{aligned}
$$
From the fact that 
$$
\langle(\boldsymbol{p}^{\eta}_h-\boldsymbol {\hat p}^{\eta}_h)\cdot
\boldsymbol {n}, (\boldsymbol{p}^{\eta}_h-\boldsymbol {\hat
p}^{\eta}_h)\cdot \boldsymbol {n}\rangle_{\partial {\mathcal T_h}}
= 2\langle \{\!\!\{\boldsymbol{p}^{\eta}_h-\boldsymbol {\hat
  p}^{\eta}_h\}\!\!\},\{\!\!\{\boldsymbol{p}^{\eta}_h-\boldsymbol
{\hat p}^{\eta}_h\}\!\!\}\rangle_{\mathcal E_h}+\frac{1}{2}\langle
[\boldsymbol{p}^{\eta}_h],[\boldsymbol{p}^{\eta}_h]\rangle_{\mathcal
  E_h},
$$
we obtain  
$$ 
\begin{aligned}
\|\boldsymbol{p}^{\eta}_h-\boldsymbol{p}^{c}_h\|_{\boldsymbol H({\rm
div})}+\|u^{\eta}_h-u_h^c\| 
& \lesssim \langle h^{-1}(\boldsymbol{p}^{\eta}_h-\hat {\boldsymbol
p}^{\eta}_h)\cdot \boldsymbol {n},
(\boldsymbol{p}^{\eta}_h-\hat{\boldsymbol p}^{\eta}_h)\cdot
\boldsymbol {n}\rangle_{\partial \mathcal T_h}^{1/2} + \sum_{e\in
  \mathcal{E}_h^i} h_e^{-1/2}\|[\boldsymbol p_h^\eta]\|_{0,e} \\
& \lesssim \langle h^{-1}(\boldsymbol{p}^{\eta}_h-\hat {\boldsymbol
p}^{\eta}_h)\cdot \boldsymbol {n},
(\boldsymbol{p}^{\eta}_h-\hat{\boldsymbol p}^{\eta}_h)\cdot
\boldsymbol {n}\rangle_{\partial \mathcal T_h}^{1/2} \lesssim
\rho^{1/2}\|f\|,
\end{aligned}
$$ 
where we used Corollary \ref{uniform_stable_WG} in the last
step. This completes the proof. 
\end{proof}

\paragraph{Numerical examples on the convergence from WG methods
to mixed methods}
We present some numerical examples to support the theoretical results.
We consider the 2D Poisson problem described in
\eqref{equ:numerical-example}.  A uniform grid with $h = 1/4$ is fixed
for different $\rho$'s with $\eta = \rho^{-1}h_K^{-1}$.  First, we
choose the RT-type discrete spaces in WG methods, see Example
\ref{Example:WG:RT}.  When $\rho \to 0$, WG methods do converge to the
mixed methods, see Table \ref{tab:WG2RT}. 

\begin{table}[!htbp] 
\centering
\begin{subtable}{\textwidth} 
\centering
\caption{From WG to RT: $\eta = \rho^{-1}h^{-1}$, $k=0$, uniform grid}
\begin{tabular}{c|cc|cc|cc}
\hline
$\rho$	& $\|u_h^\eta -u_h^{c,{\rm RT}}\|$ & $\rho^\alpha$
& $\|\boldsymbol{p}_h^\eta - \boldsymbol{p}_h^{c, {\rm RT}}\|_0$ 
& $\rho^\alpha$ & $\|{\rm div}_h (\boldsymbol{p}_h^\eta -
\boldsymbol{p}_h^{c,{\rm RT}}) \|$ & $\rho^\alpha$ \\ \hline
1/4 & 0.003539 & -- & 0.025589 & -- & 0.101364 & -- \\ \hline
1/8 & 0.001777 & 0.99 & 0.012850 & 0.99 & 0.050819 & 1.00 \\ \hline
1/16 & 0.000890 & 1.00 & 0.006439 & 1.00 & 0.025444 & 1.00 \\ \hline
\end{tabular}
\label{tab:WG-P1-P0-RT}
\end{subtable} %

\bigskip
\begin{subtable}{\textwidth} 
\centering
\caption{From WG to RT: $\eta = \rho^{-1}h^{-1}$, $k=1$, uniform grid}
\begin{tabular}{c|cc|cc|cc}
\hline
$\rho$	& $\|u_h^\eta -u_h^{c,{\rm RT}}\|$ & $\rho^\alpha$
& $\|\boldsymbol{p}_h^\eta - \boldsymbol{p}_h^{c, {\rm RT}}\|_0$ 
& $\rho^\alpha$ & $\|{\rm div}_h (\boldsymbol{p}_h^\eta -
\boldsymbol{p}_h^{c,{\rm RT}}) \|$ & $\rho^\alpha$ \\ \hline
1/4 & 0.0003681 & -- & 0.004955 & -- & 0.102957 & -- \\ \hline
1/8 & 0.0001843 & 1.00 & 0.002482 & 1.00 & 0.051582 & 1.00 \\ \hline
1/16 & 0.0000922 & 1.00 & 0.001242 & 1.00 & 0.025817 & 1.00 \\ \hline
\end{tabular}
\label{tab:WG-P2-P1-RT}
\end{subtable} 
\caption{Convergence rate from WG to RT mixed methods on 2D uniform
  grids}
\label{tab:WG2RT}
\end{table}

Next, we choose the BDM-type discrete spaces in WG methods, see
Example \ref{Example:WG:BDM}. When $\rho \to 0$, WG methods do
converge to BDM, see Table \ref{tab:WG2BDM}. Further, we observe
the first-order convergence on $\rho$ in both Table \ref{tab:WG2RT}
and Table \ref{tab:WG2BDM}, which is $1/2$-order higher than our
theoretical finding in Theorem \ref{thm:WG-mixed}.

\begin{table}[!htbp] 
\centering
\begin{subtable}{\textwidth} 
\centering
\caption{From WG to BDM: $\eta = \rho^{-1}h^{-1}$, $k=0$, uniform grid}
\begin{tabular}{c|cc|cc|cc}
\hline
$\rho$	& $\|u_h^\eta -u_h^{c,{\rm BDM}}\|$ & $\rho^\alpha$
& $\|\boldsymbol{p}_h^\eta - \boldsymbol{p}_h^{c, {\rm BDM}}\|_0$ 
& $\rho^\alpha$ & $\|{\rm div}_h (\boldsymbol{p}_h^\eta -
\boldsymbol{p}_h^{c,{\rm BDM}}) \|$ & $\rho^\alpha$ \\ \hline
1/4 & 0.005046 & -- & 0.045969 & -- & 0.096506 & -- \\ \hline
1/8 & 0.002547 & 0.99 & 0.023223 & 0.99 & 0.048526 & 0.99 \\ \hline
1/16 & 0.001280 & 0.99 & 0.011672 & 0.99 & 0.024332 & 1.00 \\ \hline
\end{tabular}
\label{tab:WG-P1-P0-BDM}
\end{subtable} %

\vspace{1mm}
\begin{subtable}{\textwidth} 
\centering
\caption{From WG to BDM: $\eta = \rho^{-1}h^{-1}$, $k=1$, uniform grid}
\begin{tabular}{c|cc|cc|cc}
\hline
$\rho$	& $\|u_h^\eta -u_h^{c,{\rm BDM}}\|$ & $\rho^\alpha$
& $\|\boldsymbol{p}_h^\eta - \boldsymbol{p}_h^{c, {\rm BDM}}\|_0$ 
& $\rho^\alpha$ & $\|{\rm div}_h (\boldsymbol{p}_h^\eta -
\boldsymbol{p}_h^{c,{\rm BDM}}) \|$ & $\rho^\alpha$ \\ \hline
1/4 & 0.000617 & -- & 0.009329 & -- & 0.102282 & -- \\ \hline
1/8 & 0.000310 & 0.99 & 0.004683 & 0.99 & 0.051316 & 1.00 \\ \hline
1/16 & 0.000155 & 1.00 & 0.002346 & 1.00 & 0.025702 & 1.00 \\ \hline
\end{tabular}
\label{tab:WG-P2-P1-BDM}
\end{subtable} 
\caption{Convergence rate from WG to BDM mixed methods on 2D uniform
  grids}
\label{tab:WG2BDM}
\end{table}

\subsection{Primal methods as the limiting case of HDG methods}
For a given mesh, we next try to prove that the HDG methods
\eqref{eq:hybridgalerkin} converge to primal methods \eqref{eq:primal}
when $\rho \to 0$ and the stabilization parameter is set to be
$\tau = \rho^{-1}h_K^{-1}$.  

Consider the $H^1$ conforming subspace $ V^c_h=V_h\cap
H^1_0(\Omega)\subset V_h$, then the primal methods \eqref{eq:primal} in the
variational form are written as: Find  $(u^c_h, \boldsymbol p^c_h)\in
V^c_h\times \boldsymbol Q_h$ such that  
\begin{equation} \label{primalmethod}
\left\{
\begin{aligned}
(c\boldsymbol p^c_h, \boldsymbol q_h)_{\mathcal T_h}+ (\nabla
u^c_h,\boldsymbol q_h)_{\mathcal T_h} &= (\boldsymbol g_1, \boldsymbol
  q_h)_{\mathcal T_h} + \langle g_2, \boldsymbol q_h\cdot
\vect{n}\rangle_{\partial \mathcal T_h}\qquad
\forall \boldsymbol q_h\in \boldsymbol Q_h,\\
-(\boldsymbol p^c_h, \nabla v^c_h)_{\mathcal T_h} &=
(f,v^c_h)_{\mathcal T_h}\qquad\qquad\qquad\qquad\qquad \forall
v^c_h\in V^c_h,
\end{aligned}
\right.
\end{equation}
where $\boldsymbol g_1=0$ and $g_2 = 0$ when applied to the Poisson
equation \eqref{H1}.  Then, by $\nabla  V^c_h\subset \nabla_h
V_h\subset \boldsymbol Q_h$, the well-posedness of the primal methods
(cf.  \cite{brenner2007mathematical}) implies that  
\begin{equation} \label{primalstability}
 \|\boldsymbol{p}^c_h\| + \|{u}^c_h\|_{1}  \leq C_p\left( \|f\|_{-1,h}
     + \sup_{\boldsymbol q_h \in \boldsymbol Q_h} \frac{(\boldsymbol
       g_1, \boldsymbol q_h)_{\mathcal T_h} + \langle g_2, \boldsymbol
     q_h\cdot \boldsymbol n\rangle_{\partial \mathcal
     T_h}}{\|\boldsymbol q_h\|} \right),
\end{equation}
where $\|f\|_{-1,h}=\sup\limits_{v^c_h\neq 0,v^c_h\in
V_h^c}\frac{(f,v^c_h)_{\mathcal T_h}}{\|{v}^c_h\|_{1}}$.

Recall that the space define on $\mathcal E_h$ (see
\eqref{Edge:spaces}) of HDG methods is given by  
$$ 
\widehat{V}_h = \{\hat{v}_h: \hat{v}_h|_e\in \widehat{V}(e), \forall e
\in \mathcal E_h^i, \hat{v}_h|_{\mathcal E_h^\partial} = 0\}.
$$ 
We make the following assumption on the finite element spaces of HDG
methods. 

\begin{assumption} \label{Ass:HDG:Primal} 
Assume that the spaces $\boldsymbol{Q}_h, V_h$ and $\widehat V_h$
satisfy
\begin{enumerate}
\item $\nabla_h V_h \subset \boldsymbol{Q}_h$;
\item $\{V_h\}|_e \subset \widehat V(e)$, ~$\forall e \in \mathcal E_h^i$;   
\item There exists a constant $C_p^I$ independent of $h$, such that
for any $u_h\in V_h$, 
\begin{equation} \label{jump}
\inf_{u_h^I \in V_h^c} \left( \|(u_h^I - u_h^c\| + \|\nabla_h
      (u_h^I-u_h^c)\| \right) \leq C_p^I
\sum_{e\in \mathcal{E}_h} h^{-1/2}_e \|\llbracket u_h \rrbracket
\|_{0,e},
\end{equation}
where $V_h^c = V_h \cap H^1(\Omega)$.
\end{enumerate}
\end{assumption}

We note that the first assumption in Assumption \ref{Ass:HDG:Primal}
ensures the well-posedness of the primal methods \eqref{primalmethod}.
The following example satisfies Assumption \ref{Ass:HDG:Primal}
(see the conforming relatives in \cite{brenner2005c,
brenner2007mathematical}).

\begin{example} \label{Example:HDG:primal}
$\boldsymbol{Q}_h = \boldsymbol{Q}_h^{k}, V_h= V_h^{k+1},
  \widehat{V}(e) = \mathcal{P}_{k+1}(e)$, for $k \geq 0$. 
\end{example}

We rewrite the HDG methods \eqref{eq:hybridgalerkin} in the
variational form as: Find $(\boldsymbol{p}_h^{\tau},u_h^{\tau}, \hat
u_h^{\tau})\in \boldsymbol Q_h\times V_h\times \widehat{V}_h$ such
that for any $(\boldsymbol q_h, v_h, \hat{v}_h) \in \boldsymbol{Q}_h
\times V_h \times \widehat{V}_h$
\begin{equation} \label{newHDG}
\left\{
\begin{aligned}
(c \boldsymbol{p}_h^{\tau}, \boldsymbol{q}_h)_{\mathcal T_h} -
(u_h^\tau, {\rm div}\boldsymbol{q}_h)_{\mathcal T_h} + \langle
\hat{u}_h^\tau ,\boldsymbol q_h \cdot \vect{n}\rangle_{\partial
  \mathcal T_h} &= 0, \\
({\rm div}\boldsymbol p_h^\tau, v_h)_{\mathcal T_h} - \langle
\boldsymbol{p}_h^\tau \cdot \vect{n}, \hat{v}_h\rangle_{\partial \mathcal
  T_h} + \rho^{-1} \langle h^{-1} (u_h^\tau - \hat{u}_h^\tau), v_h -
\hat{v}_h \rangle_{\partial \mathcal T_h} &= (f, v_h)_{\mathcal
T_h}.
\end{aligned}
\right.
\end{equation}

\begin{theorem} \label{thm:HDG-primal}
Under the Assumption \ref{Ass:HDG:Primal}, the HDG methods
\eqref{eq:hybridgalerkin} converge to the primal methods
\eqref{eq:primal} as $\rho \to 0$ with $\tau = \rho^{-1}h_K^{-1}$.
More precisely, we have  
\begin{equation} \label{equ:HDG2primal}
\|\boldsymbol{p}^{\tau}_h-\boldsymbol{p}^{c}_h\|+\|u^{\tau}_h-u_h^c\|_{1,h}
\leq C_{d,3}\rho^{1/2}\|f\|_{-\tilde 1,\rho, h},
\end{equation}
where $C_{d,3}$ is independent of both mesh size $h$ and $\rho$, and
$\|f\|_{-\tilde 1,\rho,h} = \sup\limits_{\tilde{v}_h \in \widetilde{V}_h} \frac{(f,
v_h)_{\mathcal{T}_h}}{\|\tilde{v}_h\|_{\tilde{1},\rho,h}}$.
\end{theorem}
\begin{proof}
From the assumption $\{V_h\}|_e \subset \widehat V(e)$, by taking $v_h
= v_h^c$ and $\hat{v}_h|_e = v_h^c|_e$ in \eqref{newHDG} and
integrating by parts, we see that 
\begin{equation} \label{HDG62}
-(\boldsymbol{p}_h^{\tau},\nabla v_h^c)_{\mathcal T_h} =
(f,v_h^c)_{\mathcal T_h} \qquad  \forall  v_h^c \in V_h^c.
\end{equation}
Subtracting \eqref{primalmethod} from the first equation of
\eqref{newHDG} and \eqref{HDG62}, we have 
\begin{equation}\label{eq:error}
\left\{
\begin{aligned}
(c(\boldsymbol p_h^{\tau}-\boldsymbol p^c_h),\boldsymbol
q_h)_{\mathcal T_h} + (\nabla u_h^{\tau} - \nabla u^c_h,\boldsymbol
  q_h)_{\mathcal T_h} &= \langle u_h^{\tau}-\hat u_h^{\tau},\boldsymbol
q_h\cdot \boldsymbol n\rangle_{\partial {\mathcal T_h}} \qquad
  \forall \boldsymbol q_h\in \boldsymbol Q_h, \\
-(\boldsymbol p_h^{\tau}-\boldsymbol p_h^c,\nabla v_h^c)_{\mathcal T_h}
&= 0  \qquad\qquad\qquad\qquad \qquad\forall v_h^c\in V_h^c.
\end{aligned}
\right.
\end{equation}
Again, for any $u_h^I \in V_h^c$, we have 
\begin{equation}\label{error}
\left\{
\begin{aligned}
(c(\boldsymbol p_h^{\tau}-\boldsymbol p^c_h), \boldsymbol
 q_h)_{\mathcal T_h} + (\nabla u_h^{I}-\nabla u^c_h,\boldsymbol
   q_h)_{\mathcal T_h} &= \langle u_h^{\tau}-\hat u_h^{\tau},\boldsymbol
 q_h\cdot \boldsymbol n\rangle_{\partial {\mathcal T_h}} + (\nabla
u_h^I - \nabla u_h^{\tau},\boldsymbol q_h)_{\mathcal T_h} \qquad
 \forall \boldsymbol q_h \in \boldsymbol Q_h, \\
-(\boldsymbol p_h^{\tau}-\boldsymbol p_h^c,\nabla v_h^c)_{\mathcal
  T_h} &= 0 \qquad \qquad \qquad \qquad \qquad \qquad\qquad \qquad
  ~~\quad\qquad \forall v_h^c \in V_h^c.
\end{aligned}
\right.
\end{equation}
Because $\boldsymbol p_h^{\tau}-\boldsymbol p^c_h\in \boldsymbol
Q_h$ and $v_h^{c}-u^c_h\in V_h^c$, using \eqref{primalstability},
trace inequality, inverse inequality and Cauchy inequality, we obtain
\begin{equation}\label{error1}
\begin{aligned}
\|\boldsymbol p_h^{\tau}-\boldsymbol p_h^{c}\|+\|u_h^I-u_h^c\|_1 &\leq
C_p\sup_{\boldsymbol q_h\in \boldsymbol Q_h}\frac{\langle
u_h^{\tau}-\hat u_h^{\tau},\boldsymbol q_h\cdot \boldsymbol
n\rangle_{\partial {\mathcal T_h}} + (\nabla u_h^{I}-\nabla
u^\tau_h,\boldsymbol q_h)_{\mathcal T_h}}{\|\boldsymbol q_h\|} \\
& \lesssim |u_h^I - u_h^\tau|_{1,h} + \langle h^{-1}(u_h^\tau -
\hat{u}_h^\tau), u_h^\tau - \hat{u}_h^\tau \rangle_{\partial \mathcal
T_h}^{1/2}. 
\end{aligned}
\end{equation}
Noting that the local projection $\hat{P}_h$ in
\eqref{equ:HDG-norms-grad} is an identity operator as $\{V_h\}|_e
\subset \widehat{V}(e)$, and 
\begin{equation}\label{identity}
 \langle u_h^{\tau}-\hat{u}_h^{\tau},
 u_h^{\tau}-\hat{u}_h^{\tau}\rangle_{\partial {\mathcal T_h}}
 =2\langle\{u_h^{\tau}-\hat{u}_h^{\tau}\},
 \{u_h^{\tau}-\hat{u}_h^{\tau}\}\rangle_{\mathcal{E}_h}+
 \frac{1}{2}\langle \lbrack\!\lbrack u_h^{\tau}\rbrack\!\rbrack,
 \lbrack\!\lbrack u_h^{\tau} \rbrack\!\rbrack\rangle_{\mathcal{E}_h}.
\end{equation}
Therefore, Assumption \ref{Ass:HDG:Primal}, \eqref{error1}, and
\eqref{identity} imply that 
$$ 
\begin{aligned}
\|\boldsymbol p_h^{\tau}-\boldsymbol p_h^{c}\| +
\|u_h^{\tau}-u_h^c\|_{1,h} & \leq \inf_{u_h^I \in V_h^c} 
\left( \|\boldsymbol p_h^\tau - \boldsymbol p_h^c\| + \|u_h^I -
u_h^c\|_{1} + \|u_h^\tau - u_h^I\|_{1,h} \right) \\
& \lesssim  \langle h^{-1}(u_h^\tau - \hat{u}_h^\tau), u_h^\tau -
\hat{u}_h^\tau \rangle_{\partial \mathcal T_h}^{1/2} +  \inf_{u_h^I \in
  V_h^c} \|u_h^\tau - u_h^I\|_{1,h} \\
& \lesssim  \langle h^{-1}(u_h^\tau - \hat{u}_h^\tau), u_h^\tau -
\hat{u}_h^\tau \rangle_{\partial \mathcal T_h}^{1/2} + \sum_{e\in \mathcal
  E_h} h_e^{-1/2} \|[u_h]\|_{0,e} \\
& \lesssim \langle h^{-1}(u_h^\tau - \hat{u}_h^\tau), u_h^\tau -
\hat{u}_h^\tau \rangle_{\partial \mathcal T_h}^{1/2} \\
& \lesssim \rho^{1/2}\|f\|_{-\tilde 1, \rho, h},
\end{aligned}
$$
where Corollary \ref{uniform_stable} was used in the last step. 
\end{proof}

\begin{remark}
We have $\hat{P}_h$ as an identity operator in the definition of
$\|\cdot\|_{\tilde{1},\rho, h}$ (see \eqref{equ:HDG-norms-grad}).
Therefore, when $\rho \lesssim 1$, we have  
$$ 
\begin{aligned}
\inf_{\hat{v}_h \in \widehat{V}_h} \|\tilde{v}_h\|^2_{\tilde{1},
  \rho, h} &= \inf_{\hat{v}_h \in \widehat{V}_h} (\nabla v_h, \nabla
v_h)_{\mathcal T_h} + \sum_{K\in \mathcal T_h} \rho^{-1} h_K^{-1}
\langle v_h - \hat{v}_h, v_h - \hat{v}_h \rangle_{\partial K} \\
& \simeq (\nabla v_h, \nabla v_h)_{\mathcal T_h} + \rho^{-1} 
\sum_{e\in \mathcal E_h}h_e^{-1} \|\llbracket v_h \rrbracket\|_{0,e}^2
\gtrsim \|v_h\|_{1,h}. 
\end{aligned}
$$ 
Hence, when $\rho \lesssim 1$,  
$$ 
\|f\|_{-\tilde 1,\rho,h} = \sup_{\tilde{v}_h \in \widetilde{V}_h} \frac{(f,
    v_h)_{\mathcal T_h}}{\|\tilde{v}_h\|_{\tilde{1},\rho, h}} = 
\sup_{v_h \in V_h} \frac{(f, v_h)_{\mathcal
  T_h}}{\inf_{\hat{v}_h \in
    \widehat{V}_h}\|\tilde{v}_h\|_{\tilde{1},\rho, h}} \lesssim
    \sup_{v_h \in V_h} \frac{(f, v_h)_{\mathcal T_h}}{\|v_h\|_{1,h}}
    \lesssim \|f\|, 
$$ 
which means that the solutions of HDG methods converge to the those of
primal methods of order $\rho^{1/2}$ at least.
\end{remark}

\paragraph{Numerical examples on the convergence from HDG methods to
primal methods}

We present some numerical examples to support the theoretical results.
We consider the 2D Poisson problem described in
\eqref{equ:numerical-example}.  A uniform grid with $h = 1/4$ is fixed
for different $\rho$'s with $\tau = \rho^{-1}h_K^{-1}$. We choose the  
discrete spaces in Example \ref{Example:HDG:primal} for HDG methods.
We observe the convergence from HDG methods to primal methods in Table
\ref{tab:HDG2Primal}. Similar to the numerical examples from WG to
mixed methods, the convergence rate seems to be of order one, which is
higher than our theoretical finding in Theorem \ref{thm:HDG-primal}.

\begin{table}[!htbp] 
\begin{subtable}{.5\textwidth} 
\centering
\caption{$\tau = \rho^{-1}h^{-1}$: $k=0$, uniform grid}
\begin{tabular}{c|cc|cc}
\hline
$\rho$	& $\|u_h^\tau -u_h^c\|_{0}$	&$\rho^\alpha$	&$|u_h^\tau -
u_h^c|_{1,h}$ & $\rho^\alpha$ \\ \hline
1/4 & 0.307403 & -- & 1.639017 & -- \\ \hline
1/8 & 0.205585 & 0.58 & 1.066058 & 0.62\\ \hline
1/16 & 0.137903 & 0.58 & 0.727442 & 0.55\\ \hline 
1/32 & 0.088576 & 0.64 & 0.484124 & 0.59\\ \hline 
1/64 & 0.053204 & 0.73 & 0.300815 & 0.69\\ \hline 
1/128 & 0.029933 & 0.83 & 0.173458 & 0.79\\ \hline 
1/256 & 0.016031 & 0.90 & 0.094330 & 0.88\\ \hline 
1/512 & 0.008321 & 0.95 & 0.049389 & 0.93\\ \hline 
\end{tabular}
\label{tab:HDG-P1-P0}
\end{subtable} %
\begin{subtable}{.5\textwidth} 
\centering
\caption{$\tau = \rho^{-1}h^{-1}$, $k=1$, uniform grid}
\begin{tabular}{c|cc|cc}
\hline
$\rho$	& $\|u_h^\tau -u_h^c\|_{0}$	&$\rho^\alpha$	&$|u_h^\tau -
u_h^c|_{1,h}$ & $\rho^\alpha$ \\ \hline
1/4 & 0.037789 & -- & 0.686199 & -- \\ \hline
1/8 & 0.022577 & 0.74 & 0.383381 & 0.84\\ \hline
1/16 & 0.014676 & 0.62 & 0.240527 & 0.67\\ \hline 
1/32 & 0.009914 & 0.57 & 0.165801 & 0.54 \\ \hline 
1/64 & 0.006539 & 0.60 & 0.114649 & 0.53\\ \hline 
1/128 & 0.004046 & 0.69 & 0.074111 & 0.63\\ \hline 
1/256 & 0.002331 & 0.80 & 0.044007 & 0.75\\ \hline 
1/512 & 0.001269 & 0.88 & 0.024376 & 0.85\\ \hline 
\end{tabular}
\label{tab:HDG-P2-P1}
\end{subtable}
\caption{Convergence rate from HDG to primal methods on 2D uniform grids}
\label{tab:HDG2Primal}
\end{table}

\subsection{Duality relationship}

To make DG-derivatives good approximations of classical weak
derivatives, the dual relationship between gradient and divergence
operators should be preserved as shown in Lemma
\ref{lem:conditional-dual}. The (stabilized) hybrid primal methods and
(stabilized) hybrid mixed methods approximately satisfy conditions
(i) and (ii) in Lemma \ref{lem:conditional-dual},
respectively, and the DG methods adopt condition (iii)
approximately. In this subsection, we discuss the duality relationship
of various Galerkin methods in the context of convex optimization. 

To begin with, we know that the primal form \eqref{H1} can be
characterized as the following saddle point problem: Find
$(\vect{p},u)\in \boldsymbol L^2(\Omega) \times H^1_0(\Omega)$ such
that
$$
L(\vect{p}, u) = \inf_{v\in H_0^1(\Omega)}\sup_{\vect{q}\in
  \boldsymbol L^2(\Omega)} L(\vect{q}, v),
$$
where
$$
L(\vect{q}, v) := - (\vect{q}, \nabla v)_{\Omega} -
\frac{1}{2} (c \vect{q},\vect{q})_\Omega - (f, v)_\Omega. 
$$
In contrast, the mixed form of the elliptic problem is equivalent to
the following saddle point problem: Find $(\vect{p},u)\in \boldsymbol
H({\rm div},  \Omega) \times L^2(\Omega)$ such that
$$
L^*(\vect{p}, u) = \sup_{\vect{q}\in \boldsymbol H({\rm div}, \Omega)}
\inf_{v\in L^2(\Omega)} L^*(\vect{q}, v),
$$
where
$$
L^*(\vect{q},v) := ({\rm div} \vect{q}, v)_{\Omega} -
\frac{1}{2} (c \vect{q}, \vect{q})_{\Omega} - (f,v)_{\Omega}.
$$
The above two saddle point problems are dual with each other from the
point of view of duality in convex optimization \cite{boffi2013mixed,
ekeland1976convex}.

To mimic the above duality at the discrete level, we define the
following optimization target  
$$ 
L_h(\tilde{\vect{q}}_h, v_h) = -(\DGnabla v_h,
    \tilde{\vect{q}}_h)_{\Omega} - \frac{1}{2} (c {\vect{q}}_h,
      {\vect{q}}_h)_\Omega - (f_h, v_h)_\Omega,
$$ 
and the hybrid primal methods \eqref{eq:hybrid primal} can be derived
from the optimization problem   
\begin{equation} \label{equ:hybrid-primal-optimization} 
\inf_{v_h \in V_h} \sup_{\tilde{\vect{q}}_h \in
  \widetilde{\vect{Q}}_h} L_h(\tilde{\vect{q}}_h, v_h).
\end{equation} 
Furthermore, the hybrid mixed methods \eqref{eq:hybrid-mixed} can be
derived from the optimization problem 
\begin{equation} \label{equ:hybrid-mixed-optimization} 
\sup_{\vect{q}_h \in \vect{Q}_h} \inf_{\tilde{v}_h \in
\widetilde{V}_h} L_h^*(\vect{q}_h, \tilde{v}_h),
\end{equation} 
where 
$$ 
L_h^*(\vect{q}_h,\tilde{v}_h) := ( {\rm div}_{\rm dg}
\vect{q}_h, \tilde{v}_h)_{\Omega} - \frac{1}{2} ( c\vect{q}_h,
\vect{q}_h)_{\Omega} - (f_h, v_h)_\Omega.
$$ 
We immediately see that hybrid primal and hybrid mixed methods are dual
to each other as \eqref{equ:hybrid-primal-optimization} and
\eqref{equ:hybrid-mixed-optimization} are mutually dual in the
context of convex optimization. 

A similar argument can be applied to the stabilized methods. Given the
stabilized hybrid primal (WG) methods \eqref{eq:WGprimal}, we can prove
that it is equivalent to the optimization problem  
\begin{equation} \label{equ:stab-hybrid-primal-optimization} 
\inf_{v_h \in V_h} \sup_{\tilde{\vect{q}}_h \in
\widetilde{\vect{Q}}_h} L_{h}(\tilde{\vect{q}}_h, v_h) -
\frac{1}{2} \langle \eta(\vect{p}_h - \hat{\vect p}_h) \cdot \vect{n}
, (\vect{p}_h - \hat{\vect p}_h) \cdot \vect{n}
\rangle_{\partial\mathcal T_h},
\end{equation} 
The stabilized hybrid mixed (HDG) methods \eqref{eq:hybridgalerkin} is
equivalent to the optimization problem 
\begin{equation} \label{equ:stab-hybrid-mixed-optimization} 
\sup_{\vect{q}_h \in \vect{Q}_h} \inf_{\tilde{v}_h \in
  \widetilde{V}_h} L_{h}^*(\vect{q}_h, \tilde{v}_h) +
  \frac{1}{2}\langle \tau(v_h - \hat{v}_h), v_h - \hat{v}_h
  \rangle_{\partial \mathcal T_h}.
\end{equation} 
Since primal DG and mixed DG can be deduced formally from stabilized
hybrid mixed methods and stabilized hybrid primal methods, respectively, by taking
$\hat{u}_h = \{u_h\}$ and $\hat{\vect{p}}_h = \{\vect{p}_h\}$, respectively,
they are formally dual with each other as well.  

\section{Summary} \label{sec:concluding}

In this paper, we present a unified study for the design of various
finite element methods through the concept of DG-derivatives. Then we
compare these methods and show their relationships in Table
\ref{tab:FEMs}. We find that the schemes of stabilized hybrid mixed
methods and stabilized hybrid primal methods are mutually dual, and
hybrid primal and hybrid mixed are dual with each other as well.

Furthermore, we see that each finite element method approximates
either a primal or mixed form of the problem. Continuity of $u$ is
needed for the primal form, and on the other hand, $\boldsymbol H({\rm
div})$ continuity of $\vect{p}$ is required for the mixed form. To
design finite element methods, we have to use certain mechanics to
make the numerical approximations $u_h$ or $\vect{p}_h$ satisfy
certain continuity requirements. There are five approaches: 
(1) choosing a finite element space with strongly continuity (conforming FEMs);  
(2) choosing a finite element space with weakly continuity (nonconforming FEMs); 
(3) using the Lagrange multiplier to force the continuity (hybrid methods); and
(4) using the Lagrange multiplier and stabilization to force the
continuity (HDG and WG methods); 
(5) adding a penalty term in the weak form (DG methods).

Through this study, we derive the mixed DG methods. The well-posedness
of this method is proven, and optimal error estimates are obtained. We
also present rigorous proofs of the convergence from WG to mixed
methods as well as the convergence from HDG to primal methods, as the
stabilization parameter goes to infinity.

There are some other important FEMs that are not covered by our
framework, such as finite volume methods \cite{eymard2000finite},
mimetic finite difference methods \cite{brezzi2005family,
brezzi2005convergence}, and virtual element methods
\cite{beirao2013basic, brezzi2014basic}.  Our study is mainly done for
second-order elliptic boundary value problems. Extension of this study
to higher-order problems is currently under investigation and will be
reported in a future work.

\bibliography{0XDGR1}

\begin{thebibliography}{100}

\bibitem{argyris1954energy}
J.~H. Argyris.
\newblock Energy theorems and structural analysis: a generalized discourse with
  applications on energy principles of structural analysis including the
  effects of temperature and non-linear stress-strain relations.
\newblock {\em Aircraft Engineering and Aerospace Technology}, 26(10):347--356,
  1954.

\bibitem{arnold1982interior}
D.~N. Arnold.
\newblock An interior penalty finite element method with discontinuous
  elements.
\newblock {\em SIAM Journal on Numerical Analysis}, 19(4):742--760, 1982.

\bibitem{arnold1985mixed}
D.~N. Arnold and F.~Brezzi.
\newblock Mixed and nonconforming finite element methods: implementation,
  postprocessing and error estimates.
\newblock {\em RAIRO-Mod{\'e}lisation math{\'e}matique et analyse
  num{\'e}rique}, 19(1):7--32, 1985.

\bibitem{arnold2002unified}
D.~N. Arnold, F.~Brezzi, B.~Cockburn, and L.~D. Marini.
\newblock {Unified analysis of discontinuous Galerkin methods for elliptic
  problems}.
\newblock {\em SIAM Journal on Numerical Analysis}, 39(5):1749--1779, 2002.

\bibitem{aubin1970approximation}
J.~P. Aubin.
\newblock Approximation des problemes aux limites non homogenes pour des
  op{\'e}rateurs non lin{\'e}aires.
\newblock {\em Journal of Mathematical Analysis and Applications},
  30(3):510--521, 1970.

\bibitem{babuvska1973nonconforming}
I.~Babu{\v{s}}ka and M.~Zl{\'a}mal.
\newblock Nonconforming elements in the finite element method with penalty.
\newblock {\em SIAM Journal on Numerical Analysis}, 10(5):863--875, 1973.

\bibitem{babuska1977mixed}
I.~Babu\v{s}ka, J.~T. Oden, and J.~K. Lee.
\newblock Mixed-hybrid finite element approximations of second-order elliptic
  boundary-value problems.
\newblock {\em Computer Methods in Applied Mechanics and Engineering},
  11(2):175--206, 1977.

\bibitem{bassi1997high}
F.~Bassi, S.~Rebay, G.~Mariotti, S.~Pedinotti, and M.~Savini.
\newblock A high-order accurate discontinuous finite element method for
  inviscid and viscous turbomachinery flows.
\newblock In {\em Proceedings of the 2nd European Conference on Turbomachinery
  Fluid Dynamics and Thermodynamics}, pages 99--109. Technologisch Instituut,
  Antwerpen, Belgium, 1997.

\bibitem{becker2004reduced}
R.~Becker, E.~Burman, P.~Hansbo, and M.~Larson.
\newblock {A reduced P1-discontinuous Galerkin method}.
\newblock Technical report, Chalmers Finite Element Center Preprint 2003-13,
  2004.

\bibitem{beirao2013basic}
L.~Beir{\~a}o~da Veiga, F.~Brezzi, A.~Cangiani, G.~Manzini, L.~D. Marini, and
  A.~Russo.
\newblock Basic principles of virtual element methods.
\newblock {\em Mathematical Models and Methods in Applied Sciences},
  23(1):199--214, 2013.

\bibitem{boffi2013mixed}
D.~Boffi, F.~Brezzi, and M.~Fortin.
\newblock {\em Mixed finite element methods and applications}, volume~44 of
  {\em Springer Series in Computational Mathematics}.
\newblock Springer, 2013.

\bibitem{brenner2003poincare}
S.~C. Brenner.
\newblock {Poincar{\'e}--Friedrichs inequalities for piecewise $H_1$
  functions}.
\newblock {\em SIAM Journal on Numerical Analysis}, 41(1):306--324, 2003.

\bibitem{brenner2015forty}
S.~C. Brenner.
\newblock {Forty years of the Crouzeix-Raviart element}.
\newblock {\em Numerical Methods for Partial Differential Equations},
  31(2):367--396, 2015.

\bibitem{brenner2007mathematical}
S.~C. Brenner and L.~R. Scott.
\newblock {\em The mathematical theory of finite element methods}, volume~15.
\newblock Springer Science \& Business Media, 2007.

\bibitem{brenner2005c}
S.~C. Brenner and L.-Y. Sung.
\newblock {$C^0$ interior penalty methods for fourth order elliptic boundary
  value problems on polygonal domains}.
\newblock {\em Journal of Scientific Computing}, 22(1-3):83--118, 2005.

\bibitem{brezzi1974existence}
F.~Brezzi.
\newblock {On the existence, uniqueness and approximation of saddle-point
  problems arising from Lagrangian multipliers}.
\newblock {\em Revue fran{\c{c}}aise d'automatique, informatique, recherche
  op{\'e}rationnelle. Analyse num{\'e}rique}, 8(2):129--151, 1974.

\bibitem{brezzi1987mixed}
F.~Brezzi, J.~Douglas~Jr, R.~Dur{\'a}n, and M.~Fortin.
\newblock Mixed finite elements for second order elliptic problems in three
  variables.
\newblock {\em Numerische Mathematik}, 51(2):237--250, 1987.

\bibitem{brezzi1987efficient}
F.~Brezzi, J.~Douglas~Jr, M.~Fortin, and L.~D. Marini.
\newblock Efficient rectangular mixed finite elements in two and three space
  variables.
\newblock {\em RAIRO-Mod{\'e}lisation math{\'e}matique et analyse
  num{\'e}rique}, 21(4):581--604, 1987.

\bibitem{brezzi1985two}
F.~Brezzi, J.~Douglas~Jr, and L.~D. Marini.
\newblock Two families of mixed finite elements for second order elliptic
  problems.
\newblock {\em Numerische Mathematik}, 47(2):217--235, 1985.

\bibitem{brezzi2014basic}
F.~Brezzi, R.~S. Falk, and L.~D. Marini.
\newblock Basic principles of mixed virtual element methods.
\newblock {\em ESAIM: Mathematical Modelling and Numerical Analysis},
  48(4):1227--1240, 2014.

\bibitem{brezzi1991mixed}
F.~Brezzi and M.~Fortin.
\newblock {\em Mixed and hybrid finite element methods}, volume~15 of {\em
  Springer Series in Computational Mathematics}.
\newblock Springer-Verlag, 1991.

\bibitem{brezzi2005mixed}
F.~Brezzi, T.~Hughes, L.~D. Marini, and A.~Masud.
\newblock {Mixed discontinuous Galerkin methods for Darcy flow}.
\newblock {\em Journal of Scientific Computing}, 22(1):119--145, 2005.

\bibitem{brezzi2005convergence}
F.~Brezzi, K.~Lipnikov, and M.~Shashkov.
\newblock Convergence of the mimetic finite difference method for diffusion
  problems on polyhedral meshes.
\newblock {\em SIAM Journal on Numerical Analysis}, 43(5):1872--1896, 2005.

\bibitem{brezzi2005family}
F.~Brezzi, K.~Lipnikov, and V.~Simoncini.
\newblock A family of mimetic finite difference methods on polygonal and
  polyhedral meshes.
\newblock {\em Mathematical Models and Methods in Applied Sciences},
  15(10):1533--1551, 2005.

\bibitem{brezzi1999discontinuous}
F.~Brezzi, G.~Manzini, D.~Marini, P.~Pietra, and A.~Russo.
\newblock Discontinuous finite elements for diffusion problems.
\newblock {\em Atti Convegno in onore di F. Brioschi (Milano 1997), Istituto
  Lombardo, Accademia di Scienze e Lettere}, pages 197--217, 1999.

\bibitem{brezzi2000discontinuous}
F.~Brezzi, G.~Manzini, L.~D. Marini, P.~Pietra, and A.~Russo.
\newblock {Discontinuous Galerkin approximations for elliptic problems}.
\newblock {\em Numerical Methods for Partial Differential Equations},
  16(4):365--378, 2000.

\bibitem{burman2009local}
E.~Burman and B.~Stamm.
\newblock {Local discontinuous Galerkin method for diffusion equations with
  reduced stabilization}.
\newblock {\em Communications in Computational Physics}, 5:498--514, 2009.

\bibitem{carrero2006hybridized}
J.~Carrero, B.~Cockburn, and D.~Sch{\"o}tzau.
\newblock {Hybridized globally divergence-free LDG methods. Part I: The Stokes
  problem}.
\newblock {\em Mathematics of Computation}, 75(254):533--563, 2006.

\bibitem{cea1964approximation}
J.~C{\'e}a.
\newblock Approximation variationnelle des probl{\`e}mes aux limites.
\newblock {\em Ann. Inst. Fourier (Grenoble)}, 14(fasc. 2):345--444, 1964.

\bibitem{chen2016weak}
W.~Chen, F.~Wang, and Y.~Wang.
\newblock {Weak Galerkin method for the coupled Darcy--Stokes flow}.
\newblock {\em IMA Journal of Numerical Analysis}, 36(2):897--921, 2016.

\bibitem{chen2012analysis}
Y.~Chen and B.~Cockburn.
\newblock {Analysis of variable-degree HDG methods for convection-diffusion
  equations. Part I: general nonconforming meshes}.
\newblock {\em IMA Journal of Numerical Analysis}, 32(4):1267--1293, 2012.

\bibitem{chen2014analysis}
Y.~Chen and B.~Cockburn.
\newblock {Analysis of variable-degree HDG methods for convection-diffusion
  equations. Part II: Semimatching nonconforming meshes}.
\newblock {\em Mathematics of Computation}, 83(285):87--111, 2014.

\bibitem{chen2010local}
Y.~Chen, J.~Huang, X.~Huang, and Y.~Xu.
\newblock On the local discontinuous {Galerkin} method for linear elasticity.
\newblock {\em Mathematical Problems in Engineering}, 2010, 2010.

\bibitem{chung2014staggered}
E.~Chung, B.~Cockburn, and G.~Fu.
\newblock {The staggered DG method is the limit of a hybridizable DG method}.
\newblock {\em SIAM Journal on Numerical Analysis}, 52(2):915--932, 2014.

\bibitem{ciarlet2002finite}
P.~G. Ciarlet.
\newblock {\em The finite element method for elliptic problems}, volume~4 of
  {\em Studies in Mathematics and its Applications}.
\newblock North-Holland, 1978.

\bibitem{ciarlet1971multipoint}
P.~G. Ciarlet and C.~Wagschal.
\newblock {Multipoint Taylor formulas and applications to the finite element
  method}.
\newblock {\em Numerische Mathematik}, 17(1):84--100, 1971.

\bibitem{cockburn2016static}
B.~Cockburn.
\newblock Static condensation, hybridization, and the devising of the {HDG}
  methods.
\newblock In {\em Building Bridges: Connections and Challenges in Modern
  Approaches to Numerical Partial Differential Equations}, pages 129--177.
  Springer, 2016.

\bibitem{cockburn2008superconvergent}
B.~Cockburn, B.~Dong, and J.~Guzm{\'a}n.
\newblock {A superconvergent LDG-hybridizable Galerkin method for second-order
  elliptic problems}.
\newblock {\em Mathematics of Computation}, 77(264):1887--1916, 2008.

\bibitem{cockburn2017devising}
B.~Cockburn and G.~Fu.
\newblock {Devising superconvergent HDG methods with symmetric approximate
  stresses for linear elasticity by ${\boldsymbol{M}}$-decompositions}.
\newblock {\em IMA Journal of Numerical Analysis}, 2017.

\bibitem{cockburn2017superconvergence}
B.~Cockburn, G.~Fu, and F.~Sayas.
\newblock {Superconvergence by $M$-decompositions. Part I: General theory for
  HDG methods for diffusion}.
\newblock {\em Mathematics of Computation}, 86(306):1609--1641, 2017.

\bibitem{cockburn2004characterization}
B.~Cockburn and J.~Gopalakrishnan.
\newblock A characterization of hybridized mixed methods for second order
  elliptic problems.
\newblock {\em SIAM Journal on Numerical Analysis}, 42(1):283--301, 2004.

\bibitem{cockburn2005error}
B.~Cockburn and J.~Gopalakrishnan.
\newblock {Error analysis of variable degree mixed methods for elliptic
  problems via hybridization}.
\newblock {\em Mathematics of Computation}, 74(252):1653--1677, 2005.

\bibitem{cockburn2005incompressibleI}
B.~Cockburn and J.~Gopalakrishnan.
\newblock {Incompressible finite elements via hybridization. Part I: The Stokes
  system in two space dimensions}.
\newblock {\em SIAM Journal on Numerical Analysis}, 43(4):1627--1650, 2005.

\bibitem{cockburn2005incompressibleII}
B.~Cockburn and J.~Gopalakrishnan.
\newblock {Incompressible finite elements via hybridization. Part II: The
  Stokes system in three space dimensions}.
\newblock {\em SIAM Journal on Numerical Analysis}, 43(4):1651--1672, 2005.

\bibitem{cockburn2005new}
B.~Cockburn and J.~Gopalakrishnan.
\newblock New hybridization techniques.
\newblock {\em GAMM-Mitteilungen}, 28(2):154--182, 2005.

\bibitem{cockburn2009unified}
B.~Cockburn, J.~Gopalakrishnan, and R.~Lazarov.
\newblock {Unified hybridization of discontinuous Galerkin, mixed, and
  continuous Galerkin methods for second order elliptic problems}.
\newblock {\em SIAM Journal on Numerical Analysis}, 47(2):1319--1365, 2009.

\bibitem{cockburn2010projection}
B.~Cockburn, J.~Gopalakrishnan, and F.~Sayas.
\newblock {A projection-based error analysis of HDG methods}.
\newblock {\em Mathematics of Computation}, 79(271):1351--1367, 2010.

\bibitem{cockburn2000development}
B.~Cockburn, G.~E. Karniadakis, and C.-W. Shu.
\newblock The development of discontinuous {G}alerkin methods.
\newblock In {\em Discontinuous Galerkin Methods}, pages 3--50. Springer, 2000.

\bibitem{cockburn2015hybridizable}
B.~Cockburn and K.~Mustapha.
\newblock {A hybridizable discontinuous Galerkin method for fractional
  diffusion problems}.
\newblock {\em Numerische Mathematik}, 130(2):293--314, 2015.

\bibitem{cockburn2016contraction}
B.~Cockburn, R.~H. Nochetto, and W.~Zhang.
\newblock {Contraction property of adaptive hybridizable discontinuous Galerkin
  methods}.
\newblock {\em Mathematics of Computation}, 85(299):1113--1141, 2016.

\bibitem{cockburn2012conditions}
B.~Cockburn, W.~Qiu, and K.~Shi.
\newblock {Conditions for superconvergence of HDG methods for second-order
  elliptic problems}.
\newblock {\em Mathematics of Computation}, 81(279):1327--1353, 2012.

\bibitem{cockburn2012superconvergent}
B.~Cockburn, W.~Qiu, and K.~Shi.
\newblock {Superconvergent HDG methods on isoparametric elements for
  second-order elliptic problems}.
\newblock {\em SIAM Journal on Numerical Analysis}, 50(3):1417--1432, 2012.

\bibitem{cockburn2016hybridizable}
B.~Cockburn and J.~Shen.
\newblock {A hybridizable discontinuous Galerkin method for the $p$-Laplacian}.
\newblock {\em SIAM Journal on Scientific Computing}, 38(1):A545--A566, 2016.

\bibitem{cockburn1998local}
B.~Cockburn and C.-W. Shu.
\newblock {The local discontinuous Galerkin method for time-dependent
  convection-diffusion systems}.
\newblock {\em SIAM Journal on Numerical Analysis}, 35(6):2440--2463, 1998.

\bibitem{cockburn2012posteriori}
B.~Cockburn and W.~Zhang.
\newblock {A posteriori error estimates for HDG methods}.
\newblock {\em Journal of Scientific Computing}, 51(3):582--607, 2012.

\bibitem{cockburn2013posteriori}
B.~Cockburn and W.~Zhang.
\newblock A posteriori error analysis for hybridizable discontinuous {Galerkin}
  methods for second order elliptic problems.
\newblock {\em SIAM Journal on Numerical Analysis}, 51(1):676--693, 2013.

\bibitem{courant1943variational}
R.~Courant.
\newblock Variational methods for the solution of problems of equilibrium and
  vibrations.
\newblock {\em Bulletin of the American Mathematical Society}, 49(1):1--23,
  1943.

\bibitem{crouzeix1973conforming}
M.~Crouzeix and P.~A. Raviart.
\newblock {Conforming and nonconforming finite element methods for solving the
  stationary Stokes equations I}.
\newblock {\em Revue fran{\c{c}}aise d'automatique, informatique, recherche
  op{\'e}rationnelle. Math{\'e}matique}, 7(3):33--75, 1973.

\bibitem{douglas1976interior}
J.~Douglas, Jr and T.~Dupont.
\newblock {Interior penalty procedures for elliptic and parabolic Galerkin
  methods}.
\newblock In {\em Computing Methods in Applied Sciences}, pages 207--216.
  Springer, 1976.

\bibitem{ekeland1976convex}
I.~Ekeland and R.~Temam.
\newblock {\em Convex analysis and variational problems}.
\newblock SIAM, 1976.

\bibitem{eymard2000finite}
R.~Eymard, T.~Gallou{\"e}t, and R.~Herbin.
\newblock Finite volume methods.
\newblock {\em Handbook of numerical analysis}, 7:713--1018, 2000.

\bibitem{Feng1965finite}
K.~Feng.
\newblock {Finite difference method based on variation principle}.
\newblock {\em Communication on Applied Mathematics and Computation},
  2(4):237--261, 1965.

\bibitem{fortin1985three}
M.~Fortin.
\newblock A three-dimensional quadratic nonconforming element.
\newblock {\em Numerische Mathematik}, 46(2):269--279, 1985.

\bibitem{fortin1983non}
M.~Fortin and M.~Soulie.
\newblock A non-conforming piecewise quadratic finite element on triangles.
\newblock {\em International Journal for Numerical Methods in Engineering},
  19(4):505--520, 1983.

\bibitem{fraeijs1965displacement}
B.~Fraeijs~de Veubeke.
\newblock Displacement and equilibrium models in the finite element method.
\newblock {\em Stress analysis}, pages 145--197, 1965.

\bibitem{fu2015analysis}
G.~Fu, B.~Cockburn, and H.~Stolarski.
\newblock {Analysis of an HDG method for linear elasticity}.
\newblock {\em International Journal for Numerical Methods in Engineering},
  102(3-4):551--575, 2015.

\bibitem{gong2017new}
S.~Gong, S.~Wu, and J.~Xu.
\newblock New hybridized mixed methods for linear elasticity and optimal
  multilevel solvers.
\newblock {\em arXiv preprint arXiv:1704.07540}, 2017.

\bibitem{henshell1973hybrid}
R.~D. Henshell.
\newblock On hybrid finite elements.
\newblock {\em The Mathematics of Finite Elements and Applications}, pages
  299--312, 1973.

\bibitem{hong2012discontinuous}
Q.~Hong, J.~Hu, S.~Shu, and J.~Xu.
\newblock A discontinuous {G}alerkin method for the fourth-order curl problem.
\newblock {\em Journal of Computational Mathematics}, 30(6):565--578, 2012.

\bibitem{honguniformly}
Q.~Hong and J.~Kraus.
\newblock {Uniformly stable discontinuous Galerkin discretization and robust
  iterative solution methods for the Brinkman problem}.
\newblock {\em SIAM Journal on Numerical Analysis}, 54(5):2750--2774, 2016.

\bibitem{hong2017parameter}
Q.~Hong and J.~Kraus.
\newblock {Parameter-robust stability of classical three-field formulation of
  Biot's consolidation model}.
\newblock {\em arXiv preprint arXiv:1706.00724}, 2017.

\bibitem{hong2016robust}
Q.~Hong, J.~Kraus, J.~Xu, and L.~Zikatanov.
\newblock {A robust multigrid method for discontinuous Galerkin discretizations
  of Stokes and linear elasticity equations}.
\newblock {\em Numerische Mathematik}, 132(1):23--49, 2016.

\bibitem{hong2018extended}
Q.~Hong, S.~Wu, and J.~Xu.
\newblock {Extended Galerkin methods for second order elliptic problems}.
\newblock {\em Preprint}, 2018.

\bibitem{hong2017uniformly}
Q.~Hong and J.~Xu.
\newblock {Uniformly stable results and error estimates for HDG and WG
  methods}.
\newblock {\em Preprint}, 2017.

\bibitem{hrennikoff1941solution}
A.~Hrennikoff.
\newblock Solution of problems of elasticity by the framework method.
\newblock {\em Journal of Applied Mechanics}, 8(4):169--175, 1941.

\bibitem{hu2016canonical}
J.~Hu and S.~Zhang.
\newblock {A canonical construction of $H^m$-nonconforming triangular finite
  elements}.
\newblock {\em preprint}, 2017.

\bibitem{jones1964generalization}
R.~E. Jones.
\newblock A generalization of the direct-stiffness method of structural
  analysis.
\newblock {\em AIAA Journal}, 2(5):821--826, 1964.

\bibitem{kabaria2015hybridizable}
H.~Kabaria, A.~J. Lew, and B.~Cockburn.
\newblock A hybridizable discontinuous galerkin formulation for non-linear
  elasticity.
\newblock {\em Computer Methods in Applied Mechanics and Engineering},
  283:303--329, 2015.

\bibitem{lions1968penalty}
J.~L. Lions.
\newblock Probl\`{e}ms aux limites non homog\`{e}nes \`{a} don\'{e}es
  irr\'{e}guli\`{e}res: Une m\'{e}thode d'approximation,.
\newblock {\em in Numerical Analysis of Partial Differential Equations
  (C.I.ME.2 Ciclo,Ispra.1967). Edizioni Cremonese, Rome}, pages 283--292, 1968.

\bibitem{liu2009direct}
H.~Liu and J.~Yan.
\newblock {The direct discontinuous Galerkin (DDG) methods for diffusion
  problems}.
\newblock {\em SIAM Journal on Numerical Analysis}, 47(1):675--698, 2009.

\bibitem{liu2010direct}
H.~Liu and J.~Yan.
\newblock {The direct discontinuous Galerkin (DDG) method for diffusion with
  interface corrections}.
\newblock {\em Communications in Computational Physics}, 8(3):541, 2010.

\bibitem{mikhlin1964variational}
S.~G. Mikhlin.
\newblock {\em Variational methods in mathematical physics}, volume~50.
\newblock Pergamon Press; [distributed by Macmillan, New York], 1964 (original
  Russian edition: 1957).

\bibitem{weakgalerkinLinMu}
L.~Mu, J.~Wang, and X.~Ye.
\newblock {Weak Galerkin finite element methods on polytopal meshes}.
\newblock {\em International Journal of Numerical Analysis and Modeling},
  12:31--53, 2015.

\bibitem{mu2015weak}
L.~Mu, J.~Wang, X.~Ye, and S.~Zhang.
\newblock {A weak Galerkin finite element method for the Maxwell equations}.
\newblock {\em Journal of Scientific Computing}, 65(1):363--386, 2015.

\bibitem{nedelec1986new}
J.~N{\'e}d{\'e}lec.
\newblock {A new family of mixed finite elements in $R^3$}.
\newblock {\em Numerische Mathematik}, 50(1):57--81, 1986.

\bibitem{nedelec1980mixed}
J.~C. N{\'e}d{\'e}lec.
\newblock Mixed finite elements in $\mathbb{R}^3$.
\newblock {\em Numerische Mathematik}, 35(3):315--341, 1980.

\bibitem{nicolaides1972class}
R.~A. Nicolaides.
\newblock On a class of finite elements generated by lagrange interpolation.
\newblock {\em SIAM Journal on Numerical Analysis}, 9(3):435--445, 1972.

\bibitem{oikawa2015hybridized}
I.~Oikawa.
\newblock {A hybridized discontinuous Galerkin method with reduced
  stabilization}.
\newblock {\em Journal of Scientific Computing}, 65(1):327--340, 2015.

\bibitem{park2003p}
C.~Park and D.~Sheen.
\newblock P1-nonconforming quadrilateral finite element methods for
  second-order elliptic problems.
\newblock {\em SIAM Journal on Numerical Analysis}, 41(2):624--640, 2003.

\bibitem{pian1964derivation}
T.~H.~H. Pian.
\newblock Derivation of element stiffness matrices by assumed stress
  distributions.
\newblock {\em AIAA Journal}, 2(7):1333--1336, 1964.

\bibitem{pian1972finite}
T.~H.~H. Pian.
\newblock Finite element formulation by variational principles with relaxed
  continuity requirements.
\newblock {\em The mathematical foundations of the finite element method with
  applications to partial differential equations}, pages 671--687, 1972.

\bibitem{pian1969basis}
T.~H.~H. Pian and P.~Tong.
\newblock Basis of finite element methods for solid continua.
\newblock {\em International Journal for Numerical Methods in Engineering},
  1(1):3--28, 1969.

\bibitem{pin1969variational}
T.~Pin and T.~H.~H. Pian.
\newblock A variational principle and the convergence of a finite-element
  method based on assumed stress distribution.
\newblock {\em International Journal of Solids and Structures}, 5(5):463--472,
  1969.

\bibitem{raviart1975hybrid}
P.~A. Raviart.
\newblock Hybrid finite element methods for solving 2nd order elliptic
  equations.
\newblock {\em Topics in numerical analysis, II (JJM Miller, ed.)}, pages
  141--155, 1975.

\bibitem{raviart1977mixed}
P.~A. Raviart and J.~M. Thomas.
\newblock A mixed finite element method for 2nd order elliptic problems.
\newblock In {\em Mathematical Aspects of Finite Element Methods}, pages
  292--315. Springer, 1977.

\bibitem{raviart1977primal}
P.~A. Raviart and J.~M. Thomas.
\newblock Primal hybrid finite element methods for 2nd order elliptic
  equations.
\newblock {\em Mathematics of Computation}, 31(138):391--413, 1977.

\bibitem{roberts1991mixed}
J.~E. Roberts and J.~M. Thomas.
\newblock Mixed and hybrid methods.
\newblock {\em Handbook of numerical analysis}, 2:523--639, 1991.

\bibitem{santos1999nonconforming}
J.~E. Santos, D.~Sheen, and X.~Ye.
\newblock {Nonconforming Galerkin methods based on quadrilateral elements for
  second order elliptic problems}.
\newblock {\em ESAIM: Mathematical Modelling and Numerical Analysis},
  33(4):747--770, 1999.

\bibitem{scott1990finite}
L.~R. Scott and S.~Zhang.
\newblock Finite element interpolation of nonsmooth functions satisfying
  boundary conditions.
\newblock {\em Mathematics of Computation}, 54(190):483--493, 1990.

\bibitem{soon2009hybridizable}
S.-C. Soon, B.~Cockburn, and H.~K. Stolarski.
\newblock {A hybridizable discontinuous Galerkin method for linear elasticity}.
\newblock {\em International Journal for Numerical Methods in Engineering},
  80(8):1058--1092, 2009.

\bibitem{strang1972variational}
G.~Strang.
\newblock Variational crimes in the finite element method.
\newblock {\em The Mathematical Foundations of the Finite Element Method with
  Applications to Partial Differential Equations}, pages 689--710, 1972.

\bibitem{sun2009locally}
S.~Sun and J.~Liu.
\newblock A locally conservative finite element method based on piecewise
  constant enrichment of the continuous {Galerkin} method.
\newblock {\em SIAM Journal on Scientific Computing}, 31(4):2528--2548, 2009.

\bibitem{turner1956stiffness}
M.~J. Turner, R.~W. Clough, H.~C. Martin, and L.~J. Topp.
\newblock Stiffness and deflection analysis of complex structures.
\newblock {\em Journal of the Aeronautical Sciences}, 23:805--823, 1956.

\bibitem{veeser2017quasi}
A.~Veeser and P.~Zanotti.
\newblock {Quasi-optimal nonconforming methods for symmetric elliptic problems.
  I--Abstract theory}.
\newblock {\em arXiv preprint arXiv:1710.03331}, 2017.

\bibitem{veeser2017quasiDG}
A.~Veeser and P.~Zanotti.
\newblock {Quasi-optimal nonconforming methods for symmetric elliptic problems.
  III--DG and other interior penalty methods}.
\newblock {\em arXiv preprint arXiv:1710.03452}, 2017.

\bibitem{wang2017primalFP}
C.~Wang and J.~Wang.
\newblock {A primal-dual weak Galerkin finite element method for Fokker-Planck
  type equations}.
\newblock {\em arXiv preprint arXiv:1704.05606}, 2017.

\bibitem{wang2017primal}
C.~Wang and J.~Wang.
\newblock {A primal-dual weak Galerkin finite element method for second order
  elliptic equations in non-divergence form}.
\newblock {\em Mathematics of Computation}, 2017.

\bibitem{wang2015weak}
J.~Wang and C.~Wang.
\newblock {Weak Galerkin finite element methods for elliptic PDEs}.
\newblock {\em Scientia Sinica Mathematica}, 45(7):1061--1092, 2015.

\bibitem{wang2013weak}
J.~Wang and X.~Ye.
\newblock {A weak Galerkin finite element method for second-order elliptic
  problems}.
\newblock {\em Journal of Computational and Applied Mathematics}, 241:103--115,
  2013.

\bibitem{wang2014weak}
J.~Wang and X.~Ye.
\newblock {A weak Galerkin mixed finite element method for second order
  elliptic problems}.
\newblock {\em Mathematics of Computation}, 83(289):2101--2126, 2014.

\bibitem{wang2016weak}
J.~Wang and X.~Ye.
\newblock {A weak Galerkin finite element method for the Stokes equations}.
\newblock {\em Advances in Computational Mathematics}, 42(1):155--174, 2016.

\bibitem{wang2006morley}
M.~Wang and J.~Xu.
\newblock {The Morley element for fourth order elliptic equations in any
  dimensions}.
\newblock {\em Numerische Mathematik}, 103(1):155--169, 2006.

\bibitem{wang2013minimal}
M.~Wang and J.~Xu.
\newblock {Minimal finite element spaces for $2m$-th-order partial differential
  equations in $\mathbb{R}^n$}.
\newblock {\em Mathematics of Computation}, 82(281):25--43, 2013.

\bibitem{wheeler1978elliptic}
M.~F. Wheeler.
\newblock An elliptic collocation-finite element method with interior
  penalties.
\newblock {\em SIAM Journal on Numerical Analysis}, 15(1):152--161, 1978.

\bibitem{wilson1971incompatible}
E.~L. Wilson and R.~L. Taylor.
\newblock Incompatible displacement models.
\newblock In {\em Proceedings of the Symposium on Numerical and Computer
  Methods in Structural Engineering}. University of Illinois, 1971.

\bibitem{wolf1975alternate}
J.~P. Wolf.
\newblock Alternate hybrid stress finite element models.
\newblock {\em International Journal for Numerical Methods in Engineering},
  9(3):601--615, 1975.

\bibitem{gong2015mixed}
S.~Wu, S.~Gong, and J.~Xu.
\newblock Interior penalty mixed finite element methods of any order in any
  dimension for linear elasticity with strongly symmetric stress tensor.
\newblock {\em Mathematical Models and Methods in Applied Sciences},
  27(14):2711--2743, 2017.

\bibitem{wu2017pm}
S.~Wu and J.~Xu.
\newblock {$\mathcal{P}_m$ interior penalty nonconforming finite element
  methods for $2m$-th order PDEs in $\mathbb{R}^n$}.
\newblock {\em arXiv preprint arXiv:1710.07678}, 2017.

\bibitem{wu2017nonconforming}
S.~Wu and J.~Xu.
\newblock Nonconforming finite element spaces for $2m$-th order partial
  differential equations on $\mathbb{R}^n$ simplicial grids when $m= n+1$.
\newblock {\em arXiv:1705.10873}, 2017.

\bibitem{Yamamoto1966}
Y.~Yamamoto.
\newblock {\em A formulation of matrix displacement method}.
\newblock [Cambridge, Mass.]: Massachusetts Institute of Technology, Dept. of
  Aeronautics and Astronautics, 1966.

\bibitem{zlamal1968finite}
M.~Zl{\'a}mal.
\newblock On the finite element method.
\newblock {\em Numerische Mathematik}, 12(5):394--409, 1968.

\end{thebibliography}
\bibliographystyle{abbrv}

\end{document}